\documentclass[12pt]{book}

\usepackage{amsmath, amsthm, amsfonts, amssymb}
\usepackage[all]{xy}
\usepackage{fullpage}

\theoremstyle{plain}
\newtheorem{thm}{Theorem}[chapter]
\newtheorem{lem}[thm]{Lemma}
\newtheorem{prop}[thm]{Proposition}
\newtheorem{cor}[thm]{Corollary}
\newtheorem*{question}{Question}

\theoremstyle{definition}
\newtheorem*{defn}{Definition}

\newtheoremstyle{citing}
  {3pt}{3pt}{\itshape}{}{\bfseries}{.}{.5em}{\thmnote{#3}}

\theoremstyle{citing}
\newtheorem*{thmx}{}

\newcommand{\sbinom}[2]{\genfrac{[}{]}{0pt}{}{#1}{#2}}

\newcommand{\Aut}{\mathrm{Aut}}
\newcommand{\Hom}{\mathrm{Hom}}
\newcommand{\GL}{\mathrm{GL}}

\newcommand{\qemb}{\mathrm{qemb}}
\newcommand{\pow}{\mathrm{pow}}
\newcommand{\Fcom}{\mathrm{Fcom}}
\newcommand{\com}{\mathrm{com}}
\newcommand{\proj}{\mathrm{proj}}
\newcommand{\F}{\mathbb{F}}

\newcommand{\Z}{\mathbb{Z}}
\newcommand{\C}{\mathbb{C}}
\newcommand{\ep}{\varepsilon}
\newcommand{\Inn}{\mathrm{Inn}}
\newcommand{\Out}{\mathrm{Out}}
\newcommand{\Sp}{\mathrm{Sp}}

\allowdisplaybreaks[1]

\begin{document}

\title{Automorphism Groups of Finite \lowercase{$p$}-Groups: \\ Structure and Applications \\
\vspace{.25in}
{\small A Dissertation Submitted to the \\ Department of Mathematics at Stanford University \\ in Partial Fulfillment of the Requirements for \\ the Degree of Doctor of Philosophy \\}
\vspace{.5in}
{\large Geir T. Helleloid} \\
\vspace{.25in}
{\small Department of Mathematics \\
The University of Texas at Austin \\
1 University Station C1200 \\
Austin, TX 78712 \\
Current Work Email Address: \texttt{geir@ma.utexas.edu} \\
Permanent Email Address: \texttt{geir.helleloid@gmail.com} \\}}
\date{August 2007}

\maketitle

\section*{Abstract}
This thesis has three goals related to the automorphism groups of finite $p$-groups.  The primary goal is to provide a complete proof of a theorem showing that, in some asymptotic sense, the automorphism group of almost every finite $p$-group is itself a $p$-group.  We originally proved this theorem in a paper with Martin; the presentation of the proof here contains omitted proof details and revised exposition.  We also give a survey of the extant results on automorphism groups of finite $p$-groups, focusing on the order of the automorphism groups and on known examples.  Finally, we explore a connection between automorphisms of finite $p$-groups and Markov chains.  Specifically, we define a family of Markov chains on an elementary abelian $p$-group and bound the convergence rate of some of those chains.

\section*{Acknowledgments}

First, I would like to thank my advisor, Persi Diaconis.  Over the past five years, he has shown me the beautiful connections between random walks, combinatorics, finite group theory, and plenty of other mathematics.  His continued encouragement and advice made the completion of this thesis possible.  Thanks also go to Dan Bump, Nat Thiem, and Ravi Vakil for being on my defense committee and helping me finish my last hurdle as a Ph.D. student.

My mentor Joe Gallian has influenced my life in more ways than I ever could have imagined when I first went to his REU in the summer of 2001.  Spending one summer in Duluth as a student, two summers as a research advisor, and four summers as a research visitor has done more for my development as a research mathematician, teacher, and mentor than anything else in my life.  At the same time, I first met most of my best friends at Duluth, and I will always thank Joe for bringing Phil Matchett Wood, Melanie Wood, Dan Isaksen, David Arthur, Stephen Hartke, and so many others into my life.

I would also like to thank my friends at Stanford for their steadfast friendship over the years, particularly Leo Rosales, Dan Ramras, and Dana Paquin.

I've saved the best for last.  My parents and my girlfriend Jenny are the three most important people in my life, and I cannot thank them enough for their unconditional love and support.  Jenny was my cheerleader in the last stressful year of my graduate studies, and my parents have been my cheerleaders for the last 26 years.  Thank you.

\tableofcontents

\hspace{1in}

\vfill

\pagebreak

\section*{Notation and Terminology}

Let $G$ be a group.
\begin{itemize}
\item If $x,y \in G$, then $[x,y] = x^{-1}y^{-1}xy$.
\item $G' = [G,G] = \left< [x,y] \, : \, x,y \in G \right>$
\item $Z(G) = $ the center of $G$
\item $\Phi(G) = $ the Frattini subgroup of $G$
\item $\Inn(G) = $ the group of inner automorphisms of $G$
\item $\Out(G) = \Aut(G)/\Inn(G) = $ the group of outer automorphisms of $G$
\item $C_n = $ the cyclic group on $n$ elements
\item $d(G) = $ the minimum cardinality of a generating set of $G$.  If $G$ is a free group, then $d(G)$ is called the \emph{rank} of $G$.  If $G$ is a finite elementary abelian $p$-group, then $d(G)$ is the dimension of $G$ as an $\F_p$-vector space and thus is called the \emph{dimension} of $G$.
\item $\GL(d, \F_q) = $ the general linear group of dimension $d$ over the finite field $\F_q$
\item $\sbinom{n}{k}_q = $ the Gaussian (or $q$-binomial) coefficent.  This equals the number of $k$-dimensional subspaces of an $\F_q$-vector space of dimension $n$.
\item $\mathcal{G}_n(q) = $ the Galois number.  This equals the total number of subspaces of an $\F_q$-vector space of dimension $n$.
\end{itemize}

\chapter{Introduction}
\label{c_intro}

For any fixed prime $p$, a non-trivial group $G$ is a \emph{$p$-group} if the order of every element of $G$ is a power of $p$.  When $G$ is finite, this is equivalent to saying that the order of $G$ is a power of $p$.  The study of $p$-groups (and particularly finite $p$-groups) is an important subfield of group theory.  One motivation for studying finite $p$-groups is Sylow's Theorem 1, which states that if $G$ is a finite group, $p$ divides the order of $G$, and $p^n$ is the largest power of $p$ dividing the order of $G$, then $G$ has at least one subgroup of order $p^n$.  Such subgroups are called the \emph{Sylow $p$-subgroups} of $G$.  The fact that $G$ has at least one Sylow $p$-subgroup for each prime $p$ dividing the order of $G$ suggests that in some heuristic sense, finite $p$-groups are the ``building blocks'' of finite groups, and that to understand finite groups, we must first understand finite $p$-groups.  This thesis studies the automorphism groups of finite $p$-groups, but this introductory chapter begins with a description of two other aspects of finite $p$-group theory, both to give a sense for what $p$-group theorists study, and because they are relevant to the main question addressed in this thesis.

It turns out that understanding finite $p$-groups (whatever that means) is quite hard.  Mann~\cite{man} has a wonderful survey of research and open questions in $p$-group theory.  Much of the research relies on a basic fact about finite $p$-groups: they are \emph{nilpotent groups}.  A group $G$ is nilpotent if the series of subgroups $H_0 = G$, $H_1 = [G,G]$, $H_2 = [H_1, G]$, $\dots$, eventually reaches the trivial subgroup.  Here, $[H_i, G]$ denotes the subgroup of $G$ generated by all commutators consisting of an element in $H_i$ and an element in $G$.  If $m$ is the smallest positive integer such that $H_m$ is trivial, then we say that $G$ is \emph{nilpotent of class $m$}, or just \emph{of class $m$}.  When $G$ is a finite $p$-group and the order of $G$ is $p^n$, the class of $G$ is at least 1 and at most $n-1$.  Finite $p$-groups of class 1 are the abelian $p$-groups, and those of class $n-1$ are said to be of \emph{maximal class}.

One triumph in finite $p$-group theory over the past 30 years has been the positive resolution of the five \emph{coclass conjectures} via the joint efforts of several researchers.  While we have no hope of a complete classification of finite $p$-groups up to isomorphism (see Leedham-Green and McKay~\cite[Preface]{lm}), the coclass conjectures do offer a lot of information about finite $p$-groups.  Mann~\cite[Section 3]{man} has a short discussion of the subject, and the book by Leedham-Green and McKay~\cite{lm} is devoted to a proof of the conjectures and related research.  We will state only one of the conjectures here.  The \emph{coclass} of a finite $p$-group of order $p^n$ and class $m$ is defined to be $n-m$.  The coclass Conjecture A states that for some function $f(p,r)$, every finite $p$-group of coclass $r$ has a normal subgroup $K$ of class at most 2 and index at most $f(p,r)$.  If $p = 2$, one can require $K$ to be abelian.

Another aspect of $p$-group theory is the enumeration of finite $p$-groups by their order and related questions, as described in Mann~\cite[Section 1]{man}.  Let $g(k)$ equal the number of groups of order at most $k$, let $g_{\mathrm{nil}}(k)$ equal the number of nilpotent groups of order at most $k$, let $g_p(k)$ equal the number of $p$-groups of order at most $k$, and let $g_{p,2}(k)$ equal the number of $p$-groups of order at most $k$ and class 2.  It is known that
\[
\lim_{k \to \infty}{\frac{g_2(k)}{g_{\mathrm{nil}}(k)}} = 1.
\]
It is an open question as to whether or not
\[
\lim_{k \to \infty}{\frac{g_{\mathrm{nil}}(k)}{g(k)}} = 1;
\]
if so, it would imply that most finite groups are 2-groups.  Pyber~\cite{pyb} has shown the weaker result that
\[
\lim_{k \to \infty}{\frac{\log{g_{\mathrm{nil}}(k)}}{\log{g(k)}}} = 1.
\]
Higman~\cite{hig} and Sims~\cite{sim} show that $g_p(p^n)$ and $g_{p,2}(p^n)$ are both given by the formula $p^{(2/27)k^3+o(k^3)}$ (using little-oh notation).  It is an open problem to evaluate
\[
\lim_{n \to \infty}{\frac{g_{p,2}(p^n)}{g_p(p^n)}}.
\]
It is possible that the limit is 1 and that most $p$-groups have class 2.

This thesis explores the structure of the automorphism groups of finite $p$-groups and the connections between these automorphism groups and other topics.  There are three principal goals.  The first goal is to prove Theorem~\ref{mainthm}, which says that, in a certain asymptotic sense, the automorphism group of a finite $p$-group is almost always a $p$-group.  A weaker version of this result was announced by Martin in~\cite{mar}, and Helleloid and Martin~\cite{hm} prove the general result.  The presentation of the proof in this thesis contains some omitted proof details and revised exposition.

There are many reasonable asymptotic senses in which one could ask if the automorphism group of a finite $p$-group is almost always a $p$-group.  The most obvious is to sort $p$-groups by their order as in the above questions about the number of $p$-groups.  Nilpotence class is another important invariant that one might consider, while the coclass conjectures suggest that an approach using coclass might be more successful.  As it happens, Theorem~\ref{mainthm} does not use any of these parameters, instead turning to the minimum cardinality of a generating set of the $p$-group and an invariant known as the \emph{lower $p$-length}.  While it would be of great interest to prove analogous theorems using the invariants suggested above, it seems that the proofs would require very different machinery.

Chapter~\ref{c_main} contains a statement of Theorem~\ref{mainthm} and an outline of the proof, while Chapters~\ref{c_series} through~\ref{c_submodules} and Appendix~\ref{a_estimates} complete the full proof.  The proof relies on a variety of topics: analyzing the lower $p$-series of a free group via its connection with the free Lie algebra; counting normal subgroups of a finite $p$-group; counting submodules of a module via Hall polynomials; and using numerical estimates on Gaussian coefficients.  Some of the intermediate results may be of independent interest, including Theorems~\ref{expthm},~\ref{normalthm}, and~\ref{submodulesthm}.

The second goal of this thesis is to survey much of what is known about the automorphism groups of finite $p$-groups.  The latter part of this introductory chapter discusses what is known in general about these automorphism groups.  Chapter~\ref{c_examples} focuses on three other topics: explicit computations on the automorphism groups of finite $p$-groups; constructions of finite $p$-groups whose automorphism groups satisfy certain conditions; and examples of finite $p$-groups for which it is known whether or not the automorphism group is itself a $p$-group.

There are aspects of the research on the automorphism groups of finite $p$-groups that are largely omitted in this survey.  We mention three here.  The first is the conjecture that $|G| \le |\Aut(G)|$ for all non-cyclic finite $p$-groups $G$ of order at least $p^3$.  This has been verified for many families of $p$-groups, and no counter-examples are known; there is an old survey by Davitt~\cite{dav}.  The second is the (large) body of work on finer structural questions, like how the automorphism group of an abelian $p$-group splits or examples of finite $p$-groups whose automorphism group fixes all normal subgroups.  The third is the computational aspect of determining the automorphism group of a finite $p$-group.  Eick, Leedham-Green, and O'Brien~\cite{elo} describe an algorithm for constructing the automorphism group of a finite $p$-group.  This algorithm has been implemented by Eick and O'Brien in the GAP package AutPGroup~\cite{gap4}.  There are references in~\cite{elo} to other related research as well.

There are a few survey papers that also summarize some results on automorphism groups.  Corsi Tani~\cite{cor} gives examples of finite $p$-groups whose automorphism group is a $p$-group; all these examples are included in Chapter~\ref{c_examples} along with some others.  Starostin~\cite{sta} and Mann~\cite{man} survey open questions about finite $p$-groups, and each include a section on automorphism groups.  In particular, Starostin focuses on specific examples related to the $|G| \le |\Aut(G)|$ conjecture and finer structural questions.

The third goal of this thesis is to explore a connection between the automorphisms of a finite $p$-group and random walks.  Chapter~\ref{c_apps} focuses on computing the convergence rate of a certain Markov chain on a finite abelian $p$-group that has been ``twisted'' by an automorphism.  Appendix~\ref{a_rwestimates} contains numerical estimates used in this computation.  The introduction to Chapter~\ref{c_apps} briefly mentions the appearance of automorphisms of $p$-groups in two other contexts, namely  projections of random walks and supercharacter theory.

We conclude this introduction with some general results about the automorphism group of a finite $p$-group, for the most part following the survey of Mann~\cite{man}.  First, we can identify three subgroups of $\Aut(G)$ which are themselves $p$-groups.  The inner automorphism group $\Inn(G)$ is trivially a $p$-group.  More interestingly, let $\Aut_c(G)$ be the automorphisms of $G$ which induce the identity automorphism on $G/Z(G)$ (where $Z(G)$ is the center of $G$), and let $\Aut_f(G)$ be the automorphisms of $G$ which induce the identity automorphism on $G/\Phi(G)$ (where $\Phi(G)$ is the Frattini subgroup of $G$, defined as the intersection of all maximal subgroups of $G$).  Then $\Aut_c(G)$ and $\Aut_f(G)$ are $p$-groups, both of which contain $\Inn(G)$.  More results on $\Aut_c(G)$ are given by Curran and McCaughan~\cite{cm}.

The next result is a theorem of Gasch\"{u}tz~\cite{gas}, which states that all finite $p$-groups $G$ have outer automorphisms.  Furthermore, unless $G$ is cyclic of order $p$, there is an outer automorphism whose order is a power of $p$.  It is an open question of Berkovich as to whether this outer automorphism can be chosen to have order $p$.  Schmid~\cite{sch} extends Gasch\"{u}tz' theorem to show that if $G$ is a finite nonabelian $p$-group, then there is an outer automorphism that acts trivially on $Z(G)$.  Furthermore, if $G$ is neither elementary abelian nor extraspecial, then $\Out(G)$ has a non-trivial normal $p$-subgroup.  Webb~\cite{web3} proves Gash\"{u}tz's theorem and Schmid's first generalization in a simpler way and without group cohomology.  If $G$ is not elementary abelian nor extraspecial, then M\"{u}ller~\cite{mul} shows that $\Aut_f(G) > \Inn(G)$.

As mentioned earlier, one prominent open question is whether or not  $|G| \le |\Aut(G)|$ for all non-cyclic $p$-groups $G$ of order at least $p^3$.  A related question concerns the \emph{automorphism tower} of $G$, namely
\[
G = G_0 \to G_1 = \Aut(G_0) \to G_2 = \Aut(G_1) \to \cdots,
\]
where the maps are the natural maps from $G_i$ to $\Inn(G_i)$.  For general groups $G$, a theorem of Wielandt shows that if $G$ is centerless, then the automorphism tower of $G$ becomes stationary in a finite number of steps.  Little is known about the automorphism tower of finite $p$-groups.  In particular, it is not known whether or not there exist finite non-trivial $p$-groups $G$ other than $D_8$ with $\Aut(G) \cong G$.

\chapter{The Automorphism Group is Almost Always a $p$-Group}
\label{c_main}
\chaptermark{Almost Always a \lowercase{$p$}-Group}

Over the next five chapters, we will prove that, in some asymptotic sense, the automorphism group of a finite $p$-group is almost always a $p$-group.  This chapter begins with some examples of automorphism groups of finite $p$-groups and related computational data.  All of the examples are discussed in greater detail in Chapter~\ref{c_examples}, which is a survey of results on the automorphism groups of finite $p$-groups.  We continue with a precise statement of the main theorem, Theorem~\ref{mainthm}, as well as an outline of the proof.  The following chapters (and Appendix~\ref{a_estimates}) give the details of the proof.  A weaker version of this result was announced by Martin in~\cite{mar}, and Helleloid and Martin~\cite{hm} prove the general result.

\section{Examples and Computational Data}

The claim that the automorphism group of a finite $p$-group is almost always a $p$-group may not seem entirely plausible, since many common finite $p$-groups have an automorphism group that is not a $p$-group.  The first finite $p$-groups that spring to mind are probably the abelian ones.  Any finite abelian $p$-group $G$ is isomorphic to $C_{p^{\lambda_1}} \times C_{p^{\lambda_2}} \times \cdots \times C_{p^{\lambda_k}}$ for some choice of integers $\lambda_1 \ge \lambda_2 \ge \cdots \ge \lambda_k \ge 0$, where $C_m$ denotes the cyclic group of order $m$.  Macdonald~\cite[Chapter II, Theorem 1.6]{mac} offers an exact formula for the order of $\Aut(G)$ in terms of $\lambda_1, \lambda_2, \dots, \lambda_k$ which shows that $\Aut(G)$ is a $p$-group if and only if $p = 2$ and the integers $\lambda_i$ are not all distinct.

Another family of finite $p$-groups consists of the Sylow $p$-subgroups $G$ of the general linear groups over $\F_q$, where $q$ is a power of $p$. Pavlov~\cite{pav} and Weir~\cite{wei} offer an explicit description of the automorphisms of $G$ and an exact determination of the structure of $\Aut(G)$ in terms of semi-direct products of elementary abelian and cyclic groups.  It follows from their work that $\Aut(G)$ is a $p$-group if and only if $p = q = 2$.

A third family of finite $p$-groups familiar to group theorists consists of the extraspecial $p$-groups.  These are the nonabelian $p$-groups $G$ whose center, commutator subgroup, and Frattini subgroup are all equal to each other and isomorphic to $C_p$.  Winter~\cite{win} shows that $\Aut(G) \cong H \ltimes \left< \theta \right>$, where $\theta$ is an automorphism of order $p-1$ and the quotient of $H$ by the (elementary abelian) inner automorphism group of $G$ is a certain subgroup of a symplectic group.  As a consequence, $\Aut(G)$ is not a $p$-group for any prime $p$.

Besides these explicit examples of automorphism groups which are not $p$-groups, Bryant and Kov\'{a}cs~\cite{bk} show that \emph{any} finite group occurs as a certain quotient of $\Aut(G)$ for some finite $p$-group $G$.  Of course, if this finite group is not a $p$-group, then $\Aut(G)$ will not be a $p$-group.  But in the course of proving the main theorem, we will show that the quotient in question is in fact almost always trivial.

Finding finite $p$-groups whose automorphism group is a $p$-group is reasonably easy when $p = 2$ and quite difficult when $p > 2$.  In the case of $p = 2$, as mentioned before, if $G$ is an abelian $2$-group and (in the notation used above) the integers $\lambda_i$ are not all distinct, then $\Aut(G)$ is a $2$-group.  The automorphism group of the dihedral $2$-group $D_{2^n}$ (for $n \ge 3$) is the 2-group $C_{2^{n-1}}^{\times} \ltimes C_2^{n-1}$.  The automorphism group of the generalized quaternion 2-group $Q_{2^n}$ (for $n \ge 4$) is the 2-group $C_{2^{n-1}} \ltimes (C_{2^{n-3}} \times C_2)$ (see Zhu and Zuo~\cite{zz}).  Newman and O'Brien~\cite{no} offer three more infinite families.

When $p > 2$, the known examples of finite $p$-groups whose automorphisms groups are $p$-groups are much more complicated.  In~\cite{hor}, for every integer $n \ge 2$, Horo{\v{s}}evski{\u\i} constructs such a $p$-group with class $n$ and, for every integer $d \ge 3$, constructs such a $p$-group that is minimally generated by $d$ elements.  Furthermore, Horo{\v{s}}evski{\u\i} shows in~\cite{hor} and~\cite{hor2} that for any prime $p$, if $G_1, G_2, \dots, G_n$ are finite $p$-groups whose automorphism groups are $p$-groups, then the automorphism group of the iterated wreath product $G_1 \wr G_2 \wr \cdots \wr G_n$ is also a $p$-group.  The other known examples arise from complicated and unnatural-looking constructions (see Webb~\cite{web}).  As previously mentioned, Chapter~\ref{c_examples} offers a more detailed survey of the automorphism groups of specific $p$-groups.

In a computational vein, Eick, Leedham-Green, and O'Brien~\cite{elo} describe an algorithm for constructing the automorphism group of a finite $p$-group.  This algorithm has been implemented by Eick and O'Brien in the GAP package AutPGroup~\cite{gap4}.  Compiled with the gracious help of Eamonn O'Brien (personal communication) and the GAP packages AutPGroup and SmallGroups~\cite{gap4}, Table~\ref{data} summarizes data on the proportion of small $p$-groups whose automorphism group is a $p$-group.  (More information about the SmallGroups package can be found in Besche, Eick and O'Brien~\cite{beo}.)  Our ability to compute the massive amount of data encapsulated in Table~\ref{data} is a testament to the power of the AutPGroup algorithm.

Table~\ref{data} does not offer enough data to make any firm conjectures, but we can make some observations.  First, the behavior for $p = 2$ and for $p > 2$ seems to be different; the proportions in the table for $p = 2$ are much higher than for $p > 2$.  We have no explanation for this other than the na\"{i}ve guess that $p$-groups ``often'' have automorphisms of order $p-1$, which prevents the automorphism group from being a $p$-group unless $p = 2$.  The second observation is that for $2 \le p \le 5$ and $3 \le n \le 7$, the proportion shown in the table is a non-decreasing function of $n$.  Finally, for $3 \le p \le 5$ and $3 \le n \le 7$, the proportion shown in the table is a non-decreasing function of $p$.

Again, this is hardly enough data to make any conjectures, but we might begin to hope that the proportion of $p$-groups of order $p^n$ whose automorphism group is a $p$-group tends to a limit as $p$ or $n$ goes to infinity, and perhaps even that the limit is 1.  These questions remain open (see Mann~\cite[Question 9]{man}).  Indeed, our main theorem does show that the automorphism group of a finite $p$-group is almost always a $p$-group, but the asymptotic sense in which we mean ``almost always'' does not refer to the order of the group.  The next section will explain what we mean by ``almost always'' and will state the main theorem.

\begin{table} \label{data}
\begin{center}
\begin{tabular}{|c|c|c|c|}
\hline
Order & $p = 2$ & $p = 3$ & $p = 5$ \\
\hline
$p^3$ & 3 of 5 & 0 of 5 & 0 of 5 \\
$p^4$ & 9 of 14 & 0 of 15 & 0 of 15 \\
$p^5$ & 36 of 51 & 0 of 67 & 1 of 77 \\
$p^6$ & 211 of 267 & 30 of 504 & 65 of 685 \\
$p^7$ & 2067 of 2328 & 2119 of 9310 & 11895 of 34297 \\
\hline
\end{tabular}
\parbox{3in}{\caption{The proportion of $p$-groups of a given order whose automorphism group is a $p$-group.}}
\end{center}
\end{table}

\section{The Main Theorem}

The precise statement of our theorem depends on the \emph{lower $p$-series} of a group.  The lower $p$-series of a group $G$ is the descending series of subgroups $G_1 \ge G_2 \ge \cdots$ defined inductively by $G_1 = G$ and $G_{i+1} = G_i^p [G, G_i]$.  Here, $G_i^p$ is the subgroup generated by $p$-th powers of elements of $G_i$, and $[G, G_i]$ is the subgroup generated by commutators consisting of an element from $G$ and an element from $G_i$.  Section~\ref{lowerpsec} explores the properties of the lower $p$-series in more detail; for the moment, it suffices to know that each $G_i$ is a characteristic subgroup of $G$ and that the quotients $G_i/G_{i+1}$ are all elementary abelian $p$-groups.  Since $G_i/G_{i+1}$ is an elementary abelian $p$-group, it is also an $\F_p$-vector space, and we will refer to the dimension $\dim(G_i/G_{i+1})$ of $G_i/G_{i+1}$ when we mean the dimension of $G_i/G_{i+1}$ as an $\F_p$-vector space.  We say that $G$ has \emph{lower $p$-length} $n$ if the number of non-identity terms in its lower $p$-series is $n$.  Also, for any group $G$, we let $d(G)$ denote the smallest cardinality of a generating set of $G$.  As we will see later, every $p$-group $G$ with lower $p$-length $n$ and $d(G) = d$ is finite, and there are finitely many such $p$-groups.  Finally we can state the main theorem.

\begin{thm} \label{mainthm}
Let $r_{p, d, n}$ be the proportion of $p$-groups $G$ with lower $p$-length at most $n$ and $d(G) = d$ whose automorphism group is a $p$-group.  If $n \ge 2$, then
\[
\lim_{d \to \infty}{r_{p, d, n}} = 1.
\]
If $d \ge 5$, then
\[
\lim_{n \to \infty}{r_{p, d, n}} = 1.
\]
If one of the following conditions is satisfied:
\begin{eqnarray}
\bullet && n=2, \nonumber \\
\bullet && n \ge 3 \textrm{ and } d \ge 17, \nonumber \\
\bullet && n \ge 4 \textrm{ and } d \ge 8, \label{dnres} \\
\bullet && n \ge 5 \textrm{ and } d \ge 6, \textrm{ or } \nonumber \\
\bullet && n \ge 10 \textrm{ and } d \ge 5, \nonumber
\end{eqnarray}
then
\[
\lim_{p \to \infty}{r_{p, d, n}} = 1.
\]
\end{thm}

Some of the given conditions on $d$ and $n$ are necessary.  For example, the only $p$-group with lower $p$-length 1 and $d(G) = d$ is $C_p^d$, so $r_{p,d,1} = 0$ for $p > 2$ or $d > 1$.  Similarly, the only $p$-group with lower $p$-length $n$ and $d(G) = 1$ is $C_{p^n}$, so $r_{p,1,n} = 0$ for $p > 2$ or $n > 1$.  However, it is not clear what conditions on $d$ and $n$ are absolutely necessary in Theorem~\ref{mainthm}.

\section{An Outline of the Proof of the Main Theorem}

The proof of Theorem~\ref{mainthm} breaks down into three parts, which are presented in Chapters~\ref{c_series}, \ref{c_normal}, and \ref{c_submodules}, and are assembled to prove Theorem~\ref{mainthm} at the end of this chapter.  In this section, we will outline the structure of the proof.

The first step is to connect the enumeration of finite $p$-groups to an analysis of certain subgroups and quotients of free groups.  In fact, we will prove bijections between certain families of finite $p$-groups and certain orbits of subgroups of a free group.  Let $F$ be the free group of rank $d$ and let $F_n$ be the $n$-th term in the lower $p$-series of $F$.  It turns out that the action of $\Aut(F/F_{n+1})$ on the vector space $F_n/F_{n+1}$ induces an action of $\GL(d, \F_p)$ on $F_n/F_{n+1}$, and the $\Aut(F/F_{n+1})$-orbits on the subgroups of $F/F_{n+1}$ lying in $F_n/F_{n+1}$ are also the $\GL(d, \F_p)$-orbits.  We say that an orbit is \emph{regular} if it has trivial stabilizer, that is, if the size of the orbit equals the size of the group that is acting.

For any finite $p$-group $G$, write $A(G)$ for the group of automorphisms of $G/\Phi(G)$ induced by $\Aut(G)$, where $\Phi(G)$ is the Frattini subgroup of $G$.  We shall see that if $A(G)$ is a $p$-group then so is $\Aut(G)$; in fact, our main goal is to prove that $A(G)$ is almost always trivial (in the same asymptotic sense as in Theorem~\ref{mainthm}).  In Chapter~\ref{c_series}, after defining and investigating the lower $p$-series, we prove the following theorem.

\begin{thm} \label{bijthm}
Fix a prime $p$ and integers $d, n \ge 2$.  Let $F$ be the free group of rank $d$ and define the following sets:
\begin{eqnarray*}
\mathcal{A}_{d,n} &=& \{\textrm{normal subgroups of $F/F_{n+1}$ lying in $F_2/F_{n+1}$}\} \\
\mathcal{B}_{d,n} &=& \{\textrm{normal subgroups of $F/F_{n+1}$ lying in $F_2/F_{n+1}$} \\
&& \qquad \textrm{and not containing $F_n/F_{n+1}$}\} \\
\mathcal{C}_{d,n} &=& \{\textrm{normal subgroups of $F/F_{n+1}$ lying in $F_n/F_{n+1}$}\} \\
\mathcal{D}_{d,n} &=& \{\textrm{normal subgroups of $F/F_{n+1}$ contained in the} \\
&& \qquad \textrm{regular $\GL(d,\F_p)$-orbits in $\mathfrak{C}_{d,n}$}\} \\
\\
\mathfrak{A}_{d,n} &=& \{\textrm{$\Aut(F/F_{n+1})$-orbits in $\mathcal{A}_{d,n}$}\} \\
\mathfrak{B}_{d,n} &=& \{\textrm{$\Aut(F/F_{n+1})$-orbits in $\mathcal{B}_{d,n}$}\} \\
\mathfrak{C}_{d,n} &=& \{\textrm{$\Aut(F/F_{n+1})$-orbits in $\mathcal{C}_{d,n}$}\} = \{\textrm{$\GL(d,\F_p)$-orbits in $\mathcal{C}_{d,n}$}\} \\
\mathfrak{D}_{d,n} &=& \{\textrm{regular $\GL(d,\F_p)$-orbits in $\mathcal{C}_{d,n}$}\}.
\end{eqnarray*}
Then there is a well-defined map $\pi_{d,n}: \mathfrak{A}_{d,n} \to \{\textrm{finite $p$-groups}\}$ given by $L/F_{n+1} \mapsto F/L$, where $L/F_{n+1} \in \mathcal{A}_{d,n}$.  Furthermore $\pi_{d,n}$ induces bijections
\begin{eqnarray*}
\mathfrak{A}_{d,n} &\leftrightarrow& \{\textrm{$p$-groups $H$ of lower $p$-length at most $n$ with $d(H) = d$}\} \\
\mathfrak{B}_{d,n} &\leftrightarrow& \{\textrm{$p$-groups $H$ of lower $p$-length $n$ with $d(H) = d$}\} \\
\mathfrak{D}_{d,n} &\leftrightarrow& \{\textrm{$p$-groups $H$ in $\pi_{d,n}(\mathfrak{C}_{d,n})$ with $A(H) = 1$}\}.
\end{eqnarray*}
\end{thm}

In order to understand the usefulness of this theorem, note that $\mathfrak{A}_{d,n}$ is in bijection with finite $p$-groups $H$ of lower $p$-length at most $n$ with $d(H) = d$, and this bijection restricts to a bijection between $\mathfrak{D}_{d,n}$ and some of these $p$-groups whose automorphism groups are $p$-groups.  As we will see in Chapter~\ref{c_free}, $F/F_{n+1}$ is a finite group, and so $\mathfrak{A}_{d,n}$ is finite.  Therefore the ratio $|\mathfrak{D}_{d,n}|/|\mathfrak{A}_{d,n}|$ is well-defined and is at most the proportion of $p$-groups $H$ with lower $p$-length at most $n$ and $d(H) = d$ whose automorphism group is a $p$-group; in the notation of Theorem~\ref{mainthm}, $|\mathfrak{D}_{d,n}|/|\mathfrak{A}_{d,n}| \le r_{p,d,n}$.  So to prove the limiting statements about $r_{p,d,n}$ from Theorem~\ref{mainthm}, it suffices to prove the same limiting statements about $|\mathfrak{D}_{d,n}|/|\mathfrak{A}_{d,n}|$.  We state this formally as a corollary of Theorem~\ref{bijthm}.

\begin{cor}
\label{bijcor}
Suppose that
\[
\lim_{d \to \infty}{\frac{|\mathfrak{D}_{d,n}|}{|\mathfrak{A}_{d,n}|}} = 1
\]
for $n \ge 2$,
\[
\lim_{n \to \infty}{\frac{|\mathfrak{D}_{d,n}|}{|\mathfrak{A}_{d,n}|}} = 1
\]
for $d \ge 5$, and
\[
\lim_{p \to \infty}{\frac{|\mathfrak{D}_{d,n}|}{|\mathfrak{A}_{d,n}|}} = 1.
\]
for $d$ and $n$ satisfying one of the conditions in (\ref{dnres}).  Then Theorem~\ref{mainthm} is true.
\end{cor}

We will prove the hypotheses of Corollary~\ref{bijcor} in two steps.  Namely, we will show that the ratios $|\mathfrak{C}_{d,n}|/|\mathfrak{A}_{d,n}|$ and $|\mathfrak{D}_{d,n}|/|\mathfrak{C}_{d,n}|$ satisfy the same limiting statements (using the results of Chapters~\ref{c_normal} and~\ref{c_submodules} respectively), and therefore so does the ratio $|\mathfrak{D}_{d,n}|/|\mathfrak{A}_{d,n}|$.  Both of these steps require some knowledge of the structure of $F_n/F_{n+1}$.  In Chapter~\ref{c_free}, we prove that $F_n/F_{n+1}$ is isomorphic (as a $\F_p \GL(d,\F_p)$-module) to part of the free Lie algebra over $\F_p$.  In particular, this lets us use the combinatorics of the free Lie algebra to compute certain parameters of the group $F/F_{n+1}$ (see Corollary~\ref{expcor}).

To prove that the ratios $|\mathfrak{C}_{d,n}|/|\mathfrak{A}_{d,n}|$ and $|\mathfrak{D}_{d,n}|/|\mathfrak{C}_{d,n}|$ satisfy the desired limiting statements, we obtain explicit lower bounds for each in Theorems~\ref{limit1thm} and~\ref{limit2thm} respectively.  To state these bounds, we define, for any number $x > 1$, the quantities
\begin{equation}
\label{cdeqns}
C(x) = \sum_{r=-\infty}^{\infty}{x^{-r^2}} \textrm{ and } D(x) = \prod_{j=1}^{\infty}{\frac{1}{1-x^{-j}}}.
\end{equation}

\begin{thm} \label{limit1thm}
Fix a prime $p$ and integers $d$ and $n$ so that either $n \ge 3$ and $d \ge 6$ or $n \ge 10$ and $d \ge 5$.  Let $F$ be the free group of rank $d$ and let $d_i$ be the dimension of $F_i/F_{i+1}$ for $i = 1, \dots, n$.  Then
\[
1 \le \frac{|\mathfrak{A}_{d,n}|}{|\mathfrak{C}_{d,n}|} \le 1 + C(p^{15/16}) C(p)^{n-2} D(p)^{n-2} p^{d_{n-1} - d_n/4 + d^2-11/16}.
\]
\end{thm}

\begin{cor} \label{limit1cor}
If $n \ge 2$, then
\[
\lim_{d \to \infty}{\frac{|\mathfrak{C}_{d,n}|}{|\mathfrak{A}_{d,n}|}} = 1.
\]
If $d \ge 5$, then
\[
\lim_{n \to \infty}{\frac{|\mathfrak{C}_{d,n}|}{|\mathfrak{A}_{d,n}|}} = 1.
\]
If $d$ and $n$ satisfy one of the conditions in (\ref{dnres}), then
\[
\lim_{p \to \infty}{\frac{|\mathfrak{C}_{d,n}|}{|\mathfrak{A}_{d,n}|}} = 1
\]
\end{cor}

Theorem~\ref{normalthm} gives an upper bound on the number of normal subgroups of a finite $p$-group, and the proof of Theorem~\ref{limit1thm} applies this theorem to the quotient $F/F_{n+1}$.  Corollary~\ref{limit1cor} will follow using bounds on $d_n$ from Lemma~\ref{dnasymptotics2}, showing that $|\mathfrak{C}_{d,n}|/|\mathfrak{A}_{d,n}|$ satisfies the desired limiting statements.

\renewcommand{\arraystretch}{1.2}
\begin{thm} \label{limit2thm}
Fix a prime $p$ and integers $d$ and $n$ so that either $n = 2$ and $d \ge 10$ or $n \ge 3$ and $d \ge 3$.  Let $F$ be the free group of rank $d$ and let $d_i$ be the dimension of $F_i/F_{i+1}$ for $i = 1, \dots, n$.  Let
\[
c_1 = \left\{
  \begin{array}{r@{\quad:\quad}l}
    C(p)^5 D(p)^4 p^{17/4} & \textrm{$n = 2$ and $d \ge 10$} \\
    C(p)^2 D(p) p^{3/4} & n \ge 3.
  \end{array}
  \right.
\]
Let
\[
c_2 = \left\{
  \begin{array}{r@{\quad:\quad}l}
    -d & n = 2 \\
    d^2-d_n/2 & n \ge 3.
  \end{array}
  \right.
\]
Then
\begin{enumerate}
\renewcommand{\labelenumi}{\emph{(\alph{enumi})}}
\renewcommand{\theenumi}{\ref{est2thm}(\alph{enumi})}
\item
\[
1 \le \frac{|\mathfrak{C}_{d, n}| \cdot |\GL(d,\F_p)|}{|\mathcal{C}_{d, n}|} \le 1 + c_1 p^{c_2}.
\]
\item
\[
1 \le \frac{|\mathfrak{C}_{d, n}|}{|\mathfrak{D}_{d, n}|} \le \frac{1+ c_1 p^{c_2}}{1-c_1 p^{c_2}}.
\]
\end{enumerate}
\end{thm}
\renewcommand{\arraystretch}{1}

\begin{cor} \label{limit2cor}
If $n \ge 2$, then
\[
\lim_{d \to \infty}{\frac{|\mathfrak{D}_{d,n}|}{|\mathfrak{C}_{d,n}|}} = 1.
\]
If $d \ge 3$, then
\[
\lim_{n \to \infty}{\frac{|\mathfrak{D}_{d,n}|}{|\mathfrak{C}_{d,n}|}} = 1.
\]
If $n = 2$ and $d \ge 10$, or $n \ge 3$ and $d \ge 5$, or $n \ge 4$ and $d \ge 3$, then
\[
\lim_{p \to \infty}{\frac{|\mathfrak{D}_{d,n}|}{|\mathfrak{C}_{d,n}|}} = 1
\]
\end{cor}

In proving Theorem~\ref{limit2thm}, we use the Cauchy-Frobenius Lemma to estimate $|\mathfrak{C}_{d,n}|$ by analyzing the submodule structure of $F_n/F_{n+1}$ as an $\F_p \GL(d,\F_p)$-module.  As part of this analysis, we use the theory of Hall polynomials to count the number of submodules of fixed type of a finite module over a discrete valuation ring. Corollary~\ref{limit2cor} will follow using bounds on $d_n$ from Lemma~\ref{dnasymptotics2}, showing that $|\mathfrak{D}_{d,n}|/|\mathfrak{C}_{d,n}|$ satisfies the desired limiting statements.

In stating Theorems~\ref{limit1thm} and~\ref{limit2thm}, we have judged it more satisfactory to give explicit numerical bounds, even though the proof of Theorem~\ref{mainthm} requires only asymptotic bounds.  However, since we have no expectation that our proof method gives bounds that are sharp, we have opted for (relatively) clean explicit bounds rather than the best possible.

Appendix~\ref{a_estimates} contains combinatorial estimates, including bounds on Gaussian coefficients, that are needed in Chapters~\ref{c_normal} and~\ref{c_submodules}.

\section{Concluding Remarks on the Main Theorem}
\label{summarysec}

In this section, after formally proving Theorem~\ref{mainthm} by citing results from the previous section, we will state some minor variants and consequences.

\begin{thmx}[Theorem~\ref{mainthm}]
Let $r_{p, d, n}$ be the proportion of $p$-groups $G$ with lower $p$-length at most $n$ and $d(G) = d$ whose automorphism group is a $p$-group.  If $n \ge 2$, then
\[
\lim_{d \to \infty}{r_{p, d, n}} = 1.
\]
If $d \ge 5$, then
\[
\lim_{n \to \infty}{r_{p, d, n}} = 1.
\]
If one of the following conditions is satisfied:
\begin{eqnarray}
\bullet && n=2, \nonumber \\
\bullet && n \ge 3 \textrm{ and } d \ge 17, \nonumber \\
\bullet && n \ge 4 \textrm{ and } d \ge 8, \nonumber \\
\bullet && n \ge 5 \textrm{ and } d \ge 6, \textrm{ or } \nonumber \\
\bullet && n \ge 10 \textrm{ and } d \ge 5, \nonumber
\end{eqnarray}
then
\[
\lim_{p \to \infty}{r_{p, d, n}} = 1.
\]
\end{thmx}

\begin{proof}
This follows directly from Corollaries~\ref{bijcor},~\ref{limit1cor}, and~\ref{limit2cor}.
\end{proof}

\begin{cor}
\label{hpcor}
Let $s_{p, d, n}$ be the proportion of $p$-groups $G$ with lower $p$-length at most $n$ and $d(G) \le d$ whose automorphism group is a $p$-group.  If $n \ge 2$, then
\[
\lim_{d \to \infty}{s_{p, d, n}} = 1.
\]
\end{cor}

\begin{proof}
This follows directly from Theorem~\ref{mainthm} and the trivial observation that while the number of $p$-groups generated by at most $d$ elements and with lower $p$-length at most $n$ is finite, the number of $p$-groups with lower $p$-length at most $n$ is infinite.
\end{proof}

\begin{cor} \label{maincor}
Let $t_{p,d,n}$ be the proportion of $p$-groups $G$ with lower $p$-length $n$ and $d(G) = d$ whose automorphism group is a $p$-group.  If $n \ge 2$, then
\[
\lim_{d \to \infty}{t_{p, d, n}} = 1.
\]
If $d \ge 5$, then
\[
\lim_{n \to \infty}{t_{p, d, n}} = 1.
\]
If $d$ and $n$ satisfy one of the conditions in~\ref{dnres}, then
\[
\lim_{p \to \infty}{t_{p, d, n}} = 1.
\]
\end{cor}

\begin{proof}
The number of $p$-groups $G$ with lower $p$-length $n$ and $d(G) = d$ is $|\mathfrak{B}_{d,n}|$, so $t_{p,d,n} \ge |\mathfrak{D}_{d,n}|/|\mathfrak{B}_{d,n}|$.  As $\mathfrak{D}_{d,n} \subseteq \mathfrak{B}_{d,n} \cup \{F_n/F_{n+1}\} \subseteq \mathfrak{A}_{d,n}$, it follows from Theorem~\ref{mainthm} that $(|\mathfrak{B}_{d,n}|+1)/|\mathfrak{D}_{d,n}| \to 1$ for each of the limits (with corresponding conditions on $d$ and/or $n$) in question.
Since $|\mathfrak{A}_{d,n}| \to \infty$ in each case, Theorem~\ref{mainthm} implies that $|\mathfrak{D}_{d,n}| \to \infty$ as well.  This proves that $|\mathfrak{B}_{d,n}|/|\mathfrak{D}_{d,n}| \to 1$ for each of the limits in question.
\end{proof}

\begin{cor}
Let $u_{p, d, n}$ be the proportion of $p$-groups $G$ with lower $p$-length $n$ and $d(G) \le d$ whose automorphism group is a $p$-group.  If $n \ge 2$, then
\[
\lim_{d \to \infty}{u_{p, d, n}} = 1.
\]
\end{cor}

\begin{proof}
This corollary follows from Corollary~\ref{maincor} just as Corollary~\ref{hpcor} follows from Corollary~\ref{mainthm}.
\end{proof}

Using Corollary~\ref{hpcor}, Henn and Priddy~\cite{hp} prove the following theorem.

\begin{thm}[Henn and Priddy \cite{hp}] \label{hpthm}
Let $v_{p,d,n}$ be the proportion of $p$-groups $P$ with lower $p$-length at most $n$ and $d(P) \le d$ that satisfy the following property: if $G$ is a finite group with Sylow $p$-subgroup $P$, then $G$ has a normal $p$-complement.  If $n \ge 2$, then $\lim_{d \to \infty}{v_{p,d,n}} = 1$.
\end{thm}

As mentioned earlier in this chapter, the following question remains unanswered.

\begin{question}
Let $w_{p,n}$ be the proportion of $p$-groups with order at most $p^n$ whose automorphism group is a $p$-group.  Is it true that $\lim_{n \to \infty}{w_{p,n}} = 1$?
\end{question}

\chapter{The Lower $p$-Series and the Enumeration of $p$-Groups}
\label{c_series}
\chaptermark{The Lower \lowercase{$p$}-Series}

If we hope to prove that the automorphism group of almost every $p$-group is itself a $p$-group, there are two initial questions to answer: what do we mean by ``almost always'', and how do we relate the set of finite $p$-groups (and their automorphism groups) to something we can actually work with?  As mentioned in Chapter~\ref{c_main}, the answer to both of those questions starts with the lower $p$-series, a central series defined for all groups.  This chapter begins with an introduction to the lower $p$-series and its basic properties.  It follows with the connection between the lower $p$-series and automorphisms, and it closes with theorems on the correspondence between $p$-groups in a variety and orbits of subgroups of a free group.

\section{The Lower \lowercase{$p$}-Series}
\label{lowerpsec}

The lower $p$-series of a group was introduced independently by Skopin~\cite{sko} and Lazard~\cite{laz}.  It is described in detail by Huppert and Blackburn~\cite[Chapter VIII]{hb} (under the name $\lambda$-series) and by Bryant and Kov\'{a}cs~\cite{bk}.  It has also been called the lower central $p$-series, the lower exponent-$p$ central series, or the Frattini series.

The lower $p$-series is particularly suited to computer analysis of finite $p$-groups and forms the basis of the $p$-group generation algorithm of Newman~\cite{new}.  This algorithm is described in greater detail in O'Brien~\cite{obr}.  It was modified in~\cite{obr2} and~\cite{elo} to construct automorphism groups of finite $p$-groups.  It should also be mentioned that results on the lower $p$-series have appeared in~\cite{elo} and~\cite{obr}, while the link between the lower $p$-series and automorphisms described in Section~\ref{autgroupsec} is an extension of results that Higman~\cite{hig} and Sims~\cite{sim} used to count finite $p$-groups.

\begin{defn}
Fix a prime $p$.  For any group $G$, the \emph{lower $p$-series} $G = G_1 \ge G_2 \ge \cdots$ of $G$ is defined by $G_{i+1} = G_i^p [G_i, G]$ for $i \ge 1$.  $G$ is said to have \emph{lower $p$-length} $n$ if $G_n$ is the last non-identity term in the lower $p$-series.
\end{defn}

For an example of the lower $p$-series, suppose that $G$ is a finite abelian $p$-group.  Then all commutators of elements in $G$ are trivial, so $G_{i+1} = G_i^p$.  We can say precisely what each subgroup $G_i$ is.  Recall that a partition $\lambda$ of $n$ is a sequence of integers $\lambda_1 \ge \lambda_2 \ge \cdots \ge \lambda_k \ge 0$ such that $\sum_{j=1}^k{\lambda_j} = n$.  A finite abelian $p$-group of order $p^n$ has \emph{type} $\lambda$ if it is isomorphic to
\[
C_{p^{\lambda_1}} \times C_{p^{\lambda_2}} \times \cdots \times C_{p^{\lambda_k}}.
\]
So suppose that $G$ has type $\lambda$.  For each positive integer $i$, define a new partition $\lambda^{(i)}$ by $\lambda^{(i)}_j = \max{\{\lambda_j-i+1, 0\}}$ for $j = 1, 2, \dots, k$.  Then $G_i$ is a finite abelian $p$-group of type $\lambda^{(i)}$.  In particular, $G_i$ is non-trivial if and only if $i \le \lambda_1$, and so the lower $p$-length of $G$ is $\lambda_1$.

For a second example of the lower $p$-series, let $G$ be the group of $(n+1) \times (n+1)$ upper triangular matrices with entries in $\F_p$ and ones on the diagonal; this is a Sylow $p$-subgroup of $\GL(n+1, \F_p)$.  Then $G_i$ consists of all matrices in $G$ whose entry in position $(j,k)$ is 0 if $0 < k-j < i$.

Before we list some basic facts about the lower $p$-series, recall that a subgroup is \emph{fully invariant} if every endomorphism of the group restricts to an endomorphism of the subgroup.  For any group $G$, we write $G = \gamma_1(G) \ge \gamma_2(G) \ge \cdots$ to denote the \emph{lower central series} of $G$, where $\gamma_{i+1}(G) = [\gamma_i(G), G]$.  The following proposition states five fundamental properties of the lower $p$-series; the first four facts are proved in Huppert and Blackburn~\cite[Chapter VIII, Theorem 1.5 and Corollary 1.6]{hb} and the fifth fact is obvious by induction.

\begin{prop}
\label{fratprops}
For any group $G$ and for all positive integers $i$ and $j$,
\begin{enumerate}
\item $[G_i, G_j] \le G_{i+j}$.
\item $G_i^{p^j} \le G_{i+j}$.
\item $G_i = \gamma_1(G)^{p^{i-1}} \gamma_2(G)^{p^{i-2}} \cdots \gamma_i(G)$.
\item $G_{i+1}$ is the smallest normal subgroup of $G$ lying in $G_i$ such that $G_i/G_{i+1}$ is an elementary abelian $p$-group and is central in $G/G_{i+1}$.
\item $G_i$ is fully invariant in $G$.
\end{enumerate}
\end{prop}

As we will see, the fact that $G_i/G_{i+1}$ is elementary abelian, and therefore an $\F_p$-vector space, is a key reason we are able to prove the main theorem.  It is also important that the lower $p$-length has a special significance for finite groups.

\begin{prop}
Let $G$ be a finite group.  Then $G$ is a $p$-group if and only if $G$ has finite lower $p$-length.
\end{prop}

\begin{proof}
Since the order of $G_i/G_{i+1}$ is a power of $p$ for all $i$, it is clear that if $G$ is not a $p$-group, then $G$ has infinite lower $p$-length.  In the other direction, suppose $G$ is a $p$-group.  It suffices to show that if $G_i$ is non-trivial, then $G_{i+1} < G_i$.  Since $G$ is nilpotent, $[G_i, G] < G_i$ (see Kurzweil and Stellmacher~\cite[Lemma 5.1.6]{ks}).  Then $G_i/[G_i, G]$ is a non-trivial abelian $p$-group.  Hence
\[
G_i/[G_i, G] > (G_i/[G_i, G])^p = G_i^p[G_i, G]/[G_i, G],
\]
and so $G_i > G_i^p[G_i, G] = G_{i+1}$.
\end{proof}

Note that if $G$ is a finite $p$-group, then $G_2 = \Phi(G)$, the Frattini subgroup of $G$ (see, for example, Kurzweil and Stellmacher~\cite[Lemma 5.2.8]{ks}).  As a consequence, the Burnside Basis Theorem says that the smallest cardinality $d(G)$ of a generating set of $G$ equals $\dim(G/G_2)$, and that any lift to $G$ of a generating set of $G/G_2$ generates $G$.

We are actually interested in the lower $p$-series of free groups of finite rank as well as that of finite $p$-groups.  The reason is that the lower $p$-series of a finite $p$-group is related to the lower $p$-series of a free group in the following way.  Let $G$ be a finite $p$-group, and let $F$ be the free group of rank $d(G)$.  Then $G$ is isomorphic to $F/U$ for some normal subgroup $U$ of $F$.  It is easy to see by induction that $G_i \cong F_iU/U$ for all $i$; namely, if $G_i \cong F_iU/U$, then
\begin{eqnarray*}
G_{i+1} &\cong& (F_iU/U)^p [F_iU/U, F/U] \\
&\cong& F_i^p [F_i, F] U/U \\
&\cong& F_{i+1}U/U.
\end{eqnarray*}
It follows that the lower $p$-length of $G$ is $n$, where $F_{n+1}$ is the first term in the lower $p$-series of $F$ that is contained in $U$.  The lower $p$-series of $F$ will be discussed in detail in Chapter~\ref{c_free}, but we mention here that the groups $F/F_n$ are finite $p$-groups for all $n$.

\section{The Lower \lowercase{$p$}-Series and Automorphisms}
\label{autgroupsec}

In this section we collect some necessary facts linking the lower $p$-series and automorphisms.  The first proposition is fundamental to our overall proof strategy, while the remaining propositions are easy technical lemmas that will be used in Section~\ref{enumvarsec}.  To begin, suppose that $G$ is a finite $p$-group, and let $d = d(G)$.  Any automorphism of $G$ induces an element of $\Aut(G/G_2) \cong \GL(d, \F_p)$.  Thus we obtain a map from $\Aut(G)$ to $\GL(d, \F_p)$ and an exact sequence
\[
1 \to K(G) \to \Aut(G) \to A(G) \to 1,
\]
where $A(G)$ is a subgroup of $\GL(d, \F_p)$.  The group $K(G)$ acts trivially on $G/G_2$, and hence on each factor $G_i/G_{i+1}$ (see Huppert and Blackburn~\cite[Chapter VIII, Theorem 1.7]{hb}).  As $\Aut(G)$ acts on each $G_i/G_{i+1}$ and the kernel of the action contains $K(G)$, we obtain an action of $A(G)$ on each $G_i/G_{i+1}$.  The following key proposition is due to P. Hall~\cite[Section 1.3]{hal}.

\begin{prop}
If $G$ is a finite $p$-group, then so is $K(G)$.
\end{prop}

\begin{proof}
Suppose $\sigma \in K(G)$ has order $q$, where $q$ is a prime not equal to $p$ or $q = 1$.  Any coset $xG_2$ of $G_2$ in $G$ is fixed by $\sigma$, since $\sigma$ acts trivially on $G/G_2$.  The orbit of an element of $xG_2$ under $\sigma$ has size $1$ or $q$, and $|xG_2|$ is a power of $p$, so some element of $xG_2$ is fixed by $\sigma$.  Every coset of $G_2$ contains an element fixed by $\sigma$, and since $G_2$ is the Frattini subgroup of $G$, these coset representatives generate $G$.  Thus $\sigma$ fixes $G$ and $q = 1$.  Hence $K(G)$ is a $p$-group.
\end{proof}

\begin{prop}
\label{fratautprop}
If $G$ is a finite $p$-group and $\sigma$ is an endomorphism of $G$ that induces an automorphism on $G/G_2$, then $\sigma$ is an automorphism of $G$.
\end{prop}

\begin{proof}
The image of $G$ under $\sigma$ contains coset representatives for each coset of $G/G_2$.  These coset representatives generate $G$, so the image of $G$ under $\sigma$ is all of $G$.  Hence $\sigma$ is an automorphism.
\end{proof}

Let $F$ be the free group of rank $d$ with free generating set $y_1, y_2, \dots, y_d$.

\begin{prop}
\label{f2maxlprop}
If $U$ is a proper fully invariant subgroup of $F$, then $U \le F_2$ and $d(F/U)$, the smallest cardinality of a generating set of $F/U$, equals $d$.
\end{prop}

\begin{proof}
Suppose $U \not\le F_2$ is a fully invariant subgroup of $F$.  The elements $y_1^{a_1} \cdots y_d^{a_d}$, with $0 \le a_i < p$, form a complete set of coset representatives for the cosets of $F_2$ in $F$, so $U$ contains an element $y$ in a coset $y_1^{a_1} \cdots y_d^{a_d} F_2$ with some $a_i$ nonzero.  Then the endomorphism of $F$ that sends $y_j$ to 1 for $j \neq i$ and sends $y_i$ to $y_k^{a_i^{-1}}$ for some $k = 1, \dots, d$ sends $y$ into the coset $y_k F_2$.  Since $k$ was arbitrary, this shows that $U$ contains coset representatives of $y_k F_2$ for all $k = 1, \dots, d$.  These cosets generate $F/F_2$, and so the coset representatives generate $F$.  Hence $U = F$.

Finally, since $d(F) = d(F/F_2)$, any normal subgroup $U$ of $F$ contained in $F_2$ satisfies $d(U) = d$.
\end{proof}

\begin{prop}
\label{f2liftprop}
Let $U$ be a fully invariant subgroup of $F$ contained in $F_2$ and suppose that $G = F/U$ is a finite $p$-group.  Then any automorphism $\sigma$ of $F/F_2$ lifts to an automorphism of $G$.
\end{prop}

\begin{proof}
Since $F$ is free, there is an endomorphism $\sigma'$ of $F$ such that $\sigma'(y_i) \in \sigma(y_i F_2)$ for all $i = 1, \dots, d$.  Therefore $\sigma'(y) \in \sigma(y F_2)$ for all $y \in F$.  Then $\sigma'$ induces $\sigma$ on $F/F_2$, and since $U$ is fully invariant, maps $U$ to itself.  So $\sigma'$ induces an endomorphism $\sigma''$ of $G$.  But $\sigma''$ induces $\sigma$, an automorphism of $F/F_2 \cong (F/U)/(F_2/U) \cong G/G_2$.  By Proposition~\ref{fratautprop}, $\sigma''$ is an automorphism of $G$.  Thus $\sigma$ lifts to an automorphism $\sigma''$ of $G$.
\end{proof}

\section{Enumerating Groups in a Variety}
\label{enumvarsec}

In this section, we develop a general strategy for enumerating certain sets of $p$-groups and apply this strategy to prove Theorem~\ref{bijthm}.  The key idea to use the theory of varieties of groups.  Our exposition follows Neumann~\cite[Sections 1.2--1.4]{neu}.

Let $X_{\infty}$ be the free group freely generated by $X = \{x_1, x_2, \dots\}$.  A word $w$ is an element of $X_{\infty}$.  A word $w$ is a \emph{law} for a group $G$ if $\alpha(w) = 1$ for every $\alpha \in \Hom(X_{\infty}, G)$.  Each subset $W$ of $X_{\infty}$ defines a \emph{variety of groups} $V$ consisting of all groups for which each word in $W$ is a law.  For example, the class of abelian groups forms the variety $V$ defined by the singleton set $W = \{x_1 x_2 x_1^{-1} x_2^{-1}\}$.  More relevant to our investigations is the variety of $p$-groups of lower $p$-length at most $n$.  This variety is defined by (for example) the set $W = (X_{\infty})_{n+1}$.

For each positive integer $d$, the variety $V$ contains a \emph{relatively free group of rank $d$}.  This is the group $F/U$, where $F$ is the free group of rank $d$ and $U$ is the (fully invariant) subgroup of $F$ generated by the values $\alpha(w)$ for all $\alpha \in \Hom(X_{\infty}, F)$.  In the variety of abelian groups, the relatively free group of rank $d$ is (isomorphic to) $\Z^d$.  In the variety of $p$-groups of lower $p$-length at most $n$, the relatively free group of rank $d$ is $F/F_{n+1}$.  By Proposition~\ref{f2maxlprop}, the smallest cardinality of a generating set of $F/U$ is $d(F/U) = d$.  The relatively free group of rank $d$ has a generating set of cardinality $d$, called a \emph{set of free generators}, such that every mapping of this generating set into the group can be extended to an endomorphism.

When the relatively free group $G$ of rank $d$ in a variety $V$ is a finite non-trivial $p$-group, we can describe $K(G)$ and $A(G)$ precisely.  In particular, by taking $V$ to be the variety of $p$-groups of lower $p$-length at most $n$ and $G = F/F_{n+1}$, we can find $K(F/F_{n+1})$ and $A(F/F_{n+1})$.  Furthermore, questions about groups in $V$ and their automorphism groups can be translated into questions about certain orbits of subgroups of $G$.  This is the content of Theorems~\ref{isoenum1} and~\ref{isoenum2}, from which Theorem~\ref{bijthm} follows by specializing to the variety of $p$-groups of lower $p$-length at most $n$.

\begin{thm} \label{isoenum1}
Suppose that $G$ is the relatively free group of rank $d$ in a variety of groups $V$, and suppose that $G$ is a finite $p$-group.  Let $\mathcal{A}$ be the set of normal subgroups of $G$ lying in $G_2$, and let $\mathfrak{A}$ be the $\Aut(G)$-orbits in $\mathcal{A}$.  Then
\begin{eqnarray*}
\pi : \mathfrak{A} &\to& \{H \in V \; : \; d(H) = d\} \\
L &\mapsto& G/L,
\end{eqnarray*}
where $L \in \mathcal{A}$, is a well-defined bijection.

Fix $L \in \mathcal{A}$ and let $H = G/L$.  Write $N_{\Aut(G)}(L)$ for the normalizer of $L$ in $\Aut(G)$ and $B(L)$ for the normal subgroup of $N_{\Aut(G)}(L)$ that acts trivially on $H$.  Then,
\[
1 \to B(L) \to N_{\Aut(G)}(L) \to \Aut(H) \to 1
\]
is exact.  The subgroup $B(L)$ is isomorphic to the direct product of $d$ copies of $L$.  If $L = G_2$, then $\Aut(G) = N_{\Aut(G)}(G_2)$, $K(G) = B(G_2)$, and $A(G) = \Aut(G/G_2) \cong \GL(d,\F_p)$.
\end{thm}

\begin{proof}
By Proposition~\ref{f2liftprop}, any automorphism of $F/F_2 \cong G/G_2$ lifts to an automorphism of $G$.  Thus $A(G)$ is isomorphic to the full automorphism group of $G/G_2$, namely $\GL(d,\F_p)$.

Up to isomorphism, $G/L$ depends only on the orbit of $L$, so $\pi$ is well-defined on $\mathfrak{A}$.  To prove that $\pi$ is surjective, consider any group $H \in V$ with $d(H) = d$.  Evidently $H$ is isomorphic to $G/L$ for some normal subgroup $L$ of $G$.  If $L$ were not contained in $G_2$, then we could choose $h_1 \in L \setminus G_2$ and extend $\{h_1\}$ to a generating set $\{h_1, h_2, \dots, h_d\}$ of $G$.  But then $H \cong G/L$ would be generated by the images of $h_2, \dots, h_d$, contradicting $d(H) = d$.  So $L$ is contained in $G_2$, and $H$ is in the image of the map $\pi$.

To prove that $\pi$ is injective, suppose that $L$ and $M$ are in $\mathcal{A}$ and $G/L \cong G/M$.  Let $\beta : G/L \to G/M$ be an isomorphism.  By~\cite[Theorem 44.21]{neu}, there is an endomorphism $\gamma : G \to G$ so that the diagram in Figure 1 commutes.  Then $\gamma$ induces $\beta$, and $\beta$ induces an automorphism on $G/G_2$, since $(G/L)/(G_2/L)$ and $(G/M)/(G_2/M)$ are canonically isomorphic to $G/G_2$.  It follows from Proposition~\ref{fratautprop} that $\gamma$ is an automorphism of $G$.  From Figure 1, it is also clear that $\gamma(L) \le M$.  Thus $\gamma(L) = M$, and $L$ and $M$ are in the same $\Aut(G)$-orbit.

\[
\begin{array}{c}
\xymatrix{
G \ar@{.>}[d]_{\gamma} \ar[r] & G/L \ar[d]^{\beta} \\
G \ar[r] & G/M
} \\
\textrm{Figure 1}
\end{array}
\]

If we take $L = M$, we find that any automorphism of $H = G/L$ is induced by an automorphism of $G$, so that $\Aut(H) \cong N_{\Aut(G)}(L)/B(L)$.

Let $g_1, g_2, \dots, g_d$ be a set of free generators for $G$, and let $\ell_1, \ell_2, \dots, \ell_d$ be any elements of $L$.  Since $G$ is relatively free, the map $\sigma$ that sends $g_i$ to $g_i \ell_i$ for all $i = 1, \dots, d$ extends to an endomorphism of $G$.  Then $\sigma$ acts trivially on $G/L$, and hence acts trivially on $G/G_2$, so $\sigma$ is an automorphism by Proposition~\ref{fratautprop}.  Conversely, any automorphism of $G$ that acts trivially on $G/L$ must act on each $g_i$ as multiplication by an element of $L$.  Thus $B(L)$ is isomorphic to the direct product of $d$ copies of $L$.  The specialized statements for $L = G_2$ follow directly from the definition of $K(G)$ and the fact that $G/G_2 \cong (C_p)^d$.
\end{proof}

\begin{thm} \label{isoenum2}
Suppose that $G$ is the relatively free group of rank $d$ in a variety of groups $V$, and suppose that $G$ is a finite $p$-group with lower $p$-length $n \ge 2$.  Let $\mathcal{C}$ be the set of normal subgroups of $G$ lying in $G_n$, and let $\mathfrak{C}$ be the $\Aut(G)$-orbits in $\mathcal{C}$.  Then the map $\pi$ defined in Theorem~\ref{isoenum2} restricts to a well-defined bijection
\begin{eqnarray*}
\pi |_{\mathfrak{C}} : \mathfrak{C} &\to& \{H \in V \; : \; d(H) = d \textrm{ and } H/H_n \cong G/G_n \} \\
L &\mapsto& G/L,
\end{eqnarray*}
where $L \in \mathcal{C}$.

The subgroup $K(G)$ of $\Aut(G)$ acts trivially on $\mathcal{C}$, so $A(G) \cong \Aut(G)/K(G) \cong \GL(d, \F_p)$ acts on $\mathcal{C}$, and the $\Aut(G)-$ and $\GL(d,\F_p)-$orbits on $\mathcal{C}$ are identical.

Fix $L \in \mathcal{C}$ and let $H = G/L$.  Write $N_{\Aut(G)}(L)$ for the normalizer of $L$ in $\Aut(G)$ and $B(L)$ for the normal subgroup of $N_{\Aut(G)}(L)$ that acts trivially on $H$.  There is a natural isomorphism $K(G)/B(L) \cong K(H)$, and this extends to an exact sequence
\[
1 \to K(G)/B(L) \to \Aut(H) \to N_{\GL(d,\F_p)}(L) \to 1.
\]
In particular, $A(H) \cong N_{\GL(d,\F_p)}(L)$.
\end{thm}

\begin{proof}
Let $L \in \mathcal{A}$ and write $H = G/L$.  Then $H/H_n \cong G/G_nL$ is isomorphic to $G/G_n$ if and only if $L \le G_n$.  This shows that $\pi |_{\mathfrak{C}} $ is a well-defined bijection.  As noted in Section~\ref{autgroupsec}, $K(G)$ acts trivially on $G_n \cong G_n/G_{n+1}$, and by Theorem~\ref{isoenum1}, $A(G) \cong \GL(d,\F_p)$.  Thus $A(G) \cong \GL(d, \F_p)$ acts on $\mathcal{C}$, and the $\Aut(G)-$orbits and $\GL(d,\F_p)-$orbits on $\mathcal{C}$ are identical.

Now suppose $L \in \mathcal{C}$.  Then
\[
1 \to K(G) \to N_{\Aut(G)}(L) \to N_{\GL(d,\F_p)}(L) \to 1
\]
is exact.  By the exact sequence in Theorem~\ref{isoenum1}, every automorphism in $K(H)$ is induced by an automorphism in $N_{\Aut(G)}(L)$.  Since $G/G_2 \cong H/H_2$, the automorphism in $N_{\Aut(G)}(L)$ must act trivally on $G/G_2$, that is, it must be in $K(G)$.  Therefore $K(G)$ surjects onto $K(H)$.  The kernel of this map is $B(L)$, so $K(G)/B(L) \cong K(H)$.  The above exact sequence induces the exact sequence
\[
1 \to K(G)/B(L) \to N_{\Aut(G)}(L)/B(L) \to N_{\GL(d,\F_p)}(L) \to 1.
\]
By the second exact sequence in Theorem~\ref{isoenum1}, it follows that
\[
1 \to K(G)/B(L) \to \Aut(H) \to N_{\GL(d,\F_p)}(L) \to 1
\]
is exact.
\end{proof}

We can specialize Theorems~\ref{isoenum1} and~\ref{isoenum2} to prove Theorem~\ref{bijthm}, restated here for convenience.

\begin{thmx}[Theorem~\ref{bijthm}]
Fix a prime $p$ and integers $d, n \ge 2$.  Let $F$ be the free group of rank $d$ and define the following sets:
\begin{eqnarray*}
\mathcal{A}_{d,n} &=& \{\textrm{normal subgroups of $F/F_{n+1}$ lying in $F_2/F_{n+1}$}\} \\
\mathcal{B}_{d,n} &=& \{\textrm{normal subgroups of $F/F_{n+1}$ lying in $F_2/F_{n+1}$} \\
&& \qquad \textrm{and not containing $F_n/F_{n+1}$}\} \\
\mathcal{C}_{d,n} &=& \{\textrm{normal subgroups of $F/F_{n+1}$ lying in $F_n/F_{n+1}$}\} \\
\mathcal{D}_{d,n} &=& \{\textrm{normal subgroups of $F/F_{n+1}$ contained in the} \\
&& \qquad \textrm{regular $\GL(d,\F_p)$-orbits in $\mathfrak{C}_{d,n}$}\} \\
\\
\mathfrak{A}_{d,n} &=& \{\textrm{$\Aut(F/F_{n+1})$-orbits in $\mathcal{A}_{d,n}$}\} \\
\mathfrak{B}_{d,n} &=& \{\textrm{$\Aut(F/F_{n+1})$-orbits in $\mathcal{B}_{d,n}$}\} \\
\mathfrak{C}_{d,n} &=& \{\textrm{$\Aut(F/F_{n+1})$-orbits in $\mathcal{C}_{d,n}$}\} = \{\textrm{$\GL(d,\F_p)$-orbits in $\mathcal{C}_{d,n}$}\} \\
\mathfrak{D}_{d,n} &=& \{\textrm{regular $\GL(d,\F_p)$-orbits in $\mathcal{C}_{d,n}$}\}.
\end{eqnarray*}
Then there is a well-defined map $\pi_{d,n}: \mathfrak{A}_{d,n} \to \{\textrm{finite $p$-groups}\}$ given by $L/F_{n+1} \mapsto F/L$, where $L/F_{n+1} \in \mathcal{A}_{d,n}$.  Furthermore $\pi_{d,n}$ induces bijections
\begin{eqnarray*}
\mathfrak{A}_{d,n} &\leftrightarrow& \{\textrm{$p$-groups $H$ of lower $p$-length at most $n$ with $d(H) = d$}\} \\
\mathfrak{B}_{d,n} &\leftrightarrow& \{\textrm{$p$-groups $H$ of lower $p$-length $n$ with $d(H) = d$}\} \\
\mathfrak{D}_{d,n} &\leftrightarrow& \{\textrm{$p$-groups $H$ in $\pi_{d,n}(\mathfrak{C}_{d,n})$ with $A(H) = 1$}\}.
\end{eqnarray*}
\end{thmx}

\begin{proof}
Take $V$ to be the variety of $p$-groups of lower $p$-length at most $n$.  Applying Theorems~\ref{isoenum1} and~\ref{isoenum2} with $G = F/F_{n+1}$, $\mathcal{A} = \mathcal{A}_{d,n}$, $\mathfrak{A} = \mathfrak{A}_{d,n}$, $\mathcal{C} = \mathcal{C}_{d,n}$, $\mathfrak{C} = \mathfrak{C}_{d,n}$, and $\pi = \pi_{d,n}$ proves all but the statements about $\mathfrak{D}_{d,n}$.  As for those, a subgroup $L/F_{n+1} \in \mathcal{C}_{d,n}$ is in a regular $\GL(d,\F_p)$-orbit if and only if $N_{\GL(d,\F_p)}(L/F_{n+1}) = 1$.  By Theorem~\ref{isoenum2}, this occurs precisely when $A(F/L) = 1$.  Thus the bijection for $\mathfrak{D}_{d,n}$ is proved.
\end{proof}

\chapter{The Lower $p$-Series of a Free Group}
\label{c_free}
\chaptermark{Free Groups}

Let $F$ be the free group of rank $d$ with free generating set $y_1, y_2, \dots, y_d$.  As explained in Chapter~\ref{c_series}, there is an intimate connection between the lower $p$-series of finite $p$-groups and the lower $p$-series of $F$.  As a result, Chapters~\ref{c_normal} and~\ref{c_submodules} rely on a detailed understanding of the quotients $F_n/F_{n+1}$.  In this chapter, we analyze the $\F_p \GL(d, \F_p)$-module structure of $F_n/F_{n+1}$ and power and commutator maps from $F_n/F_{n+1}$ to $F_{n+1}/F_{n+2}$.  Our main tool will be the connection between the lower $p$-series of $F$ and the free Lie algebra described in Theorem~\ref{structurethm}.  The results of Theorem~\ref{structurethm} appear several times in the literature with varying degrees of correctness and detail.  The best references are Bryant and Kov\'{a}cs~\cite{bk} and Huppert and Blackburn~\cite[Chapter VIII]{hb}.  The proof given below seems to be the first time that a complete proof has been written down.

\section{The Free Lie Algebra}

We will say just enough about free Lie algebras for our purposes.  More information about free Lie algebras can be found in Garsia~\cite{gar} and Reutenauer~\cite{reu}.  Our discussion follows Bryant and Kov\'{a}cs~\cite{bk}.

Let $K$ be any field and let $A = \{x_1, \dots, x_d\}$ be an alphabet on $d$ letters.  Write $A^{\ast}$ for the set of all $A$-words and $A^n$ for the set of all $A$-words of length $n$.  Let $K[A^{\ast}]$ denote the free associative $K$-algebra on the generators $x_1, x_2, \dots, x_d$; equivalently, $K[A^{\ast}]$ is the non-commutative algebra of polynomials
\[
f = \sum_{w \in A^{\ast}}{f_w w}
\]
with coefficients $f_w \in K$.  The algebra $K[A^{\ast}]$ is graded by degree; let $K[A^n]$ denote the homogeneous component of degree $n$.  Also, $K[A^{\ast}]$ is a Lie algebra under the Lie bracket $[f,g] = fg-gf$.  Let $K[\Lambda^{\ast}]$ denote the Lie subalgebra of $K[A^{\ast}]$ generated by $x_1, \dots, x_d$ and the Lie bracket.  Then $K[\Lambda^{\ast}]$ is the \emph{free Lie algebra} over $K$ on $x_1, \dots, x_d$.  It is also graded by degree; let $K[\Lambda^n]$ be the homogeneous component of $K[\Lambda^{\ast}]$ of degree $n$.

The group $\GL(d, K)$ acts as the group of $K$-automorphisms on the $K$-vector space $K[A^1]$.  This action extends to the $K$-vector spaces $K[A^n]$ and $K[\Lambda^n]$, and so these may be regarded as $K \GL(d,K)$-modules.

It will be convenient to specify a basis of $K[\Lambda^n]$.  Lexicographically order the set $A^{\ast}$, where $x_1 < x_2 < \cdots < x_d$.  A word $w$ is a \emph{Lyndon word} if it is smaller than all of its proper non-trivial tails.  Let $L$ be the set of Lyndon words, and let $L_n$ be the set of Lyndon words of length $n$.  Inductively define the \emph{right standard bracketing} $b[w]$ of $w \in L$ by
\[
b[w] = w
\]
if $w \in A$ and otherwise by
\[
b[w] = \left[b\left[w_1\right], b\left[w_2\right]\right],
\]
where $w = w_1w_2$ and $w_2$ is the longest proper tail of $w$ that is a Lyndon word.

\begin{thm}[{Reutenauer~\cite[Proof of Theorem 5.1]{reu}}]
\label{basisthm}
If $w \in L$, then
\[
b[w] = w + \sum_{w < v}{f_v v}
\]
for some $f_v \in K$.  The set $\{b[w] : w \in L_n\}$ forms a basis for $K[\Lambda^n]$.
\end{thm}

\section{The Free Lie Algebra and $F_{\textrm{\lowercase{$n$}}}/F_{\textrm{\lowercase{$n+1$}}}$}

The connections between the free Lie algebra and $F_n/F_{n+1}$ given in Theorems~\ref{structurethm} and~\ref{mapsthm} rely on several theorems and lemmas in the literature.  We begin with the connection between $F_n/F_{n+1}$ and the lower central series of $F$.  For each positive integer $n$, let $S_n = \gamma_n(F)/\gamma_n(F)^p \gamma_{n+1}(F)$.  If $s_n \in \gamma_n(F)$, let $\overline{s}_n$ denote the image of $s_n$ in $S_n$.

\begin{thm}[{Huppert and Blackburn~\cite[Chapter VIII, Theorem 1.9(b) and (c)]{hb}}]
\label{hbthm}
For each positive integer $n$, there is a bijection
\begin{eqnarray*}
\sigma_n : S_1 \times S_2 \times \cdots \times S_n &\to& F_n/F_{n+1} \\
(\overline{s}_1, \overline{s}_2, \dots, \overline{s}_n) &\mapsto& s_1^{p^{n-1}} s_2^{p^{n-2}} \cdots s_n F_{n+1}.
\end{eqnarray*}
When $p$ is odd or $p = 2$ and $n = 1$, this map is an isomorphism.  When $p = 2$ and $n \ge 2$, this map restricts to an isomorphism
\[
S_2 \times \cdots \times S_n \to (F_n \cap \gamma_2(F))F_{n+1}/F_{n+1}.
\]
\end{thm}

The next theorem connects the lower central series of $F$ and the free Lie algebra.

\begin{thm}[{Magnus, see Reutenauer~\cite[Corollary 6.16]{reu}}]
\label{mwthm}
For each positive integer $n$, there is a canonical isomorphism
\[
\alpha_n : \gamma_n(F)/\gamma_{n+1}(F) \to \Z[\Lambda^n]
\]
satisfying
\[
\alpha_1 : y_i \gamma_2(F) \mapsto x_i
\]
for $i = 1,\dots,d$ and
\[
\alpha_n : [z_j, z_k] \gamma_{n+1}(F) \mapsto [\alpha_j(z_j \gamma_{j+1}(F)), \alpha_k(z_k \gamma_{k+1}(F))]
\]
for all $z_j \in \gamma_j(F)$ and $z_k \in \gamma_k(F)$ such that $j + k = n$.
\end{thm}

\begin{cor}
\label{mwcor}
For each positive integer $n$, there is a canonical isomorphism
\[
\beta_n : S_n \to \F_p[\Lambda^n]
\]
induced by $\alpha_n$.  Furthermore,
\[
\beta = \beta_1 \times \beta_2 \times \cdots \times \beta_n : S_1 \times S_2 \times \cdots \times S_n \to \F_p[\Lambda^1] \oplus \F_p[\Lambda^2] \oplus \cdots \oplus \F_p[\Lambda^n]
\]
is an isomorphism.
\end{cor}

We also need some results from commutator calculus.

\begin{thm}[{P. Hall, adapted from Leedham-Green and McKay~\cite[Theorem 1.1.30]{lm}}]
\label{collectionthm}
Let $a$ and $b$ be elements of a group $G$.  Then for all positive integers $m$,
\[
(ab)^{p^m} = a^{p^m} b^{p^m} [b,a]^{\binom{p^m}{2}} \prod_{i=3}^{\infty}{\prod_j{c_{i,j}^{e_{i,j}}}}
\]
for some elements $c_{i,j} \in \gamma_i(G)$ and some integers $e_{i,j}$.  Each integer $e_{i,j}$ is a $\Z$-linear combination of $\binom{p^m}{1}, \binom{p^m}{2}, \dots, \binom{p^m}{i}$.
\end{thm}

\begin{cor}
\label{powercor}
Let $a$ and $b$ be elements of a group $G$.  Then for all positive integers $m$,
\[
(ab)^{p^m} \equiv
\left\{
\begin{array}{r@{\quad:\quad}l}
a^{2^m} b^{2^m} [a,b]^{2^{m-1}} \mod{ \gamma_2(G)^{2^m} \gamma_3(G)^{2^{m-1}} \prod_{r=2}^m{\gamma_{2^r}(G)^{2^{m-r}}}} & p = 2 \\
a^{p^m} b^{p^m} \mod{\gamma_2(G)^{p^m} \prod_{r=1}^m{ \gamma_{p^r}(G)^{p^{m-r}} }} & p > 2.
\end{array}
\right.
\]
Furthermore,
\[
(ab)^{p^m} \equiv
\left\{
\begin{array}{r@{\quad:\quad}l}
a^{2^m} b^{2^m} [a,b]^{2^{m-1}} \mod{G_{m+2}} & p = 2 \\
a^{p^m} b^{p^m} \mod{G_{m+2}} & p > 2.
\end{array}
\right.
\]
\end{cor}

\begin{proof}
Given a positive integer $j$, write $j = kp^r$ with $r \ge 0$ and $k$ relatively prime to $p$.  Then  $\binom{p^m}{j}$ is divisible by $p^{m-r}$.  It follows that any $\Z$-linear combination $e$ of $\binom{p^m}{1}, \binom{p^m}{2}, \dots, \binom{p^m}{i}$ is divisible by $m-s$, where $p^s$ is the largest power of $p$ less than or equal to $i$.  This shows that each $c_{i,j}^{e_{i,j}}$ from Theorem~\ref{collectionthm} is in $\gamma_{i}(G)^{p^{m-s}}$.  Then Theorem~\ref{collectionthm} implies that
\[
(ab)^{p^m} \equiv
\left\{
\begin{array}{r@{\quad:\quad}l}
a^{2^m} b^{2^m} [b,a]^{\binom{2^m}{2}} \mod{ \gamma_3(G)^{2^{m-1}} \prod_{r=2}^m{\gamma_{2^r}(G)^{2^{m-r}}}} & p = 2 \\
a^{p^m} b^{p^m} \mod{\gamma_2(G)^{p^m} \prod_{r=1}^m{ \gamma_{p^r}(G)^{p^{m-r}} }} & p > 2.
\end{array}
\right.
\]
Next, we must show that $[b,a]^{\binom{2^m}{2}} \equiv [a,b]^{2^{m-1}} \mod{\gamma_2(G)^{2^m}}$ when $p = 2$.  But
\begin{eqnarray*}
[b,a]^{\binom{2^m}{2}}
&=& [b,a]^{2^{m-1}(2^m-1)} \\
&=& [b,a]^{2^{2m-1}-2^{m-1}} \\
&=& [b,a]^{2^{2m-1}} [a,b]^{2^{m-1}},
\end{eqnarray*}
and $[b,a]^{2^{2m-1}} \in \gamma_2(G)^{2^m}$, proving the claim.  Finally, the congruences modulo $G_{m+2}$ follow directly from the previous congruences and the fact that $\gamma_i(G)^{p^{m-i+2}} \in G_{m+2}$ by Proposition~\ref{fratprops}.
\end{proof}

\begin{cor}
\label{comcor}
Let $a$ and $b$ be elements of a group $G$.  Let $i$ be a positive integer and suppose $a \in \gamma_i(G)$.  Then for all positive integers $m$,
\[
[a^{p^m}, b] \equiv
\left\{
\begin{array}{r@{\quad:\quad}l}
{[a,b]^{2^m} [a,[a,b]]^{2^{m-1}} \mod{G_{m+i+2}}} & p = 2 \\
{[a,b]^{p^m} \mod{G_{m+i+2}}} & p > 2.
\end{array}
\right.
\]
\end{cor}

\begin{proof}
Let $H = \left< a, [a,b] \right>$.  Then $\gamma_2(H)$ is the normal closure of $[a,[a,b]]$ in $H$ by Huppert~\cite[Chapter III, Lemma 1.11]{hup}.  So $\gamma_2(H) \le \gamma{2i+1}(G)$.  Furthermore, $\gamma_j(H) \le \gamma_{ji+1}(G)$ for all $j \ge 2$.  Thus for $p = 2$,
\begin{eqnarray*}
\gamma_2(H)^{2^m} \gamma_3(H)^{2^{m-1}} \prod_{r=2}^m{\gamma_{2^r}(H)^{2^{m-r}}}
&\le& \gamma_{2i+1}(G)^{2^m} \gamma_{3i+1}(G)^{2^{m-1}} \prod_{r=2}^m{\gamma_{i 2^r + 1}(G)^{2^{m-r}}} \\
&\le& G_{m+i+2},
\end{eqnarray*}
and for $p > 2$,
\begin{eqnarray*}
\gamma_2(H)^{p^m} \prod_{r=1}^m{ \gamma_{p^r}(H)^{p^{m-r}} }
&\le& \gamma_{2i+1}(G)^{p^m} \prod_{r=1}^m{\gamma_{i p^r + 1}(G)^{p^{m-r}}} \\
&\le& G_{m+i+2}.
\end{eqnarray*}
By Corollary~\ref{powercor},
\begin{eqnarray*}
[a^{p^m}, b]
&=& a^{-p^m} b^{-1} a^{p^m} b \\
&=& a^{-p^m} (b^{-1} a b)^{p^m} \\
&=& a^{-p^m} (a[a,b])^{p^m} \\
&\equiv& \left\{
\begin{array}{r@{\quad:\quad}l}
{[a,b]^{2^m} [a,[a,b]]^{2^{m-1}} \mod{G_{m+i+2}}} & p = 2 \\
{[a,b]^{p^m} \mod{G_{m+i+2}}} & p > 2.
\end{array}
\right.
\end{eqnarray*}
\end{proof}

These preliminaries and some extra work lead to the following two theorems.

\begin{thm}
\label{structurethm}
Let $n$ be a positive integer.  If $p$ is odd or $p = 2$ and $n = 1$, then there is an $\F_p\GL(d, \F_p)$-module isomorphism
\[
\begin{array}{rccc}
\qemb_n :& F_n/F_{n+1} &\to& \F_p[\Lambda^1] \oplus \cdots \oplus \F_p[\Lambda^n] \\
& s_1^{p^{n-1}} s_2^{p^{n-2}} \cdots s_n F_{n+1} &\mapsto& \beta(\overline{s}_1, \overline{s}_2, \dots, \overline{s}_n),
\end{array}
\]
where $\beta$ is the isomorphism from Corollary $\ref{mwcor}$ and $s_i \in \gamma_i(F)$ for $i = 1, \dots, n$.  If $p = 2$ and $n \ge 2$, then there is an $\F_2\GL(d, \F_2)$-module isomorphism
\[
\begin{array}{rccc}
\qemb_n :& F_n/F_{n+1} &\to& E \oplus \F_2[\Lambda^3] \oplus \cdots \oplus \F_2[\Lambda^n] \\
& s_1^{2^{n-1}} s_2^{2^{n-2}} \cdots s_n F_{n+1} &\mapsto& \beta(\overline{s}_1, \overline{s}_2, \dots, \overline{s}_n) + \beta_1(\overline{s}_1)^2,
\end{array}
\]
where $s_i \in \gamma_i(F)$ for $i = 1, \dots, n$ and $E \subset \F_2[A^1] \oplus \F_2[A^2]$ is an extension of $\F_2[\Lambda^2]$ by $\F_2[\Lambda^1]$.
\end{thm}

\begin{proof}
If $p$ is odd or $p = 2$ and $n = 1$, then
\[
\qemb_n = \beta \circ \sigma_n^{-1},
\]
and hence $\qemb_n$ is an isomorphism by Theorem~\ref{hbthm} and Corollary~\ref{mwcor}.  In all cases, induction on $n$ immediately shows that the action of $\F_p\GL(d,\F_p)$ commutes with $\qemb_n$, so that if $\qemb_n$ is an isomorphism, then it is an $\F_p\GL(d,\F_p)$-module isomorphism.

If $p = 2$ and $n \ge 2$, then $\qemb_n$ is injective since each $\beta_i$ is injective.  Let
\[
E = \mathrm{im}(\qemb_n) \cap (\F_2[A^1] \oplus \F_2[A^2]).
\]
Clearly $\qemb_n$ is surjective.  The map $\beta_1$ is surjective, so $E + \F_2[A^2] = \F_2[A^1] \oplus \F_2[A^2]$.  The map $\beta_2$ is surjective, so $E \cap \F_2[A^2] = \F_2[\Lambda^2]$.  It follows that $E$ is an extension of $\F_2[\Lambda^2]$ by $\F_2[\Lambda^1]$.

It remains to show that $\qemb_n$ is a homomorphism when $p = 2$ and $n \ge 2$.  Let $s,t \in F_n/F_{n+1}$.  Write
\begin{eqnarray*}
s &=& s_1^{2^{n-1}} s_2^{2^{n-2}} \cdots s_n F_{n+1} \textrm{ and } \\
t &=& t_1^{2^{n-1}} t_2^{2^{n-2}} \cdots t_n F_{n+1}
\end{eqnarray*}
with $s_i, t_i \in \gamma_i(F)$ for $i = 1,\dots,n$.  We know
\begin{eqnarray*}
\qemb_n(s) + \qemb_n(t) &=& \beta(\overline{s}_1, \overline{s}_2, \dots, \overline{s}_n) + \beta_1(\overline{s}_1)^2 + \beta(\overline{t}_1, \overline{t}_2, \dots, \overline{t}_n) + \beta_1(\overline{t}_1)^2,
\end{eqnarray*}
and we must show that this equals $\qemb_n(st)$.

Note that $[t_i, s_i]^{2^{n-i-1}} \in F_{n+i-1}$ for all $i$.  From Corollary~\ref{powercor} and the fact that $F_n/F_{n+1}$ is abelian, we find
\begin{eqnarray*}
st &=& s_1^{2^{n-1}} s_2^{2^{n-2}} \cdots s_n t_1^{2^{n-1}} t_2^{2^{n-2}} \cdots t_n F_{n+1} \\
&=& \left( \prod_{i=1}^n{s_i^{2^{n-i}} t_i^{2^{n-i}}} \right) F_{n+1} \\
&=& \left( \prod_{i=1}^n{(s_i t_i)^{2^{n-i}} [t_i, s_i]^{2^{n-i-1}}} \right) F_{n+1} \\
&=& (s_1 t_1)^{2^{n-1}} [t_1,s_1]^{2^{n-2}} (s_2 t_2)^{2^{n-2}} \cdots (s_n t_n) F_{n+1} \\
&=& (s_1 t_1)^{2^{n-1}} ([t_1,s_1] s_2 t_2)^{2^{n-2}} \cdots (s_n t_n) F_{n+1} \\
\qemb_n(st) &=& \beta(\overline{s}_1, \overline{s}_2, \dots, \overline{s}_n) + \beta(\overline{t}_1, \overline{t}_2, \dots, \overline{t}_n) \\
&& + (\beta_1(\overline{s}_1) + \beta_1(\overline{t}_1))^2 + \beta_2([\overline{t}_1, \overline{s}_1]) \\
&=& \beta(\overline{s}_1, \overline{s}_2, \dots, \overline{s}_n) + \beta_1(\overline{s}_1)^2 + \beta(\overline{t}_1, \overline{t}_2, \dots, \overline{t}_n) + \beta_1(\overline{t}_1)^2 \\
&=& \qemb_n(s) + \qemb_n(t).
\end{eqnarray*}
Thus $\qemb_n$ is a homomorphism when $p = 2$ and $n \ge 2$, completing the proof.
\end{proof}

\begin{thm}
\label{mapsthm}
For each positive integer $n$ and $j = 1, \dots, d$, define the following maps:
\[
\begin{array}{rcrcl}
\pow_n &:& F_n/F_{n+1} &\to& F_{n+1}/F_{n+2} \\
&& sF_{n+1} &\mapsto& s^pF_{n+2} \\
&&&& \\
\Fcom_{j,n} &:& F_n/F_{n+1} &\to& F_{n+1}/F_{n+2} \\
&& sF_{n+1} &\mapsto& [s,y_j] F_{n+2} \\
&&&& \\
\com_j &:& \F_p[A^{\ast}] &\to& \F_p[A^{\ast}] \\
&& f &\mapsto& [f, x_j]
\end{array}
\]
Unless $p = 2$ and $n = 1$, the diagram below on the left commutes and $\pow_n$ is an injective homomorphism.  The diagram below on the right commutes and $\Fcom_{j,n}$ is a homomorphism.
\[
\begin{array}{c}
\xymatrix{
F_n/F_{n+1} \ar[dr]_{\pow_n} \ar[rr]^{\qemb_n} &&  \F_p[A^{\ast}] \\
& F_{n+1}/F_{n+2} \ar[ur]_{\qemb_{n+1}}
}
\end{array} \qquad
\begin{array}{c}
\xymatrix{
F_n/F_{n+1} \ar[d]_{\Fcom_{j,n}} \ar[r]^{\qemb_n} & \F_p[A^{\ast}] \ar[d]^{\com_j} \\
F_{n+1}/F_{n+2} \ar[r]^{\qemb_{n+1}} &  \F_p[A^{\ast}]
}
\end{array}
\]
\end{thm}

\begin{proof}
Let $s \in F_n/F_{n+1}$.  Write $s = s_1^{p^{n-1}} s_2^{p^{n-2}} \cdots s_n F_{n+1}$ with $s_i \in \gamma_i(F)$ for $i = 1,\dots,n$.  Of course, $s_i^{p^{n-i}} \in F_n$ for $i = 1,\dots,n$.  Using Corollary~\ref{powercor} with $G = F_n$, if $p > 2$ or $n \ge 2$,
\begin{eqnarray*}
(s_1^{p^{n-1}} s_2^{p^{n-2}} \cdots s_n)^p
&\equiv& s_1^{p^n} s_2^{p^{n-2}} \cdots s_n^p \mod{F_{n+2}}.
\end{eqnarray*}
So $\pow_n(s) = s_1^{p^n} s_2^{p^{n-1}} \cdots s_n^p F_{n+2}$.  It is clear that $\qemb_n(s) = \qemb_{n+1}(\pow_n(s))$.  Thus $\pow_n = \qemb^{-1}_{n+1} \circ \qemb_n$ is an injective homomorphism unless $p = 2$ and $n = 1$.

The map $\com_j$ is an (additive) homomorphism by the linearity of the Lie bracket.  The commutator identity $[ab,c] = [a,c]^b[b,c]$ and the fact that $F_{n+1}/F_{n+2}$ is central in $F/F_{n+2}$ show that
\begin{eqnarray*}
[s_1^{p^{n-1}} s_2^{p^{n-2}} \cdots s_n, y_j]
&=& [s_1^{p^{n-1}}, y_j]^{s_2^{p^{n-2}} \cdots s_n} [s_2^{p^{n-2}}, y_j]^{s_3^{p^{n-3}} \cdots s_n} \cdots [s_n, y_j] \\
&\equiv& [s_1^{p^{n-1}}, y_j] [s_2^{p^{n-2}}, y_j] \cdots [s_n, y_j] \mod{F_{n+2}}.
\end{eqnarray*}
If $p > 2$, then Corollary~\ref{comcor} shows that
\begin{eqnarray*}
[s_1^{p^{n-1}} s_2^{p^{n-2}} \cdots s_n, y_j]
&\equiv& [s_1, y_j]^{p^{n-1}} [s_2, y_j]^{p^{n-2}} \cdots [s_n, y_j] \mod{F_{n+2}},
\end{eqnarray*}
and clearly $\qemb_{n+1}(\Fcom_j(s)) = \com_j(\qemb_n(s)$.

If $p = 2$, note that $[s_i, [s_i, y_j]]^{2^{n-i-1}} \in F_{n+i}$ for all $i$.  Thus by Corollaries~\ref{powercor} and~\ref{comcor},
\begin{eqnarray*}
[s_1^{2^{n-1}} s_2^{2^{n-2}} \cdots s_n, y_j]
&\equiv& [s_1, y_j]^{2^{n-1}} [s_1, [s_1, y_j]]^{2^{n-2}} [s_2, y_j]^{2^{n-2}} \cdots [s_n, y_j] \hspace{-.05in} \mod{F_{n+2}} \\
&\equiv& [s_1, y_j]^{2^{n-1}} ([s_1, [s_1, y_j]] [s_2, y_j])^{2^{n-2}} \cdots [s_n, y_j] \mod{F_{n+2}} \\
\qemb_{n+1}(\Fcom_j(s)) &=& [\beta(\overline{s}_1, \overline{s}_2, \dots, \overline{s}_n), x_j] + [\beta_1(\overline{s}_1), [\beta_1(\overline{s}_1), x_j]] \\
&=& [\beta(\overline{s}_1, \overline{s}_2, \dots, \overline{s}_n), x_j] + [\beta_1(\overline{s}_1)^2, x_j] \\
&=& \com_j(\qemb_n(s)).
\end{eqnarray*}
Thus in either case, $\Fcom_j = \qemb_{n+1}^{-1} \circ \com_j \circ \qemb_n$ is a homomorphism.
\end{proof}

We conclude this section with two corollaries of Theorem~\ref{structurethm}.  First, the dimension of $K[\Lambda^i]$ is given by Witt's formula:
\[
\dim(K[\Lambda^i]) = \frac{1}{i} \sum_{j | i}{\mu(i/j) \cdot d^j},
\]
where $\mu$ is the M\"{o}bius function (see Reutenauer~\cite[Appendix 0.4.2]{reu}).  Thus Theorem~\ref{structurethm} tells us the dimension of $F_n/F_{n+1}$.

\begin{cor}
\label{dimcor}
Let $n$ be a positive integer.  The dimension of $F_n/F_{n+1}$ is
\[
d_n = \sum_{i=1}^n{\frac{1}{i} \sum_{j | i}{\mu(i/j) \cdot d^j}}.
\]
\end{cor}

We also need to know some numerical bounds on $d_n$, but these are computed in Lemmas~\ref{dnasymptotics1} and~\ref{dnasymptotics2}.  For the second corollary, let $V$ be the natural $\F_p \GL(d,\F_p)$-module.  Then $\F_p[\Lambda^1] \cong V$ and $\F_p[\Lambda^2] \cong V \wedge V$ as $\F_p \GL(d,\F_p)$-modules.  Therefore Theorem~\ref{structurethm} tells us the following fact.

\begin{cor}
\label{wedgecor}
Let $n \ge 2$.  Then $F_n/F_{n+1}$ contains a $\F_p \GL(d,\F_p)$-submodule isomorphic to an extension of $V \wedge V$ by $V$, where $V$ is the natural $\F_p \GL(d,\F_p)$-module.
\end{cor}

This will be needed to apply Theorem~\ref{strongerbound} to groups of lower $p$-length 2.

\section{The Expansion of Subgroups of $F/F_{\textrm{\lowercase{$n+1$}}}$}

The remainder of this chapter is devoted to proving the following theorem and corollary.  Corollary~\ref{expcor} will be combined with Theorem~\ref{normalthm} to prove Corollary~\ref{fnormalcor}, which gives upper bounds for the number of normal subgroups of $F/F_{n+1}$ with certain properties.

\begin{thm}
\label{expthm}
Fix a prime $p$ and integers $d \ge 3$ and $i \ge 2$.  Suppose that $U$ is a normal subgroup of $F$ lying in $F_2$.  Let
\begin{eqnarray*}
Q &=& (U \cap F_i)F_{i+1}/F_{i+1} \\
R &=& (U_2 \cap F_{i+1})F_{i+2}/F_{i+2} \\
S &=& (U^p[U,F] \cap F_{i+1})F_{i+2}/F_{i+2}.
\end{eqnarray*}
Viewing $Q$, $R$, and $S$ as $\F_p$-vector spaces, their dimensions satisfy $\dim(R) \ge \dim(Q)$ and $\dim(S) \ge (3/2) \; \dim(Q)$.
\end{thm}

The third isomorphism theorem lets us replace $F$ by $F/F_n$, giving the following corollary.

\begin{cor}
\label{expcor}
Fix a prime $p$ and integers $d \ge 3$, $n \ge 3$, and $2 \le i < n$.  Let $G = F/F_{n+1}$.  Suppose that $U$ is a normal subgroup of $G$ lying in $G_2$.  Let
\begin{eqnarray*}
Q &=& (U \cap G_i)G_{i+1}/G_{i+1} \\
R &=& (U_2 \cap G_{i+1})G_{i+2}/G_{i+2} \\
S &=& (U^p[U,G] \cap G_{i+1})G_{i+2}/G_{i+2}.
\end{eqnarray*}
Then $\dim(R) \ge \dim(Q)$ and $\dim(S) \ge (3/2) \; \dim(Q)$.
\end{cor}

To prove Theorem~\ref{expthm}, we will build up to an analogous result for the free Lie algebra (Lemma~\ref{dimlem}) and then apply Theorem~\ref{structurethm}.  Informally, the result for the free Lie algebra says that if we start with a finite-dimensional subspace $W$ of $\F_p[A^{\ast}]$ and add to it the subspace generated by $\{[W,x_i] \; : \; i = 1, \dots, d\}$, we get a new subspace whose dimension is at least $(3/2) \dim(W)$.  It seems reasonable to describe this as investigating the ``expansion of a subspace when taking commutators'', hence the title of this section.

We need to define three more maps:
\[
\begin{array}{rcrcl}
\com &:& \{\textrm{subspaces of $\F_p[A^{\ast}]$}\} &\to& \{\textrm{subspaces of $\F_p[A^{\ast}]$}\} \\
&& W &\mapsto& [W, \F_p[\Lambda^1]]
\end{array}
\]
\[
\begin{array}{rcrcl}
\com_{j,n} &:& \F_p[A^n] &\to& \F_p[A^{n+1}] \\
&& \com_{j,n} &=& \com_j|_{\F_p[A^n]} \\
&&&& \\
\proj_n &:& \F_p[A^{\ast}] &\to& \F_p[A^n] \\
&&&& \textrm{the projection map onto $\F_p[A^n]$}
\end{array}
\]

\begin{lem}
\label{injlem}
The following diagram commutes:
\[
\begin{array}{c}
\xymatrix{
\F_p[A^{\ast}] \ar[d]_{\proj_n} \ar[r]^{\com_j} & \F_p[A^{\ast}] \ar[d]^{\proj_{n+1}} \\
\F_p[A^n] \ar[r]_{\com_{j,n}} & \F_p[A^{n+1}]
}
\end{array}
\]
If $n = 1$, then the kernel of $\com_{j,n}$ is spanned by $x_j$.  If $n > 1$, then $\com_{j,n}$ is injective.
\end{lem}

\begin{proof}
The only statements requiring proof are those about the kernel and injectivity of $\com_{j,n}$.  Without loss of generality, we may assume that $j = 1$.  Suppose that $w \in L_n$.  Unless $n = 1$ and $w = x_1$, we see that $x_1 w$ is smaller than $w$, and hence smaller than all of its proper non-trivial tails.  So $x_1 w \in L_{n+1}$.  Furthermore, $w$ is the longest tail of $x_1 w$ that is a Lyndon word, so $b[x_1 w] = -[b[w], x_1]$.  Thus the image of $b[w]$ under $\com_{1,n}$ is the negative of a basis element in $L_{n+1}$, unique for each $w$.  It follows that the kernel of $\com_{1,1}$ is spanned by $x_1$ and $\com_{1,n}$ is injective for $n > 1$.
\end{proof}

\begin{lem}\label{explem}
Fix $d \ge 3$ and $n \ge 2$.  Suppose that $W$ is a subspace of $\F_p[\Lambda^n]$.  Then $\dim(\com(W)) \ge (3/2)\dim(W)$.
\end{lem}

\begin{proof}
Let $\F_p[\Lambda^{\ast}]_{ij}$ denote the free Lie algebra on two generators $x_i$ and $x_j$; there is a natural embedding of $\F_p[\Lambda^{\ast}]_{ij}$ into $\F_p[\Lambda^{\ast}]$.  Let $\F_p[\Lambda^n]_{ij}$ be the homogeneous component of degree $n$ in $\F_p[\Lambda^{\ast}]_{ij}$.

First, we claim that if $f$ and $g$ are distinct elements of $\F_p[\Lambda^n]$ and $[f,x_i] = [g,x_j]$, then in fact $f, g \in \F_p[\Lambda^n]_{ij}$.  We may assume that $i,j > 1$.  Suppose that $f \notin \F_p[\Lambda^n]_{ij}$.  Then writing
\[
f = \sum_{w \in L_n}{f_w b[w]},
\]
there must be some word $w \in L_n$ where $f_w \neq 0$ and $w$ contains a letter other than $x_i$ and $x_j$.  We may assume that $w$ contains the letter $x_1$, and since $w \in L_n$, it must be that $w$ starts with $x_1$.  In that case, by Theorem~\ref{basisthm}, there is a word beginning with $x_1$ that appears in $f$ with non-zero coefficient.  Thus there is a word beginning with $x_1$ and ending with $x_i$ that appears in $[f, x_i]$ with non-zero coefficient.  No such word can appear in $[g,x_j]$, contradicting the fact that $[f,x_i] = [g,x_j]$.  Hence $f \in \F_p[\Lambda^n]_{ij}$ and similarly $g \in \F_p[\Lambda^n]_{ij}$.

Note that $\F_p[\Lambda^n]_{ij} \cap \F_p[\Lambda^n]_{kl} = 0$ if $\{i,j\} \neq \{k,l\}$ (the letters $x_i$ and $x_j$ appear in every element of $\F_p[\Lambda^n]_{ij}$ since $n > 1$).  Choose $i$ and $j$ so that $\dim(W \cap \F_p[\Lambda^n]_{ij})$ is as small as possible; in particular this intersection has dimension at most $(1/2) \dim(W)$.  Let $X$ be a complement to $W \cap \F_p[\Lambda^n]_{ij}$ in $W$.

We can define a more restrictive commutator map on subspaces by $\com_{ij}: \bullet \mapsto [\bullet, \F_p[\Lambda^1]_{ij}]$.  Obviously $\com_{ij}(W) \subseteq \com(W)$.  Using Lemma~\ref{injlem} and the above claim,
\begin{eqnarray*}
\dim(\com_{ij}(W)) &=& \dim(\com_{ij}(W \cap \F_p[\Lambda^n]_{ij})) + \dim(\com_{ij}(X)) \\
&\ge& \dim(W \cap \F_p[\Lambda^n]_{ij}) + 2 \dim(X) \\
&\ge& (3/2) \dim{W}.
\end{eqnarray*}
\end{proof}

\begin{lem}\label{onelem}
Fix $d \ge 2$.  Suppose that $W$ is a subspace of $\F_p[\Lambda^1]$.  Then
$\dim(W + \com(W)) \ge (3/2)\dim(W)$.
\end{lem}

\begin{proof}
Recalling Lemma~\ref{injlem}, this is clear if $\dim(W) = 1$, and otherwise
\[
\dim(\com_{1,1}(W)) \ge \dim(W)-1,
\]
implying the result since $W$ and $\com(W)$ are disjoint.
\end{proof}

\begin{lem}\label{elem}
Let $p = 2$.  Suppose that $W$ is a subspace of $E$, where $E$ is defined in Theorem~\ref{structurethm}.  Then
$\dim(W + \com(W)) \ge (3/2)\dim(W)$.
\end{lem}

\begin{proof}
Let $X = W \cap \F_2[\Lambda^2]$ and let $Y$ be a complement to $X$ in $W$.  Note that $\dim(Y) = \dim(\proj_1(Y))$.  By Lemma~\ref{onelem},
\[
\dim(\proj_1(Y) + \com(\proj_1(Y))) \ge (3/2)\dim(\proj_1(Y)).
\]
By the commutative diagram in Lemma~\ref{injlem}, it follows that $Y + \com(Y)$ contains a subspace of dimension at least $(3/2)\dim(Y)$ that has trivial intersection with $\F_2[\Lambda^3]$.  By Lemma~\ref{explem}, $\com(X) \le \F_2[\Lambda^3]$ contains a subspace of dimension at least $(3/2) \dim(X)$.
Then
\begin{eqnarray*}
\dim(W + \com(W)) \ge (3/2)\dim(X) + (3/2)\dim(Y) = (3/2)\dim(W).
\end{eqnarray*}
\end{proof}

\begin{lem}\label{dimlem}
Fix $d \ge 3$.  Let
\[
U_n = \left\{
\begin{array}{r@{\quad:\quad}l}
\F_p[\Lambda^1] \oplus \cdots \oplus \F_p[\Lambda^n] & \textrm{$p$ is odd or $n = 1$} \\
E \oplus \F_2[\Lambda^3] \oplus \cdots \oplus \F_2[\Lambda^n] & \textrm{$p = 2$ and $n \ge 2$}.
\end{array} \right.
\]
Suppose that $W$ is a subspace of $\F_p[A^{\ast}]$ contained in $U_n$.  Then $\dim(W + \com(W)) \ge (3/2) \dim(W)$.
\end{lem}

\begin{proof}
The proof will be by induction on $n$.  When $p$ is odd and $n = 1$, Lemma~\ref{onelem} gives the result.  When $p = 2$ and $n = 2$, Lemma~\ref{elem} gives the result.  So assume that $p$ is odd and $n > 1$ or that $p = 2$ and $n > 2$.  Assume the result holds for $n-1$.  Let $X = W \cap U_{n-1}$.  By the inductive hypothesis,
\[
\dim(X + \com(X)) \ge (3/2)\dim(X).
\]
Furthermore, $X + \com(X) \le U_n$.  Let $Y$ be a complement to $X$ in $W$.  By the commutative diagram in Lemma~\ref{injlem}, $\com(\proj_n(Y)) = \proj_{n+1}(\com(Y))$.  By the definition of $X$ and $Y$, $\dim(\proj_n(Y)) = \dim(Y)$.  By Lemma~\ref{explem},
\[
\dim(\proj_{n+1}(\com(Y))) \ge (3/2) \dim(\proj_n(Y)).
\]
Thus $\com(Y)$ contains a subspace of dimension at least $(3/2) \dim(\proj_n(Y))$ that has trivial intersection with $U_n$.  Therefore
\[
\dim(W + \com(W)) \ge (3/2) \dim(X) + (3/2) \dim(Y) = (3/2) \dim(W).
\]
\end{proof}

\begin{proof}[Proof of Theorem~\ref{expthm}]
Replacing $U$ by $(U \cap F_n)F_{n+1}$ does not change $Q$, $R$, or $S$, so we may assume that $F_{n+1} \le U \le F_n$.  Recall that by Corollary~\ref{mapsthm}, $\pow_n$ is injective.   Since $\pow_n(Q) = R$, it follows that $\dim(R) \ge \dim(Q)$.

Also by Corollary~\ref{mapsthm},
\[
S = \qemb_{n+1}^{-1}(\qemb_n(U) + (\com \circ \qemb_n)(U)).
\]
Since $\qemb_n$ is injective, and
\[
\dim(\qemb_n(U) + (\com \circ \qemb_n)(U)) \ge (3/2) \dim(\qemb_n(U))
\]
by Lemma~\ref{dimlem}, it follows that $\dim(S) \ge (3/2) \; \dim(Q)$.
\end{proof}

It is reasonable to wonder if the factor of $3/2$ in Lemma~\ref{dimlem} (and hence in Theorem~\ref{expthm}) is the best possible.  It almost certainly is not; intuitively a factor of about $d$ seems right, but this appears to be much harder to prove and is an interesting question in its own right.  Fortunately, $3/2$ suffices for our purposes.

\chapter{Counting Normal Subgroups of Finite $p$-Groups}
\label{c_normal}
\chaptermark{Counting Normal Subgroups}

The goal of this chapter is to prove Theorem~\ref{limit1thm}.  Recall from Chapter~\ref{c_main} that if
\[
\begin{array}{rcl}
\mathcal{A}_{d,n} &=& \{\textrm{normal subgroups of $F/F_{n+1}$ lying in $F_2/F_{n+1}$}\}, \\
\mathfrak{A}_{d,n} &=& \{\textrm{$\Aut(F/F_{n+1})$-orbits in $\mathcal{A}_{d,n}$}\}, \\
\mathcal{C}_{d,n} &=& \{\textrm{normal subgroups of $F/F_{n+1}$ lying in $F_n/F_{n+1}$}\}, \textrm{ and } \\
\mathfrak{C}_{d,n} &=& \{\textrm{$\Aut(F/F_{n+1})$-orbits in $\mathcal{C}_{d,n}$}\} = \{\textrm{$\GL(d,\F_p)$-orbits in $\mathcal{C}_{d,n}$}\},
\end{array}
\]
then Theorem~\ref{limit1thm} essentially shows that the number of $\GL(d, \F_p)$-orbits in $\mathfrak{C}_{d,n}$ is large relative to the number of orbits in $\mathfrak{A}_{d,n}$.  This is nominally a result about counting orbits, but we can actually show that the number of subgroups in $\mathcal{C}_{d,n}$ is so large that even if every orbit in $\mathfrak{C}_{d,n}$ was regular (that is, if $\mathfrak{C}_{d,n}$ was as small as possible) and if every orbit in $\mathfrak{A}_{d,n} \setminus \mathfrak{C}_{d,n}$ was trivial (that is, if $\mathfrak{A}_{d,n} \setminus \mathfrak{C}_{d,n}$ was as big as possible), Theorem~\ref{limit1thm} would still hold.

\section{The Number of Normal Subgroups of a Finite \lowercase{$p$}-Group}
\sectionmark{An Upper Bound}

Given a finite $p$-group $G$ of lower $p$-length $n$ and non-negative integers $u_1, \dots, u_n$, we can form a set $S(G, u_1, \dots, u_n)$ of normal subgroups of $G$ by
\[
S(G, u_1, \dots, u_n) = \{U \lhd G \; : \; \dim((U \cap G_i)G_{i+1}/G_{i+1}) = u_i \textrm{ for } i=1,\dots,d \}.
\]
Our goal in Theorem~\ref{normalthm} is to give an upper bound on the size of $S(G, u_1, \dots, u_n)$.  Specializing this bound to $F/F_{n+1}$ and summing over certain choices of $u_1, \dots, u_n$ will give us an upper bound on the size of $\mathcal{A}_{d,n} \setminus \mathcal{C}_{d,n}$, allowing us to prove Theorem~\ref{limit1thm}.  The bound in Theorem~\ref{normalthm} depends on certain parameters of the group which are difficult to work out in general, but were calculated for $F/F_{n+1}$ in Corollary~\ref{expcor}.

An alternate way to view $S(G, u_1, \dots, u_n)$ is to note that
\[
(U \cap G_i)G_{i+1}/G_{i+1} \cong (U \cap G_i)/(U \cap G_{i+1})
\]
by the Second Group Isomorphism Theorem.  Then
\[
U \cap G_1 \ge U \cap G_2 \ge \cdots \ge U \cap G_{n+1}
\]
is a central series of $U$ with elementary abelian quotients of order $p^{u_1}, p^{u_2}, \dots, p^{u_n}$ (in that order).  The set $S(G, u_1, \dots, u_n)$ consists of all normal subgroups $U$ of $G$ whose associated series has specified quotients.

Furthermore, each subgroup $U \in S(G, u_1, \dots, u_n)$ determines the following data, which we will denote collectively by $\Theta(U)$:
\begin{enumerate}
\item A subgroup $J$ of $G_n$, given by $U \cap G_n$;
\item A normal subgroup $K$ of $H$, given by $UG_n/G_n$; and
\item A complement to $G_n/J$ in $UG_n/J$, given by $U/J$.
\end{enumerate}
In fact, the data $\Theta(U)$ uniquely determines $U \in S(G, u_1, \dots, u_n)$.  Given $\Theta(U)$, the subgroup $UG_n$ is uniquely determined as the inverse image of $K$ in $G$.  Then the complement to $G_n/J$ in $UG_n/J$ given by $\Theta(U)$ is $V/J$ for a unique normal subgroup $V$ of $G$.  Thus $U = V$ and $\Theta(U)$ uniquely determines $U$.

There does not seem to be any prior literature on the number of normal subgroups of an arbitrary finite $p$-group.  Birkhoff~\cite{bir} gave an exact formula for the number of subgroups of a finite abelian $p$-group, but Theorem~\ref{normalthm} is apparently unrelated to that result.  It is unclear how good the upper bound in Theorem~\ref{normalthm} is in general; it is simply sufficient for our needs.

\begin{thm}
\label{normalthm}
Suppose $G$ is a finite $p$-group with lower $p$-length $n$.  Let $g_i = \dim{(G_i/G_{i+1})}$ for all $i = 1, \dots, n$.  Suppose $u_1, u_2, \dots, u_n$ are integers satisfying $0 \le u_i \le g_i$ for all $i = 1, \dots, n$, and define $S(G, u_1, \dots, u_n)$ as above.  Suppose that for each $U \in S(G, u_1, \dots, u_n)$ and $1 \le i \le n$,
\[
\dim((U_2 \cap G_i)G_{i+1}/G_{i+1}) \ge v_i
\]
and
\[
\dim((U^p[U, G] \cap G_i)G_{i+1}/G_{i+1}) \ge w_i.
\]
Then
\[
|S(G, u_1, \dots, u_n)| \le \sbinom{g_1}{u_1}_p \prod_{i=2}^n{\sbinom{g_i-w_i}{u_i-w_i}_p p^{(g_i-u_i)(u_1 + \cdots + u_{i-1}- v_1 - \cdots - v_{i-1})}}.
\]
\end{thm}

\begin{proof}
The proof will be by induction on $n$.  If $n = 1$, then $G$ is elementary abelian of dimension $g_1$.  In this case, $|S(G, u_1)|$ counts the number of subgroups of $G$ of dimension $u_1$, and this number is $\sbinom{g_1}{u_1}_p$.

For our inductive hypothesis, suppose that $n \ge 2$ and that the result holds in $H = G/G_n$, a $p$-group of lower $p$-length $n-1$.  For $1 \le i \le n-1$, it is clear that $H_i = G_i/G_n$ and $g_i = \dim(H_i/H_{i+1})$.  Thus
\[
|S(H, u_1, \dots, u_{n-1})| \le \sbinom{g_1}{u_1}_p \prod_{i=2}^{n-1}{\sbinom{g_i-w_i}{u_i-w_i}_p p^{(g_i-u_i)(u_1 + \cdots + u_{i-1}- v_1 - \cdots - v_{i-1})}}.
\]
Since there is a bijective correspondence between subgroups $U \in S(G, u_1, \dots, u_n)$ and data $\Theta(U)$, we can give an upper bound for $|S(G, u_1, \dots, u_n)|$ by giving an upper bound for the number of possibilities for $\Theta(U)$.  First, $J = U \cap G_n$ is a subspace of $G_n$ of dimension $u_n$.  Furthermore, $J$ must contain $U^p [U,G] \cap G_n$, which by assumption has dimension at least $w_n$.  Thus the number of choices for $J$ is at most $\sbinom{g_n-w_n}{u_n-w_n}_p$.

Next, we can show that $K = UG_n/G_n \in S(H, u_1, \dots, u_{n-1})$.  Namely, for each $i = 1, \dots, n-1$,
\begin{eqnarray}
(K \cap H_i)H_{i+1}/H_{i+1} &=& (UG_n/G_n \cap G_i/G_n)(G_{i+1}/G_n)/(G_{i+1}/G_n) \nonumber \\
&\cong& (UG_n \cap G_i)G_{i+1}/G_{i+1} \label{congeqn} \\
&\cong& (U \cap G_i)G_{i+1}/G_{i+1}, \nonumber
\end{eqnarray}
and so $\dim{(K \cap H_i)H_{i+1}/H_{i+1}} = u_i$.  It follows that $K \in S(H, u_1, \dots, u_{n-1})$.  So there are at most $|S(H, u_1, \dots, u_{n-1})|$ choices for $K$.

Finally, we must bound the number of complements $U/J$ to $G_n/J$ in $UG_n/J$.  Note that $G_n/J$ is central in $UG_n/J$ since $G_n$ is central in $G$.  It follows from (say) Lubotzky and Segal~\cite[Lemma 1.3.1]{ls} that the number of complements to $G_n/J$ in $UG_n/J$ is
\begin{eqnarray*}
|\Hom((UG_n/J)/(G_n/J), G_n/J)|
&=&
|\Hom(UG_n/G_n, G_n/J)| \\
&=& |\Hom(K, G_n/J)| \\
&=& |\Hom(K/K_2, G_n/J)|.
\end{eqnarray*}
The dimension of $G_n/J$ is $h_n - u_n$.  Also,
\begin{eqnarray*}
\dim(K/K_2)
&=& \dim(K) - \dim(K_2) \\
&=& \sum_{i=1}^{n-1}{\dim((K \cap H_i)H_{i+1}/H_{i+1})} - \sum_{i = 1}^{n-1}{\dim((K_2 \cap H_i)H_{i+1}/H_{i+1})}.
\end{eqnarray*}
Note that $K_2 = U_2G_n/G_n$, and a similar calculation to Equation~\ref{congeqn} shows that
\[
(K_2 \cap H_i)H_{i+1}/H_{i+1} \cong (U_2 \cap G_i)G_{i+1}/G_{i+1},
\]
which by hypothesis has dimension at least $v_i$.  Thus
\[
\dim(K/K_2) \le u_1 + \cdots + u_{n-1} - (v_1 + \cdots + v_{n-1})
\]
and
\[
|\Hom(K/K_2, G_n/J)| \le p^{(g_n - u_n)(u_1 + \cdots + u_{n-1} - v_1 - \cdots - v_{n-1})}.
\]
Using the inductive hypothesis gives
\begin{eqnarray*}
|S(G, u_1, \dots, u_n)| &\le& |S(H, u_1, \dots, u_{n-1})| \\
&& \qquad \qquad \cdot \sbinom{g_n - w_n}{u_n - w_n}_p \cdot p^{(g_n - u_n)(u_1 + \cdots + u_{n-1} - v_1 - \cdots - v_{n-1})} \\
&\le& \sbinom{g_1}{u_1}_p \prod_{i=2}^n{\sbinom{g_i-w_i}{u_i-w_i}_p p^{(g_i-u_i)(u_1 + \cdots + u_{i-1}- v_1 - \cdots - v_{i-1})}}.
\end{eqnarray*}
\end{proof}

\begin{cor}
\label{fnormalcor}
Fix $d \ge 3$ and $n \ge 3$.  Let $F$ be the free group of rank $d$ and let $d_i = \dim(F_i/F_{i+1})$ for each $i \ge 1$.  Then
\[
|S(F/F_{n+1}, 0, u_2, \dots, u_n)|
\le D(p)^{n-1} \prod_{i=2}^n{p^{(u_i-u_{i-1}/2)(d_i-u_i)}},
\]
where
\[
D(p) = \prod_{j=1}^{\infty}{\frac{1}{1-p^{-j}}}.
\]
\end{cor}

\begin{proof}
Letting $G = F/F_{n+1}$ in Theorem~\ref{normalthm}, it is clear that $g_i = d_i$.  Since $u_1 = 0$, each $U \in S(F/F_{n+1}, 0, u_2, \dots, u_n)$ is contained in $G_2$, and so $G_3$ contains $U_2$ and $U^p[U,G]$.  Thus we can choose $v_1 = v_2 = w_1 = w_2 = 0$.  By Corollary~\ref{expcor}, for $3 \le i \le n$, we can choose $v_i = u_{i-1}$ and $w_i = \lceil (3/2) u_{i-1} \rceil$.
Finally, Lemma~\ref{qests} Equation~\ref{coefbounds} gives an upper bound for the Gaussian coefficient $\sbinom{g_i-w_i}{u_i-w_i}_p$.  The formula from Theorem~\ref{normalthm} becomes
\begin{eqnarray*}
|S(F/F_{n+1}, 0, u_2, \dots, u_n)|
&\le& D(p)^{n-1} \prod_{i=2}^n{p^{(u_i-w_i)(d_i-u_i) + (d_i-u_i) u_{i-1}}} \\
&\le& D(p)^{n-1} \prod_{i=2}^n{p^{(u_i-(3/2)u_{i-1})(d_i-u_i) + (d_i-u_i) u_{i-1}}} \\
&\le& D(p)^{n-1} \prod_{i=2}^n{p^{(u_i-u_{i-1}/2)(d_i-u_i)}}.
\end{eqnarray*}
\end{proof}

\section{A Proof of Theorem~\ref{limit1thm}}

We can now prove Theorem~\ref{limit1thm} and Corollary~\ref{limit1cor}, restated here for convenience.

\begin{thmx}[Theorem~\ref{limit1thm}]
Fix a prime $p$ and integers $d$ and $n$ so that either $n \ge 3$ and $d \ge 6$ or $n \ge 10$ and $d \ge 5$.  Let $F$ be the free group of rank $d$ and let $d_i$ be the dimension of $F_i/F_{i+1}$ for $i = 1, \dots, n$.  Then
\[
1 \le \frac{|\mathfrak{A}_{d,n}|}{|\mathfrak{C}_{d,n}|} \le 1 + C(p^{15/16}) C(p)^{n-2} D(p)^{n-2} p^{d_{n-1} - d_n/4 + d^2 - 11/16}.
\]
\end{thmx}

\begin{proof}
Note that a normal subgroup $U$ of $F/F_{n+1}$ lies in $F_2/F_{n+1}$ if and only if $U \in S(F/F_{n+1}, 0, u_2, \dots, u_n)$ for some integers $u_2, \dots, u_n$.  Also, $U$ lies in $F_n/F_{n+1}$ if and only if $u_n = \cdots = u_{n-1} = 0$.  If $U$ does not lie in $F_n/F_{n+1}$, then $1 \le u_{n-1} < w_n \le u_n$, so $u_n \ge 2$.  Thus, using Corollary~\ref{fnormalcor}, we find
\[
\mathcal{A}_{d,n} = \mathcal{C}_{d,n} \cup \bigcup_{u_2, \dots, u_n}{S(F/F_{n+1}, 0, u_2, \dots, u_n)}
\]
and
\[
|\mathcal{A}_{d,n}| \le |\mathcal{C}_{d, n}| + \sum_{u_2, \dots, u_n}{D(p)^{n-1} \prod_{i=2}^n{p^{(d_i-u_{i-1}/2)(d_i-u_i)}}},
\]
where the sums are over
\[
\begin{array}{rcl}
0 \le &u_i& \le d_j \textrm{ for } i = 2, \dots, n-2, \\
1 \le &u_{n-1}& \le d_{n-1}, \textrm{ and} \\
2 \le &u_n& \le d_n.
\end{array}
\]
The above sum is precisely the quantity $D(p)^{n-1} A_1(0)$ from the statement of Lemma~\ref{gaussprods}.  Hence
\[
|\mathcal{A}_{d,n}| \le |\mathcal{C}_{d, n}| +  C(p^{15/16}) C(p)^{n-2} D(p)^{n-1} p^{d_n^2/4 - 15/16 - d_n/4 + d_{n-1}}.
\]
Since $|\mathcal{C}_{d, n}|$ is the number of subspaces of a $d_n$-dimensional $\F_p$-vector space, denoted $\mathcal{G}_{d_n}(p)$ in Appendix~\ref{a_estimates}, it follows from Lemma~\ref{qests} and the fact that $2 - 9p^{(1-d_n)/2}/2 > 1$ that
\begin{eqnarray*}
|\mathcal{A}_{d,n}|/|\mathcal{C}_{d,n}| &\le& 1 + C(p^{15/16}) C(p)^{n-2} D(p)^{n-1} p^{d_n^2-15/16-d_n/4+d_{n-1}} / \mathcal{G}_{d_n}(p) \\
&\le& 1 + C(p^{15/16}) C(p)^{n-2} D(p)^{n-2} p^{d_{n-1} - d_n/4 - 11/16}.
\end{eqnarray*}
Now $\mathfrak{A}_{d,n} \setminus \mathfrak{C}_{d,n}$ are the $\Aut(F/F_{n+1})$-orbits on $\mathcal{A}_{d, n} \setminus \mathcal{C}_{d,n}$, so
\[
0 \le |\mathfrak{A}_{d,n}| - |\mathfrak{C}_{d,n}| \le |\mathcal{A}_{d,n}| - |\mathcal{C}_{d,n}|.
\]
Also $|\mathcal{C}_{d,n}| \le |\mathfrak{C}_{d,n}| \cdot |\GL(d,\F_p)|$, since $\mathcal{C}_{d,n}$ falls into $|\mathfrak{C}_{d,n}|$ orbits, each of size at most $|\GL(d,\F_p)|$.  Then
\begin{eqnarray*}
0 &\le& \frac{|\mathfrak{A}_{d,n}|}{|\mathfrak{C}_{d,n}|} - 1 \\
&=& \frac{|\mathcal{C}_{d,n}|}{|\mathfrak{C}_{d,n}|} \left( \frac{|\mathfrak{A}_{d,n}| - |\mathfrak{C}_{d,n}|}{|\mathcal{C}_{d,n}|} \right) \\
&\le& |\GL(d,\F_p)| \left( \frac{|\mathcal{A}_{d,n}| - |\mathcal{C}_{d,n}|}{|\mathcal{C}_{d,n}|} \right) \\
&\le& C(p^{15/16}) C(p)^{n-2} D(p)^{n-2} p^{d_{n-1} - d_n/4 + d^2 - 11/16}.
\end{eqnarray*}
Therefore
\[
1 \le \frac{|\mathfrak{A}_{d,n}|}{|\mathfrak{C}_{d,n}|} \le 1 + C(p^{15/16}) C(p)^{n-2} D(p)^{n-2} p^{d_{n-1} - d_n/4 + d^2 - 11/16}.
\]
\end{proof}

\begin{cor}
If $n \ge 2$, then
\[
\lim_{d \to \infty}{\frac{|\mathfrak{C}_{d,n}|}{|\mathfrak{A}_{d,n}|}} = 1.
\]
If $d \ge 5$, then
\[
\lim_{n \to \infty}{\frac{|\mathfrak{C}_{d,n}|}{|\mathfrak{A}_{d,n}|}} = 1.
\]
If $d$ and $n$ satisfy one of the conditions in (\ref{dnres}), then
\[
\lim_{p \to \infty}{\frac{|\mathfrak{C}_{d,n}|}{|\mathfrak{A}_{d,n}|}} = 1
\]
\end{cor}

\begin{proof}
When $n = 2$, the sets $\mathfrak{A}_{d,n}$ and $\mathfrak{C}_{d,n}$ are the same, so trivially
\[
\lim_{d \to \infty}{\frac{|\mathfrak{C}_{d,n}|}{|\mathfrak{A}_{d,n}|}}
= \lim_{p \to \infty}{\frac{|\mathfrak{C}_{d,n}|}{|\mathfrak{A}_{d,n}|}}
= 1.
\]
For all other cases, we will use Theorem~\ref{limit1thm}.  We have the inequality
\begin{eqnarray}
d_{n-1} - d_n/4 + d^2 - 11/16
&=& -\frac{1}{4} \left( d_n - 4d_{n-1} - 4d^2 + \frac{11}{4} \right) \nonumber \\
&\le& -\frac{1}{4} \left( \frac{d^n}{n} - \frac{30d^{n-1}}{7(n-1)} - 4d^2 + \frac{11}{4} \right). \label{upperboundexpression}
\end{eqnarray}
When $n \ge 3$ and $d \to \infty$, the quantity (\ref{upperboundexpression}) has limit $-\infty$.  Combined with Theorem~\ref{limit1thm}, this shows that
\[
\lim_{d \to \infty}{\frac{|\mathfrak{C}_{d,n}|}{|\mathfrak{A}_{d,n}|}} = 1.
\]
When $d \ge 5$ and $n \to \infty$, the quantity (\ref{upperboundexpression}) is asymptotically $-d^n/4n$, which shows that
\[
\lim_{n \to \infty}{\frac{|\mathfrak{C}_{d,n}|}{|\mathfrak{A}_{d,n}|}} = 1.
\]
Finally, by Lemma~\ref{dnasymptotics2} Equation~\ref{dninequality3},
\[
d_{n-1} - d_n/4 + d^2 - 11/16 \le 0
\]
for all values of $d$ and $n$ satisfying one of the conditions in (\ref{dnres}) (except the condition $n = 2$, which we have already dealt with).  This, combined with Theorem~\ref{limit1thm} and the fact that $C(p)$ and $D(p)$ go to 1 as $p \to \infty$, implies that
\[
\lim_{p \to \infty}{\frac{|\mathfrak{C}_{d,n}|}{|\mathfrak{A}_{d,n}|}} = 1.
\]
\end{proof}

\chapter{Counting Submodules}
\label{c_submodules}

In this chapter we shall prove Theorem~\ref{limit2thm}.  This depends on estimating $|\mathfrak{C}_{d,n}|$, the number of $\GL(d,\F_p)$-orbits on subspaces of $F_n/F_{n+1}$, via the Cauchy-Frobenius Lemma.  To do this, we obtain in Theorem~\ref{upperbound} an upper bound for the number of submodules of an $\F_p \left< g \right>$-module, where $g \in \GL(d, \F_p)$.  Theorem~\ref{strongerbound} strengthens this bound in a special case to deal with $F_2/F_3$.  Both theorems draw heavily on the theory of Hall polynomials, which count the number of submodules of fixed type and cotype of a finite module over a discrete valuation ring.

\section{The Number of Submodules of a Module}
\label{submodulesec}

Suppose $M$ is an $\F_p \GL(d, \F_p)$-module.  Let $g \in \GL(d, \F_p)$.  We want to count the number of subspaces of $M$ (viewed as an $\F_p$-vector space) fixed by $g$, which is the number of submodules of $M$ as a $\F_p \left< g \right>$-module.  We note that when $M$ is the natural $\F_p \GL(d, \F_p)$-module, Eick and O'Brien~\cite{eo} give an explicit formula for this number.  The following preliminaries are based on Macdonald~\cite[Chapter IV, Section 2]{mac}.

Any $\F_p \left< g \right>$-module $M$ can be viewed as an $\F_p[t]$-module, where $t.v = gv$ for all $v \in M$.  Furthermore, the number of $\F_p \left< g \right>$-submodules of $M$ equals the number of $\F_p[t]$-submodules of $M$.  Let $\Phi$ be the set of all polynomials in $\F_p[t]$ which are irreducible over $\F_p$, and let $P$ be the set of all partitions of non-negative integers.  Let $U$ be the set of all functions $\mu: \Phi \to P$.  Since $\F_p[t]$ is a principal ideal domain, $M$ has a unique decomposition of the form
\[
M \cong \bigoplus_{f \in \Phi}{\bigoplus_i{\frac{\F_p[t]}{(f)^{\mu_i(f)}}}},
\]
for some $\mu \in U$.  Here, $\mu_i(f)$ is the $i$-th part of $\mu(f)$.  Let
\[
M_f = \bigoplus_i{\frac{\F_p[t]}{(f)^{\mu_i(f)}}}.
\]
For each $f \in \Phi$, let $\F_p[t]_f$ denote the localization of $\F_p[t]$ at the prime ideal $(f)$.  Then $\F_p[t]_f$ is a discrete valuation ring with residue field of order $q = p^{\deg(f)}$, and $M_f$ is a finite $\F_p[t]_f$-module.  We call $\mu(f)$ the \emph{type} of $M_f$.

Any submodule $N$ of $M$ can be written $N = \oplus_{f \in \Phi}{N_f}$ with $N_f \subseteq M_f$ for each $f \in \Phi$.  That is, every submodule of $M$ is the direct sum of submodules of the summands $M_f$.  By Macdonald~\cite[Chapter II, Lemma 3.1]{mac} the type $\lambda$ of any $\F_p[t]$-submodule or quotient module of $M_f$ satisfies $\lambda \subseteq \mu(f)$.

Both Theorems~\ref{upperbound} and \ref{strongerbound} depend on Theorem~\ref{submodulesthm}, where we calculate the number of submodules of fixed type in a module of fixed type over a discrete valuation ring.  This generalizes Birkhoff's formula for the number of subgroups of a finite abelian $p$-group (see~\cite{bir}); to recover Birkhoff's result, let $\mathfrak{a}$ be the ring of $p$-adic integers.  The reliance of Theorem~\ref{submodulesthm} (and its proof) on the theory of Hall polynomials is hidden in the citation of results from Macdonald~\cite[Chapter II]{mac}.  While we will not pursue this connection, it should be noted that the quantity $S(\alpha', \beta', q)$ appearing in Theorem~\ref{submodulesthm} is equal to $\sum_{\mu \in P}{g_{\mu \nu}^{\lambda}}(q)$, where $g_{\mu \nu}^{\lambda}(q)$ is the Hall polynomial corresponding to $\lambda$, $\mu$, and $\nu$, and so Theorem~\ref{submodulesthm} can also be phrased as a result about a sum of Hall polynomials.

\begin{thm} \label{submodulesthm}
Let $\mathfrak{a}$ be a discrete valuation ring with maximal ideal $\mathfrak{p}$ and let $\mathfrak{k} = \mathfrak{a}/\mathfrak{p}$ be the residue field of order $q$.  Let $\alpha = (\alpha_1, \alpha_2, \dots, \alpha_r)$ and $\beta = (\beta_1, \beta_2, \dots, \beta_s)$ be partitions with $\beta \subseteq \alpha$ and let $M$ be a finite $\mathfrak{a}$-module of type $\alpha'$.  Then the number of submodules of $M$ of type $\beta'$ is
\[
S(\alpha', \beta', q) = \prod_{i=1}^s{\sbinom{\alpha_i - \beta_{i+1}}{\beta_i - \beta_{i+1}}_q q^{\beta_{i+1} (\alpha_i - \beta_i)}},
\]
where $\beta_{s+1}$ is taken to be $0$.
\end{thm}

\begin{proof}
The proof is by induction on $\beta_1$.  If $\beta_1 = 0$, then $S(\alpha', \beta', q) = 1$ and the result holds.  Suppose $\beta_1 > 0$, and let the smallest part of $\beta'$ be $t$, so that either $\beta_1 = \cdots = \beta_t > \beta_{t+1}$ and $t < s$, or $\beta_1 = \cdots = \beta_s$ and $t = s$.  Write
\[
\overline{\beta} = (\beta_1 - 1, \beta_2 - 1, \dots, \beta_t - 1, \beta_{t+1}, \dots, \beta_s).
\]
Let $N$ be any submodule of $M$ of type $\overline{\beta}'$, and let $x$ be any element of $M$ with $\mathfrak{p}^t x = 0$, $\mathfrak{p}^{t-1} x \neq 0$, and $\mathfrak{a} x \cap N = 0$.  Then $\left< N, x \right>$ has type $\beta'$.  There are $S(\alpha', \overline{\beta}', q)$ choices for $N$, and for each $N$ it follows from~\cite[Chapter II, Equation 1.8]{mac} that the number of choices for $x$ is just
\begin{equation}
q^{\alpha_1 + \cdots + \alpha_t} ( 1 - q^{\beta_t - \alpha_t - 1}).
\end{equation}
On the other hand, fix a submodule $L$ of $M$ of type $\beta'$; we can count the number of choices of $N$ and $x$ so that $L = \left< N, x \right>$.  Here $N$ is a submodule of $L$ of type $\overline{\beta}'$ whose quotient has type $(t)$, and by~\cite[Chapter II, Equation 4.13]{mac}, the number of choices for $N$ is
\begin{eqnarray*}
&& \frac{1 - q^{\beta_{t+1} - \beta_t}}{1 - q^{-1}} \; q^{\sum_{i=1}^s{\binom{\beta_i}{2}} - \sum_{i=1}^s{\binom{\overline{\beta}_i}{2}}} \\
&=& \frac{1 - q^{\beta_{t+1} - \beta_t}}{1 - q^{-1}} \; q^{t (\beta_t - 1)}.
\end{eqnarray*}
Given $N$, it follows from~\cite[Chapter II, Equation 1.8]{mac} that there are
\[
q^{\beta_1 + \cdots + \beta_t}(1-q^{-1})
\]
choices for $x$.  Thus any submodule $L$ of $M$ of type $\beta'$ arises as $\left< N, x \right>$ in
\[
q^{\beta_1 + \cdots + \beta_t + t(\beta_t - 1)} (1 - q^{\beta_{t+1} - \beta_t})
\]
ways.  The total number of submodules $L$ of $M$ of type $\beta'$ is then
\begin{eqnarray}
S(\alpha', \beta', q) &=& \frac{S(\alpha', \overline{\beta}', q) q^{\alpha_1 + \cdots + \alpha_t} ( 1 - q^{\beta_t - \alpha_t - 1})}{
q^{\beta_1 + \cdots + \beta_t + t(\beta_t - 1)} (1 - q^{\beta_{t+1} - \beta_t})} \nonumber \\
&=& \frac{S(\alpha', \overline{\beta}', q)
q^{\alpha_1 + \cdots + \alpha_t} ( 1 - q^{\beta_t - \alpha_t - 1})}{
q^{2t \beta_t - t} (1 - q^{\beta_{t+1} - \beta_t})}, \label{indexp}
\end{eqnarray}
where the second inequality uses $\beta_1 = \cdots = \beta_t$.  By induction, we know that
\begin{eqnarray*}
S(\alpha', \overline{\beta}', q) &=& \prod_{i=1}^s{\sbinom{\alpha_i - \overline{\beta}_{i+1}}{\overline{\beta}_i - \overline{\beta}_{i+1}}_q q^{\overline{\beta}_{i+1} (\alpha_i - \overline{\beta}_i)}} \\
&=& \prod_{i=1}^{t-1}{\sbinom{\alpha_i - \beta_{i+1} + 1}{\beta_i - \beta_{i+1}}_q q^{(\beta_{i+1} - 1)(\alpha_i - \beta_i + 1)}} \\
&& \qquad \cdot  \sbinom{\alpha_t - \beta_{t+1}}{\beta_t - \beta_{t+1} - 1}_q q^{\beta_{t+1} (\alpha_t - \beta_t + 1)} \\
&& \qquad \cdot \prod_{i=t+1}^s{\sbinom{\alpha_i - \beta_{i+1}}{\beta_i - \beta_{i+1}}_q q^{\beta_{i+1} (\alpha_i - \beta_i)}} \\
&=& \prod_{i=1}^s{\sbinom{\alpha_i - \beta_{i+1}}{\beta_i - \beta_{i+1}}_q q^{\beta_{i+1} (\alpha_i - \beta_i)}} \\
&& \qquad \cdot \prod_{i=1}^{t-1}{ \frac{q^{\alpha_i-\beta_{i+1}+1}-1}{q^{\alpha_i-\beta_i+1}-1} q^{\beta_{i+1} + \beta_i - \alpha_i - 1}} \cdot \frac{q^{\beta_t-\beta_{t+1}}-1}{q^{\alpha_t-\beta_t+1}-1} q^{\beta_{t+1}} \\
&=& \prod_{i=1}^s{\sbinom{\alpha_i - \beta_{i+1}}{\beta_i - \beta_{i+1}}_q q^{\beta_{i+1} (\alpha_i - \beta_i)}} \\
&& \qquad \cdot q^{2(t-1)\beta_t - \alpha_1 - \cdots - \alpha_{t-1} - (t-1)} \cdot \frac{q^{\beta_t-\beta_{t+1}}-1}{q^{\alpha_t-\beta_t+1}-1} q^{\beta_{t+1}} \\
&=& \prod_{i=1}^s{\sbinom{\alpha_i - \beta_{i+1}}{\beta_i - \beta_{i+1}}_q q^{\beta_{i+1} (\alpha_i - \beta_i)}} \cdot \frac{q^{2t \beta_t}}{q^{\alpha_1 + \cdots + \alpha_t + t}} \cdot \frac{1 - q^{\beta_{t+1} - \beta_t}}{1 - q^{\beta_t - \alpha_t - 1}}.
\end{eqnarray*}
Substituting this expression into Equation~\ref{indexp} gives the result.
\end{proof}

Using Theorem~\ref{submodulesthm}, the techniques of Appendix~\ref{a_estimates}, and the definitions of $C(x)$ and $D(x)$ from Equation~\ref{cdeqns}, we can give an upper bound for the total number of submodules of a finite $\F_p \left< g \right>$-module $M$.  Note that every subspace of $M$ is a $\F_p \left< g \right>$-module if and only if $g$ acts as a scalar on $M$, that is, as multiplication by an element of $\F_p$.

\begin{thm} \label{upperbound}
Fix $d \ge 2$ and $g \in \GL(d, \F_p)$.  Suppose that $M$ is an $\F_p \left< g \right>$-module.  Let $m = \dim_{\F_p}(M)$ and let $S_M$ be the number of submodules of $M$.  Then either $g$ acts as a scalar on $M$ and $S_M = \mathcal{G}_m(p)$, or $g$ does not act as a scalar and
\[
\log_p{S_M} \le (m^2-2m+2)/4 + 2\ep,
\]
where $\ep = \log_p(C(p)D(p))$.
\end{thm}

\begin{proof}
Write $M = \oplus_{i=1}^k{M_i}$, where for each $i$, $M_i = M_{f_i}$ for some $f_i \in \Phi$ and $\dim_{\F_p}{M_i} = m_i$.

\hspace{1in}

\noindent
\emph{Case 1: $k \ge 2$}.

Each submodule of $M$ is a direct sum of submodules of the summands $M_i$, so $S_M = \prod_{i=1}^k{S_{M_i}} \le \mathcal{G}_{m_1}(p) \mathcal{G}_{m-m_1}(p)$.  Then by Lemma~\ref{qests},
\[
S_M \le C(p)^2 D(p)^2 p^{m_1^2/4 + (m-m_1)^2/4} \le C(p)^2 D(p)^2 p^{(m^2-2m+2)/4},
\]
since $0 < m_1 < m$.

\hspace{1in}

\noindent
\emph{Case 2: $k = 1$}.

In this case, $M = M_f$ for some $f \in \Phi$.  Let $u = \deg(f)$ and $q = p^u$, and let $M$ have type $\alpha'$ as a $\F_p[t]_f$-module, where $\alpha = (\alpha_1, \dots, \alpha_r)$.

\hspace{1in}

\noindent
\emph{Subcase 2.1: $\alpha$ has at least two parts}.

If $\beta = (\beta_1, \dots, \beta_s)$ and $\beta \subseteq \alpha$, then by Theorem~\ref{submodulesthm} and Lemma~\ref{qests} Equation~\ref{coefbounds}, the number of submodules of $M$ of type $\beta'$ is
\begin{eqnarray*}
S(\alpha', \beta', q) &\le& \prod_{i=1}^s{D(q) q^{(\beta_i - \beta_{i+1})(\alpha_i - \beta_i) + \beta_{i+1}(\alpha_i - \beta_i)}} \\
&=& D(q)^s \prod_{i=1}^s{q^{\beta_i (\alpha_i - \beta_i)}}.
\end{eqnarray*}
Thus
\begin{eqnarray*}
S_M &=& \sum_{\beta' \subseteq \alpha'}{S(\alpha', \beta', q)} \\
&\le& D(q)^r \sum_{\beta' \subseteq \alpha'}{\prod_{i=1}^r{q^{\beta_i (\alpha_i - \beta_i)}}} \\
&\le& D(q)^r \prod_{i=1}^r{\sum_{\beta_i=0}^{\alpha_i}{q^{\beta_i(\alpha_i-\beta_i)}}} \\
&\le& D(q)^r C(q)^r \prod_{i=1}^r{q^{\alpha_i^2/4}},
\end{eqnarray*}
where the last inequality follows from Lemma~\ref{polybound}.  Now $D(q) \le D(p)$ and $C(q) \le C(p)$ so, remembering that $u(\alpha_1 + \cdots + \alpha_r) = m$ and using Lemma~\ref{quadbound},
\begin{eqnarray}
\log_p{S_M} &\le& u(\alpha_1^2 + \cdots + \alpha_r^2)/4 + r\ep \label{logbound} \\
&\le& ((u\alpha_1)^2 + \cdots + (u\alpha_r)^2 + 4r \ep)/4\nonumber \\
&\le& ((m-1)^2 + 1 + 8 \ep)/4\nonumber \\
&\le& (m^2-2m+2)/4 + 2\ep, \nonumber
\end{eqnarray}
if $m \ge 4\ep + 1$.  For $m < 4\ep + 1$,
\begin{eqnarray*}
\log_p{S_M} &\le& m^2/4 \\
&\le& (m^2-2m+2)/4 + (m-1)/2 \\
&\le& (m^2-2m+2)/4 + 2\ep.
\end{eqnarray*}

\hspace{1in}

\noindent
\emph{Subcase 2.2: $\alpha$ has one part}.

In this case, $\alpha_1 = m/u$.  If $u \ge 2$, then by Lemma~\ref{qests} Equation~\ref{coefbounds},
\begin{eqnarray*}
S_M &=& \sum_{0 \le \beta_1 \le \alpha_1}{\sbinom{\alpha_1}{\beta_1}_q} \\
&\le& C(q) D(q) q^{m^2/4u^2} \\
&\le& C(p)^2 D(p)^2 p^{m^2/4u} \\
&\le& C(p)^2 D(p)^2 p^{(m^2-2m+2)/4},
\end{eqnarray*}
since $u \ge 2$.  On the other hand, if $u = 1$, then $f = t-c$ for some $c \in \F_p$ and $M \cong \oplus^m \{\F_p[t]/(f)\}$ so that $g$ acts as the scalar $c$ on $M$ and $S_M = \mathcal{G}_m(p)$.
\end{proof}

The next theorem strengthens the preceding result when the module structure is known more precisely and will be needed to deal with groups of lower $p$-length 2.

\begin{thm} \label{strongerbound}
Fix $d \ge 2$ and $g \in \GL(d, \F_p)$ with $g \neq 1$.  Suppose that $V$ is an $\F_p \left< g \right>$-module on which $g$ acts non-trivially and that $M$ is an $\F_p \left< g \right>$-module extension of $V \wedge V$ by $V$.  Let $v = \dim_{\F_p}(V)$, let $m = \dim_{\F_p}(M) = v(v+1)/2$, and let $S_M$ be the number of submodules of $M$.  Then
\[
\log_p{S_M} \le (m-4)^2/4 + C,
\]
where $\ep = \log_p{(C(p)D(p))}$ and
\[
C = \left\{
  \begin{array}{c@{\quad:\quad}l}
    \ep + 2m-4 & m \le 45 \\
    5\ep + 4 & m > 45.
  \end{array}
    \right.
\]
\end{thm}

\begin{proof}
First, if $v \le 9$, then $m \le 45$.  In this case,
\begin{eqnarray*}
S_M &\le& \mathcal{G}_m(p) \\
&\le& C(p) D(p) p^{m^2/4} \\
&=& C(p) D(p) p^{(m-4)^2/4 + 2m-4},
\end{eqnarray*}
proving the result.  So we may assume that $v \ge 10$.

Write $M = \oplus_{i=1}^k{M_i}$, where for each $i$, $M_i = M_{f_i}$ for some $f_i \in \Phi$ and $\dim_{\F_p}{M_i} = m_i$; we may assume that $m_1 \ge m_2 \ge \cdots \ge m_k$.  Note that $m_1 + \cdots + m_k = m$.  Then $V = \oplus_{i=1}^m{\pi M_i}$ where $\pi$ is the projection from $M$ onto $V$.

Fix $0 < t < k$ and set $W = M_1 \oplus \cdots \oplus M_t$.  Also let $w = \dim{W} = m_1 + \cdots + m_t$.  Then $S_M \le \mathcal{G}_w(p) \mathcal{G}_{M-w}(p)$ since any submodule of $M$ is a direct sum of submodules of the summands $M_i$.  By Lemma~\ref{qests},
\[
S_M \le C(p)^2 D(p)^2 p^{w^2/4 + (M-w)^2/4}.
\]
When $4 \le w \le M-4$, it follows that
\begin{eqnarray*}
S_M &\le& C(p)^2 D(p)^2 p^{4 + (M-4)^2/4} \textrm{ and} \\
\log_p{S_M} &\le& (M-4)^2/4 + 2\ep + 4,
\end{eqnarray*}
proving the result.  If we cannot choose $t$ so that $4 \le w \le m-4$, then since $m > 9$ implies that $m_1 \not\le 3$, it must be that $m_1 \ge m-3$ and $k \le 4$.  Write $Y = M_2 \oplus \cdots \oplus M_k$; then $y = \dim{Y} \le 3$.  (It is possible that $Y$ is the zero module and that $y = 0$.)  At this point we need to prove a technical claim which we will use twice.

\hspace{1in}

\textbf{Claim:}  Suppose that $V$ is the direct sum of $\F_p \left< g \right>$-modules $A$ and $B$ of dimensions $a \ge 4$ and $v-a$ over $\F_p$, and suppose that $A \subset M_1 \pi$.  If $g$ acts as a scalar $c$ on $A$, then $c = 1$ and $A \otimes B$ is the direct sum of $a$ copies of $B$.

\hspace{1in}

\emph{Proof of claim:}  If $V = A \oplus B$, then $V \wedge V \cong (A \wedge A) \oplus (B \wedge B) \oplus (A \otimes B)$.  If $g$ acts as a scalar $c$ on $A$, then $A \cong \oplus \{ \F_p[t]/(t-c) \}^a$ and $M_1 = M_{f_1}$ with $f_1 = t-c$.  In this case $g$ acts as the scalar $c^2$ on $A \wedge A$, so $A \wedge A \cong \{ \F_p[t]/(t-c^2) \}^{a(a-1)/2}$.  If $c \neq 1$, then $A \wedge A \not\subseteq M_1$ and hence $A \wedge A \subseteq Y$.  But then $a(a-1)/2 = \dim(A \wedge A) \le \dim{Y} \le 3$, which is impossible.  Therefore $c = 1$.  Since $g$ acts on $V$ non-trivially, the action on $B$ is non-trivial and $A \otimes B$ is the direct sum of $a$ copies of $B$.

\hspace{1in}

Now take $A = \pi M_1$ and $B = \pi Y$ so that $V = A \oplus B$.  Suppose that $g$ acts on $A$ as a scalar $c$.  Since $v \ge 7$ and $\dim{B} \le \dim{Y} \le 3$, we see that $a \ge 4$, and by the claim, $c = 1$ and $A \otimes B$ is the direct sum of $a$ copies of $B$.  If $B$ is the zero module, this contradicts the fact that $g$ acts non-trivially on $V$.  Otherwise, $v-a > 0$.  Since $B$ is the image of $Y$, it follows that $A \otimes B \subseteq Y$, and $a(v-a) \le \dim{Y} \le 3$, which is false.  Therefore $g$ does not act on $\pi M_1$ as a scalar, and hence does not act on $M_1$ as a scalar.

We may assume that $M_1 = M_f$ where $f$ has degree $u$ over $\F_p$ and $M_1$ and $M_1 \pi$ have types $\alpha'$ and $\beta'$ respectively, where $\beta \subseteq \alpha$.  Write $\alpha = (\alpha_1, \dots, \alpha_r)$ and $\beta = (\beta_1, \dots, \beta_s)$.

\hspace{1in}

\noindent
\emph{Case 1: $u > 1$}.

Writing $S_{M_1}$ for the number of submodules of $M_1$, we have
\begin{eqnarray*}
S_{M_1} &\le& \mathcal{G}_{m_1/u}(q) \\
&\le& C(q) D(q) q^{m_1^2/4u^2} \\
&\le& C(p) D(p) p^{m_1^2/4u} \\
&\le& C(p) D(p) p^{m_1^2/8}.
\end{eqnarray*}
Then
\begin{eqnarray*}
S_{M} &\le& S_{M_1} \mathcal{G}_y(p) \\
&\le& C(p)^2 D(p)^2 p^{m_1^2/8 + y^2/4} \\
&\le& C(p)^2 D(p)^2 p^{m^2/8 + 9/4} \\
&\le& C(p)^2 D(p)^2 p^{(m-4)^2/4 + 9/4},
\end{eqnarray*}
where the last line uses the fact that $m \ge 14$.  Thus $\log_p{S_M} \le C + (m-4)^2/4$.

\hspace{1in}

\noindent
\emph{Case 2: $u = 1$}.

In this case, $f = t-c$ for some $c \in \F_p$.  Since $g$ does not act as a scalar on $M_1$ or $\pi M_1$, we know $\alpha_2 \ge \beta_2 > 0$.

By Equation~\ref{logbound},
\[
\log_p{S_M} \le (\alpha_1^2 + \cdots + \alpha_r^2)/4 + r\ep,
\]
so
\[
\log_p{S_M} \le \log_p{S_{M_1}} + \log_p{\mathcal{G}_y(p)} \le (\alpha_1^2 + \cdots + \alpha_r^2 + y^2)/4 + (r+1)\ep.
\]

\hspace{1in}

\noindent
\emph{Subcase 2.1: $\alpha_1 \le m-4$}

If $r = 2$, then
\begin{eqnarray*}
\log_p{S_M} &\le& (\alpha_1^2 + \alpha_2^2 + y^2)/4 + 3\ep\\
&\le& ((m-4)^2 + 4^2 + 0^2)/4 + 3\ep\\
&\le& (m-4)^2/4 + C.
\end{eqnarray*}
If $r = 3$, then
\begin{eqnarray*}
\log_p{S_M} &\le& (\alpha_1^2 + \alpha_2^2 + \alpha_3^2 + y^2)/4 + 4\ep \\
&\le& ((m-4)^2 + 3^2 + 1^2 + 0^2)/4 + 4\ep \\
&\le& (m-4)^2/4 + C.
\end{eqnarray*}
Finally, if $4 \le r \le m$, then by Lemma~\ref{quadbound}, we get
\begin{eqnarray*}
\log_p{S_M} &\le& ((m-r)^2 + r)/4 + (r+1)\ep.
\end{eqnarray*}
The right-hand side is maximized at $r = 4$ or $r = m$.  Since $m > 45$ and $\ep \le 6$, it turns out that it is maximized at $r = 4$, where we get a bound of $(m-4)^2/4 + 5\ep + 1$.

\hspace{1in}

\noindent
\emph{Subcase 2.2: $\alpha_1 \ge m-3$}.

So we may assume that $\alpha_1 \ge m-3$.  Then $\alpha_2+\cdots+\alpha_r+y \le 3$, and so $\beta_2 + \cdots + \beta_s + \dim(\pi Y) \le 3$.  Since $\beta_1 + \cdots + \beta_s + \dim(\pi Y) = v \ge 10$, it follows that $\beta_1 \ge 7$ and $\beta_1 - \beta_2 \ge 4$.  Note that $\beta_1 - \beta_2$ is the number of summands of $\pi M_1$ that are isomorphic to $\F_p[t]/(f-c)$.  So write $\pi M_1 = A \oplus C$, where $a = \dim{A} = \beta_1 - \beta_2$ and $g$ acts as the scalar $c$ on $A$ and not on $C$.  Set $B = C \oplus \pi Y$.  Then $V = A \oplus B$ and by the claim, $c = 1$ and $A \otimes B$ is a direct sum of $a$ copies of $B$.  Then $A \otimes B$ is contained in $Y$ plus the components of $M_1$ that $g$ does not act as a scalar on, so that $a \beta_2 \le \dim{(A \otimes B)} \le \alpha_2 + y \le 3$, which is impossible.
\end{proof}

\section{A Proof of Theorem~\ref{limit2thm}}

We can now prove Theorem~\ref{limit2thm} and Corollary~\ref{limit2cor}, restated here for convenience.

\renewcommand{\arraystretch}{1.2}
\begin{thmx}[Theorem~\ref{limit2thm}]
Fix a prime $p$ and integers $d$ and $n$ so that either $n = 2$ and $d \ge 10$ or $n \ge 3$ and $d \ge 3$.  Let $F$ be the free group of rank $d$ and let $d_n$ be the dimension of $F_n/F_{n+1}$.  Let
\[
c_1 = \left\{
  \begin{array}{r@{\quad:\quad}l}
    C(p)^5 D(p)^4 p^{17/4} & \textrm{$n = 2$ and $d \ge 10$} \\
    C(p)^2 D(p) p^{3/4} & n \ge 3.
  \end{array}
  \right.
\]
Let
\[
c_2 = \left\{
  \begin{array}{r@{\quad:\quad}l}
    -d & n = 2 \\
    d^2-d_n/2 & n \ge 3.
  \end{array}
  \right.
\]
Then
\begin{enumerate}
\renewcommand{\labelenumi}{\emph{(\alph{enumi})}}
\renewcommand{\theenumi}{\ref{est2thm}(\alph{enumi})}
\item
\[
1 \le \frac{|\mathfrak{C}_{d, n}| \cdot |\GL(d,\F_p)|}{|\mathcal{C}_{d, n}|} \le 1 + c_1 p^{c_2}.
\]
\item
\[
1 \le \frac{|\mathfrak{C}_{d, n}|}{|\mathfrak{D}_{d, n}|} \le \frac{1+c_1 p^{c_2}}{1-c_1 p^{c_2}}.
\]
\end{enumerate}
\end{thmx}
\renewcommand{\arraystretch}{1}

\begin{proof}
Let $(\mathcal{C}_{d,n})^g$ be the set of elements of $\mathcal{C}_{d,n}$ fixed by $g$.  Then $|(\mathcal{C}_{d,n})^g|$ is just the number of submodules of $F_n/F_{n+1}$ viewed as a $\F_p \left< g \right>$-module.  We explain first why only the identity element of $\GL(d, \F_p)$ can act as a scalar on $F_n/F_{n+1}$.  By Corollary~\ref{wedgecor}, $F_n/F_{n+1}$ has a $\F_p \GL(d, \F_p)$-submodule $M$ which is isomorphic to an extension of $V \wedge V$ by $V$, where $V$ is the natural $\F_p \GL(d,\F_p)$-module.  If $g \in \GL(d,\F_p)$ acts on $F_n/F_{n+1}$ as a scalar $c \in \F_p$, then it acts on $V$ as the scalar $c$, and hence on $V \wedge V$ as the scalar $c^2$.  Thus $c = c^2$ and $c = 1$, so that $g$ is the identity on $V$, that is, the identity element in $\GL(d, \F_p)$.

Suppose first that $n > 2$.  We know from Theorem~\ref{upperbound} that if $g \neq 1$,
\[
|(\mathcal{C}_{d,n})^g| \le C(p)^2 D(p)^2 p^{(d_n^2 - 2d_n + 2)/4}.
\]
By the Cauchy-Frobenius Lemma,
\begin{eqnarray*}
|\GL(d,\F_p)| \cdot |\mathfrak{C}_{d,n}| &=& \sum_{g \in \GL(d,\F_p)}{|(\mathcal{C}_{d,n})^g|} \\
&=& |\mathcal{C}_{d,n}| + \sum_{1 \neq g \in \GL(d,\F_p)}{|(\mathcal{C}_{d,n})^g|} \\
&\le& |\mathcal{C}_{d,n}| + (|\GL(d,\F_p)| - 1) C(p)^2 D(p)^2 p^{(d_n^2-2d_n+2)/4}.
\end{eqnarray*}
By Equation~\ref{gnbounds} and the fact that $2-9p^{(1-d_n)/2}/2 > 1$,
\[
|\mathcal{C}_{d,n}| \ge D(p) p^{d_n^2/4-1/4}.
\]
Since $|\GL(d,\F_p)| \le p^{d^2}$, it follows that
\begin{eqnarray*}
1 &\le& \frac{|\GL(d,\F_p)| \cdot |\mathfrak{C}_{d,n}|}{|\mathcal{C}_{d,n}|} \\
&\le& 1 + C(p)^2 D(p) \; p^{(d_n^2-2d_n+2)/4 + d^2 - d_n^2/4 + 1/4} \\
&=& 1 + c_1 p^{d^2 - d_n/2}.
\end{eqnarray*}
If $n = 2$, then $F_2/F_3$ is an extension of $V \wedge V$ by $V$, and using the estimates of Lemma~\ref{strongerbound} and the argument above we obtain
\begin{eqnarray*}
1 &\le& \frac{|\GL(d,\F_p)| \cdot |\mathfrak{C}_{d,n}|}{|\mathcal{C}_{d,n}|} \\
&\le& 1 + c_1 p^{-d}.
\end{eqnarray*}
This proves part $(a)$.

To prove part $(b)$, we observe that $|\mathcal{C}_{d,n}| = \sum{|\GL(d,\F_p)|/|\GL(d,\F_p)_{(w)}|}$, where the sum is over all $\GL(d,\F_p)$-orbits in $\mathcal{C}_{d,n}$ and $|\GL(d,\F_p)_{(w)}|$ is the order of the stabilizer in $\GL(d,\F_p)$ of any element $w$ of the orbit under consideration.  Now $|\mathfrak{D}_{d, n}|$ is just the number of orbits for which $|\GL(d,\F_p)_{(w)}| = 1$, so
\[
|\mathcal{C}_{d,n}| \le |\GL(d,\F_p)| \cdot |\mathfrak{D}_{d,n}| + |\GL(d,\F_p)| (|\mathfrak{C}_{d,n}| - |\mathfrak{D}_{d,n}|) / 2.
\]
That is,
\[
(2/|\GL(d,\F_p)|) |\mathcal{C}_{d,n}| - |\mathfrak{C}_{d,n}| \le |\mathfrak{D}_{d,n}|,
\]
so that
\begin{eqnarray*}
\frac{|\mathfrak{C}_{d,n}|}{|\mathfrak{D}_{d,n}|} &\le& \frac{|\mathfrak{C}_{d,n}|}{2|\mathcal{C}_{d,n}|/|\GL(d,\F_p)| - |\mathfrak{C}_{d,n}|} \\
&\le& \frac{|\mathfrak{C}_{d,n}| \cdot |\GL(d,\F_p)| /|\mathcal{C}_{d,n}|}{2 - |\mathfrak{C}_{d,n}| \cdot |\GL(d,\F_p)| /|\mathcal{C}_{d,n}|} \\
&\le& \frac{1 + c_1 p^{c_2}}{1 - c_1 p^{c_2}}.
\end{eqnarray*}
\end{proof}

\begin{cor}
If $n \ge 2$, then
\[
\lim_{d \to \infty}{\frac{|\mathfrak{D}_{d,n}|}{|\mathfrak{C}_{d,n}|}} = 1.
\]
If $d \ge 3$, then
\[
\lim_{n \to \infty}{\frac{|\mathfrak{D}_{d,n}|}{|\mathfrak{C}_{d,n}|}} = 1.
\]
If $n = 2$ and $d \ge 10$, or $n \ge 3$ and $d \ge 5$, or $n \ge 4$ and $d \ge 3$, then
\[
\lim_{p \to \infty}{\frac{|\mathfrak{D}_{d,n}|}{|\mathfrak{C}_{d,n}|}} = 1
\]
\end{cor}

\begin{proof}
It is clear from Theorem~\ref{limit2thm} that the first two limits hold.  For the third limit, note that by Lemma~\ref{dnasymptotics2} Equation~\ref{dninequality3},
\[
d^2 - d_n/2 < -3/4
\]
for all values of $d$ and $n$ for which Equation~\ref{dninequality3} holds.  For the finitely many values of $d$ and $n$ for which Equation~\ref{dninequality3} does not hold but $n = 2$ and $d \ge 10$, or $n \ge 3$ and $d \ge 5$, or $n \ge 4$ and $d \ge 3$, a computer check shows that
\[
d^2 - d_n/2 < -3/4.
\]
Combined with Theorem~\ref{limit2thm}, this shows that
\[
\lim_{p \to \infty}{\frac{|\mathfrak{C}_{d,n}|}{|\mathfrak{A}_{d,n}|}} = 1
\]
for the given values of $d$ and $n$.
\end{proof}

\chapter{A Survey on the Automorphism Groups of Finite $p$-Groups}
\label{c_examples}
\chaptermark{Survey of Automorphism Groups}

This chapter constitutes a survey of some of what is known about the automorphism groups of finite $p$-groups.  The focus is on three topics: explicit computations for familiar finite $p$-groups; constructions of finite $p$-groups whose automorphism groups satisfy certain conditions; and the discovery of finite $p$-groups whose automorphism groups are or are not $p$-groups themselves.

\section{The Automorphisms of Familiar \lowercase{$p$}-Groups}
\label{examplesec}

There are several familiar families of finite $p$-groups for which some information about their automorphism groups is known.  In particular, for a couple of these families, the automorphism groups have been described in a reasonably complete manner.  The goal of this section to present these results as concretely as possible.  We begin with a nearly exact determination of the automorphism groups of the extraspecial $p$-groups.  The next subsection discusses the maximal unipotent subgroups of Chevalley groups, for which Gibbs~\cite{gib} describes six types of automorphisms that generate the automorphism group.  For type $A_{\ell}$, Pavlov~\cite{pav} and Weir~\cite{wei} have (essentially) computed the exact structure of the automorphism group.  The last three subsections summarize what is known about the automorphism groups of the Sylow $p$-subgroups of the symmetric group, $p$-groups of maximal class, and certain stem covers.  We note that Barghi and Ahmedy~\cite{ba} claim to determine the automorphism group of a class of special $p$-groups constructed by Verardi~\cite{ver}; unfortunately, as pointed out in the MathSciNet review of~\cite{ba}, the proofs are incorrect.

\subsection{The Extraspecial $p$-Groups}

Winter~\cite{win} gives a nearly complete description of the automorphism group of an extraspecial $p$-group.  (Griess~\cite{gri} states many of these results without proof.)  Following Winter's exposition, we will present some basic facts about extraspecial $p$-groups and then describe their automorphisms.

Recall that a finite $p$-group $G$ is \emph{special} if either $G$ is elementary abelian or $Z(G) = G' = \Phi(G)$.  Furthermore, a non-abelian special $p$-group $G$ is \emph{extraspecial} if $Z(G) = G' = \Phi(G) \cong C_p$.  The order of an extraspecial $p$-group is always an odd power of $p$, and there are two isomorphism classes of extraspecial $p$-groups of order $p^{2n+1}$ for each prime $p$ and positive integer $n$, as proved in Gorenstein~\cite[Theorem 5.2]{gor}.  When $p = 2$, both isomorphism classes have exponent $4$.  When $p$ is odd, one of these isomorphism classes has exponent $p$ and the other has exponent $p^2$.

Any extraspecial $p$-group $G$ of order $p^{2n+1}$ has generators $x_1, x_2, \dots, x_{2n}$ satisfying the following relations, where $z$ is a fixed generator of $Z(G)$:
\begin{alignat*}{3}
\left[x_{2i-1}, x_{2i}\right] &= z && \quad \textrm{for $1 \le i \le n$,} \\
\left[x_i, x_j\right] &= 1 && \quad \textrm{for $1 \le i,j \le n$ and $|i-j| > 1$, and} \\
x_i^p &\in Z(G) && \quad \textrm{for $1 \le i \le 2n$}.
\end{alignat*}

When $p$ is odd, either $x_i^p = 1$ for $1 \le i \le 2n$, in which case $G$ has exponent $p$, or $x_1^p = z$ and $x_i^p = 1$ for $2 \le i \le 2n$, in which case $G$ has exponent $p^2$.  When $p = 2$, either $x_i^2 = 1$ for $1 \le i \le 2n$, or $x_1^2 = x_2^2 = z$ and $x_i^2 = 1$ for $3 \le i \le 2n$.

Recall that if two groups $A$ and $B$ have isomorphic centers $Z(A) \stackrel{\phi}{\cong} Z(B)$, then the \emph{central product} of $A$ and $B$ is the group
\[
(A \times B)/\{(z, \phi(z)^{-1}) \; : \; z \in Z(A)\}.
\]
All extraspecial $p$-groups can be written as iterated central products as follows.  If $p$ is odd, let $M$ be the extraspecial $p$-group of order $p^3$ and exponent $p$, and let $N$ be the extraspecial $p$-group of order $p^3$ and exponent $p^2$.  The extraspecial $p$-group of order $p^{2n+1}$ and exponent $p$ is the central product of $n$ copies of $M$, while the extraspecial $p$-group of order $p^{2n+1}$ and exponent $p^2$ is the central product of $n-1$ copies of $M$ and one copy of $N$.  If $p = 2$, the extraspecial $2$-group of order $2^{2n+1}$ and $x_1^2 = x_2^2 = 1$ is isomorphic to the central product of $n$ copies of the dihedral group $D_8$, while the extraspecial $2$-group of order $2^{2n+1}$ and $x_1^2 = x_2^2 = z$ is isomorphic to the central product of $n-1$ copies of $D_8$ and one copy of the quaternion group $Q_8$.

One of the isomorphism classes can be viewed more concretely.  The group of $(n+1) \times (n+1)$ matrices with ones on the diagonal, arbitrary entries from $\F_p$ in the rest of the first row and last column, and zeroes elsewhere is an extraspecial $p$-group of order $p^{2n+1}$.  When $p$ is odd, this is the extraspecial $p$-group of exponent $p$.  When $p = 2$, this is the central product of $n$ copies of $D_8$.

In~\cite{win}, Winter states the following theorem on the automorphism groups of the extraspecial $p$-groups for all primes $p$ (an explicit description of the automorphisms follows the theorem).

\begin{thm}[Winter~\cite{win}]
\label{spthm}
Let $G$ be an extraspecial $p$-group of order $p^{2n+1}$.  Let $I = \Inn(G)$ and let $H$ be the normal subgroup of $\Aut(G)$ which acts trivially on $Z(G)$.  Then
\begin{enumerate}
\item $I \cong (C_p)^{2n}$.
\item $\Aut(G) \cong H \rtimes \left< \theta \right>$, where $\theta$ is an automorphism of order $p-1$.
\item If $p$ is odd and $G$ has exponent $p$, then $H/I \cong \Sp(2n, \F_p)$, and the order of $H/I$ is $p^{n^2} \prod_{i=1}^n{(p^{2i}-1)}$.
\item If $p$ is odd and $G$ has exponent $p^2$, then $H/I \cong Q \rtimes \Sp(2n-2, \F_p)$, where $Q$ is a normal extraspecial $p$-group of order $p^{2n-1}$, and the order of $H/I$ is $p^{n^2} \prod_{i=1}^{n-1}{(p^{2i} - 1)}$.  The group $Q \rtimes \Sp(2n-2, \F_p)$ is isomorphic to the subgroup of $\Sp(2n, \F_p)$ consisting of elements whose matrix $(a_{ij})$ with respect to a fixed basis satisfies $a_{11} = 1$ and $a_{1i} = 0$ for $i > 1$.
\item If $p = 2$ and $G$ is isomorphic to the central product of $n$ copies of $D_8$, then $H/I$ is isomorphic to the orthogonal group of order $2^{n(n-1)+1}(2^n-1)\prod_{i=1}^{n-1}{(2^{2i}-1)}$ that preserves the quadratic form $\xi_1 \xi_2 + \xi_3 \xi_4 + \cdots + \xi_{2n-1} \xi_{2n}$ over $\F_2$.
\item If $p = 2$ and $G$ is isomorphic to the central product of $n-1$ copies of $D_8$ and one copy of $Q_8$, then $H/I$ is isomorphic to the orthogonal group of order $2^{n(n-1)+1}(2^n+1)\prod_{i=1}^{n-1}{(2^{2i}-1)}$ that preserves the quadratic form $\xi_1 \xi_2 + \xi_3 \xi_4 + \cdots + \xi_{2n-1}^2 + \xi_{2n-1} \xi_{2n} + \xi_{2n}^2$ over $\F_2$.
\end{enumerate}
\end{thm}

The automorphisms in $\Aut(G)$ can be described more explicitly.  First, the automorphism $\theta$ may be chosen as follows.  Let $m$ be a primitive root modulo $p$ with $0 < m < p$.  Then define $\theta$ by $\theta(x_{2i-1}) = x_{2i-1}^m$ and $\theta(x_{2i}) = x_{2i}$ for $1 \le i \le n$ and by $\theta(z) = z^m$.  Next, the inner automorphisms are the $p^{2n}$ automorphisms $\sigma$ such that $\sigma(z) = z$ and $\sigma(x_i) = x_i z^{d_i}$ for each $i$ and some choice of integers $0 \le d_i < p$.

It remains to describe $H$.  For each $x \in G$, let $\overline{x}$ denote the coset $x Z(G)$.  Now $G/Z(G)$ becomes a non-degenerate symplectic space over $\F_p$ with the symplectic form $(\overline{x}, \overline{y}) = a$, where $[x,y] = z^a$ and $0 \le a < p$.  The symplectic group $\Sp(2n, \F_p)$ acts on $G/Z(G)$, preserving the given symplectic form.  Let $T \in \Sp(2n, \F_p)$ and let $A = (a_{ij})$ be the matrix of $T$ relative to the basis $\{\overline{x}_i\}$ (with $0 \le a_{ij} < p$).  Each element $x \in G$ can be uniquely expressed as $x = \left( \prod_{i=1}^{2n}{x_i^{a_i}} \right) z^c$ with $0 \le a_i, c < p$.  Define $\phi : G \to G$ by
\[
\phi(x) = \left[ \prod_{i=1}^{2n} \left( \prod_{j=1}^{2n}{x_j^{a_{ij}}} \right)^{a_i} \right] z^c.
\]
Then $\phi$ induces $T$ on $G/Z(G)$, and $\phi$ is an automorphism of $G$ if and only if $T$ is in the subgroup of $\Sp(2n, \F_p)$ to which $H/I$ is isomorphic (as given in statement 3, 4, 5, or 6 of Theorem~\ref{spthm}, depending on $p$ and the isomorphism class of $G$).

Note that the set of automorphisms $\phi$ does not necessarily constitute a subgroup of $H$, and so it is not obvious that $H$ splits over $I$ (and Winter does not address this issue).  However, as Griess proves in~\cite{gri}, when $p = 2$, $H$ splits if $n \le 2$ and does not split if $n \ge 3$.  Griess also states, but does not prove, that when $p$ is odd, $H$ always splits over $I$.  This observation is also made in, and can be deduced from, Isaacs~\cite{isa2} and~\cite{isa3} and Glasby and Howlett~\cite{gh}.

A short exposition of this proof when $p$ is odd and $G$ has exponent $p$ was communicated via the group-pub-forum mailing list by Isaacs~\cite{isa}.  Let $J/I$ be the central involution of the symplectic group $H/I$, and let $P$ be a Sylow 2-subgroup of $J$.  Then $J$ is normal in $H$, $|P| = 2$, and the non-identity element of $P$ acts on $I$ by sending each element to its inverse.  Then $1 = C_I(P) = I \cap N_H(I)$.  On the other hand, by the Frattini argument, $H = J N_H(P)$, and since $P \le J \cap N_H(P)$, it follows that $H = I N_H(P)$.  But this means that $N_H(P)$ is a complement of $I$ in $H$, and so $H$ splits over $I$.  According to Griess~\cite{gri}, the proof when $G$ has exponent $p^2$ is more technical.

\subsection{Maximal Unipotent Subgroups of a Chevalley Group}

Associated to any simple Lie algebra $\mathcal{L}$ over $\C$ and any field $K$ is the \emph{Chevalley group $G$ of type $\mathcal{L}$ over $K$}.  Table~\ref{chevtab} lists the Chevalley groups of types $A_{\ell}$, $B_{\ell}$, $C_{\ell}$, and $D_{\ell}$ over the finite field $\F_q$, as given in Carter~\cite{car}.  A few clarifications are necessary: the entry for type $B_{\ell}$ requires that $\F_q$ have odd characteristic; $O_{2\ell+1}(\F_q)$ is the orthogonal group which leaves the quadratic form $\xi_1 \xi_2 + \xi_3 \xi_4 + \cdots + \xi_{2\ell-1} \xi_{2\ell} + \xi_{2\ell+1}^2$ invariant over $\F_q$; and $O_{2\ell}(\F_q)$ is the orthogonal group which leaves the quadratic form $\xi_1 \xi_2 + \xi_3 \xi_4 + \cdots + \xi_{2\ell-1} \xi_{2\ell}$ invariant over $\F_q$.  Gibbs~\cite{gib} examines the automorphisms of a maximal unipotent subgroup of a Chevalley group over a field of characteristic not two or three.  We are only interested in finite groups, so from now on we will let $K = \F_q$, where $\F_q$ has characteristic $p > 3$ and $q = p^n$.  In this case, the maximal unipotent subgroups are the Sylow $p$-subgroups of the Chevalley group.  After some preliminaries on maximal unipotent subgroups, we will present his results.

\begin{table}
\begin{center}
\begin{tabular}{|c|c|}
\hline
Type & Chevalley Group \\
\hline
$A_{\ell}$ & $\mathrm{PSL}_{\ell+1}(\F_q)$ \\
$B_{\ell}$ & $\mathrm{P(O_{2\ell+1}'(\F_q))}$ \\
$C_{\ell}$ & $\mathrm{PSp_{2\ell}(\F_q)}$ \\
$D_{\ell}$ & $\mathrm{P(O_{2\ell}'(\F_q))}$ \\
\hline
\end{tabular} \\
\parbox{3in}{\caption{\label{chevtab}The Chevalley groups of types $A_{\ell}$, $B_{\ell}$, $C_{\ell}$ and $D_{\ell}$.}}
\end{center}
\end{table}

Let $\Sigma$, $\Sigma^+$, and $\pi$ denote the sets of roots, positive roots, and fundamental roots, respectively, of $\mathcal{L}$ relative to some Cartan subalgebra.  Then the Chevalley group $G$ is generated by $\{x_r(t) \; : \; r \in \Sigma, t \in \F_q \}$.  One maximal unipotent subgroup $U$ of $G$ is constructed as follows.  As a set,
\[
U = \{x_r(t) \; : \; r \in \Sigma^+, t \in \F_q \}.
\]
For any $r, s \in \Sigma^+$ and $t, u \in \F_q$, the multiplication in $U$ is given by
\begin{eqnarray*}
x_r(t) x_r(u) &=& x_r(t+u) \\
\left[x_s(u), x_r(t)\right] &=& \left\{
  \begin{array}{c@{\quad:\quad}l}
  1 & r+s \textrm{ is not a root} \\
  \prod\limits_{ir+js \in \Sigma}{x_{ir+js}(C_{ij,rs}(-t)^iu^j)} & r+s \textrm{ is a root}. \\
  \end{array} \right.
\end{eqnarray*}
Here $i$ and $j$ are positive integers and $C_{ij,rs}$ are certain integers which depend on $\mathcal{L}$.  The order of $U$ is $q^N$, where $N = |\Sigma^+|$, and $U$ is a Sylow $p$-subgroup of $G$.

Gibbs~\cite{gib} shows that $\Aut(G)$ is generated by six types of automorphisms, namely graph automorphisms, diagonal automorphisms, field automorphisms, central automorphisms, extremal automorphisms, and inner automorphisms.  Let the subgroup of $\Aut(G)$ generated by each type of automorphism be denoted by $P$, $D$, $F$, $C$, $E$, and $I$ respectively.  Let $P_r$ be the additive group generated by the roots of $\mathcal{L}$ and let $r_N$ be the highest root.  Label the fundamental roots $r_1, r_2, \dots, r_{\ell}$.

\begin{enumerate}
\item \emph{Graph Automorphisms}:  An automorphism $\sigma$ of $P_r$ that permutes both  $\pi$ and $\Sigma$ induces a graph automorphism of $U$ by sending $x_r(t)$ to $x_{\sigma(r)}(t)$ for all $r \in \pi$ and $t \in \F_q$.  Graph automorphisms correspond to automorphisms of the Dynkin diagram, and so types $A_{\ell} \; (\ell > 1)$, $D_{\ell} \; (\ell > 4)$, and $E_6$ have a graph automorphism of order 2, while the graph automorphisms in type $D_4$ form a group isomorphic to $S_3$.
\item \emph{Diagonal Automorphisms}:  Every character $\chi$ of $P_r$ with values in $\F_q^{\ast}$ induces a diagonal automorphism which maps $x_r(t)$ to $x_r(\chi(r)t)$ for all $r \in \Sigma^+$ and $t \in \F_q$.
\item \emph{Field Automorphisms}:  Every automorphism $\sigma$ of $\F_q$ induces a field automorphism of $U$ which maps $x_r(t)$ to $x_r(\sigma(t))$ for all $r \in \Sigma^+$ and $t \in \F_q$.
\item \emph{Central Automorphisms}:  Let $\sigma_i$ be endomorphisms of $\F_q^+$.  These induce a central automorphism that maps $x_{r_i}(t)$ to $x_{r_i}(t) x_{r_N}(\sigma_i(t))$ for $i = 1, \dots, \ell$ and all $t \in \F_q$.
\item \emph{Extremal Automorphisms}:  Suppose $r_j$ is a fundamental root such that $r_N - r_i$ is also a root.  Let $u \in \F_q^{\ast}$.  This determines an extremal automorphism which acts trivially on $x_{r_i}(t)$ for $i \neq j$ and sends $x_{r_j}(t)$ to
\[
x_{r_j}(t) x_{r_N-r_j}(ut) x_{r_N}((1/2)N_{r_N-r_j,r_j}ut^2).
\]
Here $N_{r_N-r_j, r_j}$ is a certain constant that depends on the type.  In type $C_{\ell}$, $r_N - 2r_j$ is also a root, and the map that acts trivially on $x_i(t)$ for $i \neq j$ and sends $x_{r_j}(t)$ to
\[
x_{r_j}(t) x_{r_N-2r_j}(ut) x_{r_N-r_j}((1/2)N_{r_N-2r_j,r_j}ut^2) x_{r_N}((1/3)C_{12,r_N-r_j,r_j}ut^3)
\]
is also an automorphism of $U$.
\end{enumerate}

Steinberg~\cite{ste} showed that the automorphism group of a Chevalley group over a finite field is generated by graph, diagonal, field, and inner automorphisms, which shows that $P$, $D$, and $F$ are, in fact, subgroups of $\Aut(U)$.  It is easy to see that the central automorphisms are automorphisms, and a quick computation verifies this for the extremal automorphisms as well.  Note that by multiplying an extremal automorphism by a judicious choice of central automorphism, the $x_{r_N}(\cdot)$ term in the description of the extremal automorphisms disappears.  Therefore, it is legitimate to omit the $x_{r_N}(\cdot)$ term in the definition of an extremal automorphism, and this is what we will use for what follows.

Gibbs does not compute the precise structure of $\Aut(U)$.  This has been done in type $A_{\ell}$, however, for all characteristics; Pavlov~\cite{pav} computes $\Aut(U)$ over $\F_p$, while Weir~\cite{wei} computes it over $\F_q$ (although his computations contain a mistake which we will address in a moment).  We will present the result for type $A_{\ell}$ as explicitly as possible, pausing to note that it does seem feasible to compute the structure of $\Aut(U)$ for other types in a similar manner.

As mentioned before, in type $A_{\ell}$, we can view $U$ as the set of $(\ell+1) \times (\ell+1)$ upper triangular matrices with ones on the diagonal and arbitrary entries from $\F_q$ above the diagonal.  There are $\binom{\ell+1}{2}$ positive roots in type $A_{\ell}$, given by $r_i + r_{i+1} + \cdots + r_j$ for $1 \le i \le j \le \ell$.  Let $E_{i,j}$ be the $(\ell+1) \times (\ell+1)$ matrix with a $1$ in the $(i,j)$-entry and zeroes elsewhere.  Then $x_r(t) = I + tE_{i,j}$, where $r = r_i + r_{i+1} + \cdots + r_j$.  In particular, $x_{r_i}(t) = I + tE_{i,i+1}$.  In type $A_{\ell}$, some of the given types of automorphisms admit simpler descriptions; we will be content to describe their action on the elements $x_{r_i}(t)$.

As mentioned, in type $A_{\ell}$ there is one nontrivial graph automorphism of order $2$, and it acts by sending $x_{r_i}(t)$ to $x_{r_{n+1-i}}(t)$.  The diagonal automorphisms correspond to selecting $\chi_1, \dots, \chi_n \in \F_q^{\ast}$ and mapping $x_{r_i}(t)$ to $x_{r_i}(\chi_i t)$.  This is equivalent to conjugation by a diagonal matrix of determinant 1.  The diagonal automorphisms form an elementary abelian subgroup of order $(p-1)^{n-1}$.  The field automorphisms of $\F_q$ are generated by the Frobenius automorphism and form a cyclic subgroup of order $n$.

Let $a_1, \dots, a_n$ generate the additive group of $\F_q$.  Then the elements $x_{r_i}(a_j)$ for $i,j=1,\dots,n$ generate $U$.  The central automorphisms are generated by the automorphisms $\tau_j^m$ which send $x_{r_j}(a_m)$ to $x_{r_j}(a_m) x_{r_N}(1)$ and fix $x_{r_i}(a_k)$ for $i \neq j$ and $k \neq m$, where $j = 2, \dots, \ell-1$ and $m = 1, \dots, n$.  (When $j = 1$ or $j = \ell$, this automorphism is inner.)  The extremal automorphisms are generated by the automorphism that sends $x_{r_1}(t)$ to $x_{r_1}(t) x_{r_N-r_1}(t)$ and fixes $x_{r_i}(t)$ for $2 \le i \le \ell$ and the automorphism that sends $x_{r_{\ell}}(t)$ to $x_{r_{\ell}}(t) x_{r_N-r_{\ell}}(t)$ and fixes $x_{r_i}(t)$ for $1 \le i \le \ell-1$.  Finally the inner automorphism group is, of course, isomorphic to $U/Z(U)$ (the center of $U$ is generated by $x_{r_N}(t)$).

It is not hard to use these descriptions to deduce that
\[
\Aut(U) \cong ((I \rtimes (E \times C)) \rtimes (D \rtimes F)) \rtimes P.
\]
Furthermore $E \times C$ is elementary abelian of order $q^{n(\ell-2)+2}$, $D$ is elementary abelian of order $(q-1)^{\ell}$, $F$ is cyclic of order $n$, and $I$ has order $q^{(\ell^2+\ell-2)/2}$.  It follows that the order of $\Aut(U)$ is
\[
2 n (q-1)^{\ell} q^{(\ell^2+\ell+2n\ell-4n+2)/2}.
\]

The error in Weir's paper~\cite{wei} stems from his claim that any $g \in \GL_n(\F_p)$ acting on $\F_q$ induces an automorphism of $U$ that maps $x_{r_i}(a_k)$ to $x_{r_i}(g(a_k))$, generalizing the field automorphisms.  However, it is clear that $g$ must, in fact, be a field automorphism, as for any $t, u \in \F_q$,
\[
\left[ x_{r_1}(t), x_{r_{2}}(u) \right] = x_{r_1+r_2}(tu) = \left[ x_{r_1}(tu), x_{r_2}(1) \right].
\]
Applying $g$ to all terms shows that $g(tu) = g(t)g(u)$.

\subsection{Sylow $p$-Subgroups of the Symmetric Group}

The automorphism groups of Sylow $p$-subgroups of the symmetric group for $p > 2$ were examined independently by Bondarchuk~\cite{bon} and Lentoudis~\cite{len, len2, len3, lt}.  Their results are reasonably technical.  They do show that the order of the automorphism group of the Sylow $p$-subgroup of $S_{p^m}$ is
\[
(p-1)^m p^{n(m)},
\]
where
\[
n(m) = p^{m-1} + p^{m-2} + \cdots + p^2 + \frac{1}{2} (m^2-m+2)p - 1.
\]
Note that the Sylow-$p$ subgroup of $S_{p^m}$ is isomorphic to the $m$-fold iterated wreath product of $C_p$.

\subsection{$p$-Groups of Maximal Class}

A $p$-group of order $p^n$ is of \emph{maximal class} if it has nilpotence class $n-1$.  Many examples are given in~\cite[Examples 3.1.5]{lm}, with the most familiar being the dihedral and quaternion groups of order $2^n$ when $n \ge 4$.  It is not too hard to prove some basic results about the automorphism group of an arbitrary $p$-group of maximal class.  Our presentation follows Baartmans and Woeppel~\cite[Section 1]{bw}.

\begin{thm}
Let $G$ be a $p$-group of maximal class of order $p^n$, where $n \ge 4$ and $p$ is odd.  Then $\Aut(G)$ has a normal Sylow $p$-subgroup $P$ and $P$ has a $p'$-complement $H$, so that $\Aut(G) \cong H \rtimes P$.  Furthermore, $H$ is isomorphic to a subgroup of $C_{p-1} \times C_{p-1}$.
\end{thm}

The proof of this theorem begins by observing that $G$ has a characteristic cyclic series $G = G_0 \rhd G_1 \rhd \cdots \rhd G_n = 1$; that is, each $G_i$ is characteristic and $G_i/G_{i+1}$ is cyclic (see Huppert~\cite[Lemmas 14.2 and 14.4]{hup}).  By a result of Durbin and McDonald~\cite{dm}, $\Aut(G)$ is supersolvable and so has a normal Sylow $p$-subgroup $P$ with $p'$-complement $H$, and the exponent of $\Aut(G)$ divides $p^t (p-1)$ for some $t > 0$.  The additional result about the structure of $H$ comes from examining the actions of $H$ on the characteristic cyclic series and on $G/\Phi(G)$.  Baartmans and Woeppel remark that the above theorem holds for any finite $p$-group $G$ with a characteristic cyclic series.

Baartmans and Woeppel~\cite{bw} follow up these general results by focusing on automorphisms of $p$-groups of maximal class of exponent $p$ with a maximal subgroup which is abelian.  More specifically, the characteristic cyclic series can be taken to be a composition series, in which case $G_i = \gamma_i(G)$ for $i \ge 2$ and $G_1 = C_G(G_2/G_4)$.   Baartmans and Woeppel assume that $G_2$ is abelian.

In this case, they show by construction that $H \cong C_{p-1} \times C_{p-1}$.  Furthermore, $P$ is metabelian of nilpotence class $n-2$ and order $p^{2n-3}$.  (Recall that a \emph{metabelian group} is a group whose commutator subgroup is abelian.)  The group $\Inn(G)$ has order $p^{n-1}$ and maximal class $n-2$.  The commutator subgroup $P'$ is the subgroup of $\Inn(G)$ induced by $G_2$.  Baartmans and Woeppel do explicitly describe the automorphisms of $G$, but the descriptions are too complicated to include here.

Other authors who investigate automorphisms of certain finite $p$-groups of maximal class include: Abbasi~\cite{abb}; Miech~\cite{mie}, who focuses on metabelian groups of maximal class; and Wolf~\cite{wol}, who looks at the centralizer of $Z(G)$ in certain subgroups of $\Aut(G)$.

Finally, in~\cite{juh}, Juh\'{a}sz considers more general $p$-groups than $p$-groups of maximal class.  Specifically, he looks at $p$-groups $G$ of nilpotence class $n-1$ in which $\gamma_1(G)/\gamma_2(G) \cong C_{p^m} \times C_{p^m}$ and $\gamma_i(G)/\gamma_{i+1}(G) \cong C_{p^m}$ for $2 \le i \le n-1$.  He refers to such groups as being of type $(n,m)$.  Groups of type $(n,1)$ are the $p$-groups of maximal class of order $p^n$.

Assume that $n \ge 4$ and $p > 2$.  As with groups of maximal class, the automorphism group of a group $G$ of type $(n,m)$ is a semi-direct product of a normal Sylow $p$-subgroup $P$ and its $p'$-complement $H$, and $H$ is isomorphic to a subgroup of $C_{p-1} \times C_{p-1}$.  Juh\'{a}sz' results are largely technical, dealing with the structure of $P$, especially when $G$ is metabelian.

\subsection{Stem Covers of an Elementary Abelian $p$-Group}
\label{stemcoversec}

In~\cite{web2}, Webb looks at the automorphism groups of stem covers of elementary abelian $p$-groups.  We start with some preliminaries on stem covers.  A group $G$ is a \emph{central extension} of $Q$ by $N$ if $N$ is a normal subgroup of $G$ lying in $Z(G)$ and $G/N \cong Q$.  If $N$ lies in $[G, G]$ as well, then $G$ is a \emph{stem extension} of $Q$.  The Schur multiplier $M(Q)$ of $Q$ is defined to be the second cohomology group $H^2(Q, \C^{\ast})$, and it turns out that $N$ is isomorphic to a subgroup of $M(Q)$.  Alternatively, $M(Q)$ can be defined as the maximum group $N$ so that there exists a stem extension of $Q$ by $N$.  Such a stem extension is called a $\emph{stem cover}$.

Webb takes $Q$ to be elementary abelian of order $p^n$ with $p$ odd and $n \ge 2$.  Let $G$ be a stem cover of $Q$.  Then $N = Z(G) = [G,G] = M(Q) \cong Q \wedge Q$ and has order $p^{\binom{n}{2}}$.  Therefore $\Aut_c(G)$ are the automorphisms of $G$ which act trivially on $G/N \cong Q$.  Each automorphism $\alpha \in \Aut_c(G)$ corresponds uniquely to a homomorphism $\overline{\alpha} \in \Hom(Q, N)$ via the relationship $\overline{\alpha}(g N) = g^{-1} \cdot \alpha(g)$ for all $g \in G$.  Of course, $\Hom(Q, N)$ is an elementary abelian $p$-group of order $n \binom{n}{2}$, and so $\Aut(G)$ is an extension of a subgroup of $\Aut(Q) \cong \GL(n, \F_p)$ by an elementary abelian $p$-group of order $n \binom{n}{2}$.  Webb proves that the subgroup of $\Aut(Q)$ in question is usually trivial, leading to her main theorem.

\begin{thm}[Webb~\cite{web2}]
Let $G$ be elementary abelian of order $p^n$ with $p$ odd.  As $n \to \infty$, the proportion of stem covers of $G$ with elementary abelian automorphism group of order $p^{n \binom{n}{2}}$ tends to 1.
\end{thm}

\section{Quotients of Automorphism Groups}
\label{constructionsec}

Not every finite $p$-group is the automorphism group of a finite $p$-group.  A recent paper in this vein is by Cutolo, Smith, and Wiegold~\cite{csw}, who show that the only $p$-group of maximal class which is the automorphism group of a finite $p$-group is $D_8$.  But there are several extant results which show that certain quotients of the automorphism group can be arbitrary.

\subsection{The Quotient $\Aut(G)/\Aut_c(G)$}

\begin{thm}[Heineken and Liebeck~\cite{hl2}]
\label{hlthm}
Let $K$ be a finite group and let $p$ be an odd prime.  There exists a finite $p$-group $G$ of class $2$ and exponent $p^2$ such that $\Aut(G)/\Aut_c(G) \cong K$.
\end{thm}

The construction given by Heineken and Liebeck can be described rather easily.  Let $K$ be a group generated by elements $x_1, x_2, \dots, x_d$.  Let $D'(K)$ be the directed Cayley graph of $K$ relative to the given generators.  Form a new digraph $D(K)$ by replacing every arc in $D'(K)$ by a directed path of length $i$ if the original arc corresponded to the generator $x_i$.  Then $\Aut(D(K)) = K$.

Let $v_1, v_2, \dots, v_m$ be the vertices of $D(K)$.  Let $G$ be the $p$-group generated by elements $v_1, v_2, \dots, v_m$ where
\begin{enumerate}
\item $G'$ is the elementary abelian $p$-group freely generated by
\[
\{ [v_i, v_j] \; : \; 1 \le i < j \le m \}.
\]
\item For each vertex $v_i$, if $v_i$ has outgoing arcs to $v_{i_1}, v_{i_2}, \dots, v_{i_k}$, then
\[
v_i^p = [v_i, v_{i_1} \cdots v_{i_k}].
\]
\end{enumerate}
Heineken and Liebeck show that $\Aut(G)/\Aut_c(G) \cong K$ when $|K| \ge 5$.  (They give a special construction for $|K| < 5$.)  As Webb~\cite{web2} notes, $G$ is a special $p$-group.

They are actually able to determine the automorphism group of $G$ much more precisely, at least when $|K| \ge 5$.  Let the vertices of $D'(K)$ be called \emph{group-points}; they are naturally identified with vertices of $D(K)$.  Let $S$ be the set of vertices of $D(K)$ consisting of the group-point $e$ corresponding to the identity of $K$ and all vertices that can be reached along a directed path from $e$ that does not pass through any other group-points.  Assume that the vertices of $D(K)$ are labeled so that $v_1, \dots, v_s$ are the elements of $S$.  The central automorphisms which fix $v_{s+1}, v_{s+2}, \dots, v_m$ generate an elementary abelian $p$-group $U$ of dimension $s |G'| = (1/2)ks^2(ks-1)$.  Every central automorphism of $G$ is of the form $\prod_{v \in K}{v^{-1} \alpha_v v}$, where the elements $\alpha_v \in U$ and $v \in K$ are uniquely determined.  Thus $\Aut_c(G)$ is the direct product of the conjugates of $U$ in $\Aut(G)$ and $\Aut(G) = U \wr K$.  It follows that $\Aut(G)$ has order $k p^{\ell}$, where $\ell = (1/2)k^2 s^2(ks-1)$ and $s = (1/2)d(d+1) + 1$ when $d \ge 2$ and $s = 1$ when $d = 1$.

Lawton~\cite{law} modifies Heineken and Liebeck's techniques to construct smaller groups $G$ with $\Aut(G)/\Aut_c(G) \cong K$.  He uses undirected graphs which are much smaller, and the $p$-groups he defines is significantly simpler.

Webb~\cite{web} uses similar, though more complicated techniques, to obtain further results.  She defines a class of graphs called \emph{$Z$-graphs}; it turns out that almost all finite graphs are $Z$-graphs (that is, the proportion of graphs on $n$ vertices which are $Z$-graphs goes to 1 as $n$ goes to infinity).  To each $Z$-graph $\Lambda$, Webb associates a special $p$-group $G$ for which $\Aut(G)/\Aut_c(G) \cong \Aut(\Lambda)$.  The set of all special $p$-groups that arise from $Z$-graphs on $n$ vertices is denoted by $\mathcal{G}(p, n)$.

\begin{thm}[Webb]
Let $p$ be any prime.  Then the proportion of groups in $\mathcal{G}_{p,n}$ whose automorphism group is $(C_p)^r$, where $r = n^2(n-1)/2$, goes to 1 as $n \to \infty$.
\end{thm}

The reason the group $(C_p)^r$ arises as the automorphism group is that for $G \in \mathcal{G}(p, n)$, $\Aut_c(G)$ is isomorphic to $\Hom(G/Z(G), Z(G))$, and hence it is isomorphic to $(C_p)^r$.  Webb then shows that $\Aut(G)/\Aut_c(G)$ is usually trivial.

\begin{thm}[Webb]
\label{webbthm}
Let $K$ be a finite group which is not cyclic of order five or less.  Then for any prime $p$, there is a special $p$-group $G \in \mathcal{G}(p, 2|K|)$ with $\Aut(G)/\Aut_c(G) \cong K$.
\end{thm}

In particular, Theorem~\ref{webbthm} extends Heineken and Liebeck's result to the case $p = 2$.  Note that in Theorems~\ref{hlthm} and~\ref{webbthm}, the constructed groups are special and $\Aut_c(G) = \Aut_f(G)$, so that these theorems also prescribe $\Aut(G)/\Aut_f(G)$.  The $p = 2$ analogue of Heineken and Liebeck's result was discussed by Hughes~\cite{hug}.

\subsection{The Quotient $\Aut(G)/\Aut_f(G)$}

Bryant and Kov\'{a}cs~\cite{bk} look at prescribing the quotient $\Aut(G)/\Aut_f(G)$, taking a different approach from Heineken and Liebeck in that they assign $\Aut(G)/\Aut_f(G)$ as a linear group (and they do not bound the class of $G$).

\begin{thm}[Bryant and Kov\'{a}cs~\cite{bk}]
Let $p$ be any prime.  Let $K$ be a finite group with dimension $d \ge 2$ as a linear group over $\F_p$.  Then there exists a finite $p$-group $G$ such that $\Aut(G)/\Aut_f(G) \cong K$ and $d(G) = d$.
\end{thm}

This theorem is non-constructive, in contrast to the results of Heineken and Liebeck.  To understand the main idea, let $F$ be the free group of rank $d$.  By Theorems~\ref{isoenum1} and~\ref{isoenum2}, if $U$ is a normal subgroup of $F$ with $F_{n+1} \le U \le F_n$, then $G = F/U$ is a finite $p$-group and $\Aut(G)/\Aut_f(G)$ is isomorphic to the normalizer of $U$ in $\GL(d,\F_p)$.  Bryant and Kov\'{a}cs show that if $n$ is large enough, then $F_n/F_{n+1}$ contains a regular $\F_p \GL(d, \F_p)$-module, which shows that any subgroup $K$ of $\GL(d,\F_p)$ occurs as the normalizer of some normal subgroup $U$ of $F$ with $F_n \le U \le F_{n+1}$.

\section{Orders of Automorphism Groups}
\label{ordersec}

The first two subsections in this section describe some general theorems about the orders of automorphism groups of finite $p$-groups.  The third subsection gives the order of the automorphism group of an abelian $p$-group, and the last subsection offers many explicit examples of $p$-groups whose automorphism group is a $p$-group.  As proved in Chapters~\ref{c_main} through~\ref{c_submodules}, in some asymptotic senses, the automorphism group of a finite $p$-group is almost always a $p$-group.  However, as mentioned in Section~\ref{summarysec}, the answer to the following question is unknown.

\begin{question}
Let $w_{p,n}$ be the proportion of $p$-groups with order at most $p^n$ whose automorphism group is a $p$-group.  Is it true that $\lim_{n \to \infty}{w_{p,n}} = 1$?
\end{question}

\subsection{Nilpotent Automorphism Groups}

In~\cite{yin}, Ying states two results about the occurrence of automorphism groups of $p$-groups which are $p$-groups, the second being a generalization of a result of Heineken and Liebeck~\cite{hl}.

\begin{thm}
If $G$ is a finite $p$-group and $\Aut(G)$ is nilpotent, then either $G$ is cyclic or $\Aut(G)$ is a $p$-group.
\end{thm}

\begin{thm}
Let $p$ be an odd prime and let $G$ be a finite $p$-group generated by two elements and with cyclic commutator subgroup.  Then $\Aut(G)$ is not a $p$-group if and only if $G$ is the semi-direct product of an abelian subgroup by a cyclic subgroup.
\end{thm}

Heineken and Liebeck~\cite{hl} also have a criterion which determines whether or not a $p$-group of class 2 and generated by two elements has an automorphism of order 2 or if the automorphism group is a $p$-group.  If $p$ is an odd prime and $G$ is a $p$-group that admits an automorphism which inverts some non-trivial element of $G$, then $G$ is an \emph{s.i. group} (a some-inversion group).  Clearly if $G$ is an s.i. group, it has an automorphism of order 2.  If $G$ is not an s.i. group, it is called an \emph{n.i. group} (a no-inversion group).

\begin{thm}
Let $p$ be an odd prime and let $G$ be a $p$-group of class 2 generated by two elements.  Choose generators $x$ and $y$ such that
\[
\left< x, G' \right> \cap \left< y, G' \right> = G',
\]
and suppose that
\[
\left< x \right> \cap G' = \left< x^{p^m} \right> \textrm{ and } \left< y \right> \cap G' = \left< y^{p^n} \right>.
\]
\begin{enumerate}
\item If either $x^{p^m} = 1$ or $y^{p^n} = 1$, then $G$ is an s.i. group.
\item If $x^{p^m} = [x,y]^{rp^k} \neq 1$ and $y^{p^n} = [x,y]^{sp^l} \neq 1$ with $(r,p) = (s,p) = 1$, and $(n-l+k-m)(k-l)$ is non-negative, then $G$ is an s.i. group.
\item If $k$ and $l$ are defined as in (2) and $(n-l+k-m)(k-l)$ is negative, then $G$ is an n.i. group and its automorphism group is a $p$-group.
\end{enumerate}
\end{thm}

\subsection{Wreath Products}

For any group $G$, let $\pi(G)$ be the set of distinct prime factors of $|G|$.  In~\cite{hor2}, Horo{\v{s}}evski{\u\i} gives the following two theorems on the order of the automorphism group of a wreath product.

\begin{thm}
Let $G$ and $H$ be non-trivial finite groups, and let $G_1$ be a maximal abelian subgroup of $G$ which can be distinguished as a direct factor of $G$.  Then
\[
\pi(\Aut(G \wr H)) = \pi(G) \cup \pi(H) \cup \pi(\Aut(G)) \cup \pi(\Aut(H)) \cup \pi(\Aut(G_1 \wr H)).
\]
\end{thm}

\begin{thm}
Let $P_1, P_2, \dots, P_m$ be non-trivial finite $p$-groups.  Then
\[
\pi(\Aut(P_1 \wr P_2 \wr \cdots \wr P_m)) = \bigcup_{i=1}^m{\pi(\Aut(P_i))} \cup \{p\}.
\]
\end{thm}

Thus given any finite $p$-groups whose automorphism groups are $p$-groups, we can construct infinitely many more by taking iterated wreath products.

\subsection{The Automorphism Group of an Abelian $p$-Group}

Macdonald~\cite[Chapter II, Theorem 1.6]{mac} calculates the order of the automorphism group of an abelian $p$-group using Hall polynomials.

\begin{thm}
Let $G$ be an abelian $p$-group of type $\lambda$.  Then
\[
|\Aut(G)| = p^{|\lambda| + 2n(\lambda)} \prod_{i \ge 1}{\phi_{m_i(\lambda)}(p^{-1})},
\]
where $m_i(\lambda)$ is the number of parts of $\lambda$ equal to $i$, $n(\lambda) = \sum_{i \ge 1}{\binom{\lambda_i'}{2}}$, and $\phi_m(t) = (1-t)(1-t^2) \cdots (1-t^m)$.
\end{thm}

There are a variety of results in the literature on the automorphism groups of abelian $p$-groups, of which we will mention three that are interesting and not so technical.  Morgado~\cite{mor, mor2} proves the following theorem about the splitting of the sequence $1 \to K(G) \to \Aut(G) \to A(G) \to 1$ described in Chapter~\ref{c_series}.

\begin{thm}
Let $G$ be an elementary abelian $p$-group.  Let $K(G)$ be the subgroup of $\Aut(G)$ that acts trivially on $G/\Phi(G)$ and let $A(G)$ be the subgroup of $\Aut(G/\Phi(G))$ induced by the action of $\Aut(G)$ on $G/\Phi(G)$.  If $p \ge 5$, then the exact sequence $1 \to K(G) \to \Aut(G) \to A(G) \to 1$  splits if and only $G$ has type $(p^m, p, \dots, p)$ for some positive integer $m$.  If $p = 2$ or $p = 3$, this condition is sufficient but not necessary.
\end{thm}

In a similar vein, Avi{\~n}{\'o} discusses the splitting of $\Aut(G)$ over a different naturally defined normal subgroup.  A different type of result comes from Abraham~\cite{abr}, who shows that for any integer $n \ge 0$ and for $p \ge 3$, the automorphism group of any abelian $p$-group G contains a unique subgroup which is maximal with respect to being normal and having exponent at most $p^n$.

\subsection{Other $p$-Groups Whose Automorphism Groups are $p$-Groups}

In this subsection, we collect constructions of finite $p$-groups whose automorphism groups are $p$-groups.

The first example of a finite $p$-group whose automorphism group is a $p$-group was given by Miller~\cite{mil}, who constructed a non-abelian group of order 64 with an abelian automorphism group of order 128.  Generalized Miller's construction, Struik~\cite{str} gave the following infinite family of 2-groups whose automorphism groups are abelian $2$-groups:
\begin{eqnarray*}
G &=& \left< a,b,c,d \; : \; a^{2^n} = b^2 = c^2 = d^2 = 1, \right. \\
&& \qquad \left. [a,c] = [a,d] = [b,c] = [c,d] = 1, bab = a^{2^{n-1}}, bdb = cd \right>,
\end{eqnarray*}
where $n \ge 3$.  ($G$ can be expressed as a semi-direct product as well.)  Struik shows that $\Aut(G) \cong (C_2)^6 \times C_{2^{n-2}}$.  (As noted in~\cite{str}, it turns out that Macdonald~\cite[p. 237, Revision Problem \#46]{mac2} asks the reader to show that $\Aut(G)$ is an abelian $2$-group.)  Also, Jamali~\cite{jam} has constructed, for $m \ge 2$ and $n \ge 3$, a non-abelian $n$-generator group of order $2^{2n+m-2}$ with exponent $2^m$ and abelian automorphism group $(C_2)^{n^2} \times C_{2^{m-2}}$.

More examples of 2-groups whose automorphism groups are 2-groups are given by Newman and O'Brien~\cite{no}. As an outgrowth of their computations on $2$-groups of order dividing 128, they present (without proof) three infinite families of 2-groups for which $|G| = |\Aut(G)|$.  They are, for $n \ge 3$,
\begin{enumerate}
\item $C_{2^{n-1}} \times C_2$,
\item $\left< a, b \; : a^{2^{n-1}} = b^2 = 1, a^b = a^{1+2^{n-2}} \right>$, and
\item $\left< a, b, c \; : \; a^{2^{n-2}} = b^2 = c^2 = [b, a] = 1, a^c = a^{1+2^{n-4}}, b^c = ba^{2^{n-3}} \right>$.
\end{enumerate}

Moving on to finite $p$-groups where $p$ is odd, for each $n \ge 2$ Horo{\v{s}}evski{\u\i}~\cite{hor} constructs a $p$-group with nilpotence class $n$ whose automorphism group is a $p$-group, and for each $d \ge 3$ he constructs a $p$-group on $d$ generators for each $d \ge 3$ whose automorphism group is a $p$-group.  (He gives explicit presentations for these groups.)

Curran~\cite{cur2} shows that if $(p-1,3) = 1$, then there is exactly one group of order $p^5$ whose automorphism group is a $p$-group (and it has order $p^6$).  It has the following presentation:
\begin{eqnarray*}
G &=& \left< a,b \; : \; b^p = [a,b]^p = [a,b,b]^p = [a,b,b,b]^p = [a,b,b,b,b] = 1, \right. \\
&& \qquad \left. a^p = [a,b,b,b] = [b,a,b]^{-1} \right>.
\end{eqnarray*}
When $(p-1,3) = 3$, there are no groups of order $p^5$ whose automorphism group is a $p$-group.  However, in this case, there are three groups of order $p^5$ which have no automorphisms of order 2.  Curran also shows that $p^6$ is the smallest order of a $p$-group which can occur as an automorphism group (when $p$ is odd).

Then, in~\cite{cur3}, Curran constructs 3-groups $G$ of order $3^n$ with $n \ge 6$ where $|\Aut(G)| = 3^{n+3}$ and $p$-groups $G$ for certain primes $p > 3$ with $|\Aut(G)| = p|G|$.  The MathSciNet review of~\cite{cur3} remarks that F. Menegazzo notes that for odd $p$ and $n \ge 3$, the automorphism group of
\[
G = \left< a,b \; : \; a^{p^n} = 1, b^{p^n} = a^{p^{n-1}}, a^b = a^{1+p} \right>
\]
has order $p |G|$.

Ban and Yu~\cite{by2} prove the existence of a group $G$ of order $p^n$ with $|\Aut(G)| = p^{n+1}$, for $p > 2$ and $n \ge 6$.  In~\cite{hl}, Heineken and Liebeck construct a $p$-group of order $p^6$ and exponent $p^2$ for each odd prime $p$ which has an automorphism group of order $p^{10}$.

Jonah and Konvisser~\cite{jk} exhibit $p+1$ nonisomorphic groups of order $p^8$ with elementary abelian automorphism group of order $p^{16}$ for each prime $p$. All of these groups have elementary abelian and isomorphic commutator subgroups and commutator quotient groups, and they are nilpotent of class two. All their automorphisms are central.

Malone~\cite{mal} gives more examples of $p$-groups in which all automorphisms are central: for each odd prime $p$, he constructs a nonabelian finite $p$-group $G$ with a nonabelian automorphism group which comprises only central automorphisms.  Moreoever, his proof shows that if $F$ is any nonabelian finite $p$-group with $F' =Z(F)$ and $\Aut_c(F)=\Aut(F)$, then the direct product of $F$ with a cyclic group of order $p$ has the required property for $G$.

Caranti and Scoppola~\cite{cs} show that for every prime $p > 3$, if $n \ge 6$, there is a metabelian $p$-group of maximal class of order $p^n$ which has automorphism group of order $p^{2(n-2)}$, and if $n \ge 7$, there is a metabelian $p$-group of maximal class of order $p^n$ with an automorphism group of order $p^{2(n-2)+1}$.  They also show the existence of non-metabelian $p$-groups ($p > 3$) of maximal class whose automorphism groups have orders $p^7$ and $p^9$.

\chapter{An Application of Automorphisms of $p$-Groups}
\label{c_apps}
\chaptermark{Applications}

Given all of the preceding information on the automorphism groups of finite $p$-groups, we would like to consider the connections between these automorphism groups and topics outside of group theory.  The connection we will explore in this chapter involves a certain Markov chain on a finite $p$-group that has been ``twisted'' by a simple automorphism of the group.  In Section~\ref{randomwalksec}, we bound the convergence rate of this Markov chain.  This generalizes results of Chung, Diaconis, and Graham~\cite{cdg}.

There are two appearances of automorphisms of finite $p$-groups in other contexts that we will mention but will not explore.  One is that if we have a group $G$, a subgroup $A$ of $\Aut(G)$, and a random walk on $G$ driven by a probability measure constant on the $A$-orbits in $G$, then the random walk projects to a Markov chain on the $A$-orbits in $G$.  A classical example of this is the fact that a standard random walk on the hypercube projects to the Ehrenfest urn model.  See Diaconis~\cite[Section 3A.3]{dia} for more information.

Automorphisms of $p$-groups also arise in supercharacter theory, which has been developed recently in the work of Diaconis and Isaacs~\cite{di} and Diaconis and Thiem~\cite{dt}, generalizing results of Andr\'{e}, Yan, and others.  One way to construct a supercharacter theory on a group $G$ uses the action of a subgroup of $\Aut(G)$ on $G$.  Neither the projection construction nor the supercharacter construction involving automorphisms seems to have been explored in general.

\section{A Twisted Markov Chain}
\sectionmark{Markov Chain}
\label{randomwalksec}

We begin this section by reviewing some basic facts about probabilities on groups.  Let $G$ be a finite group and let $P$ and $Q$ be probabilities on $G$.  The \emph{convolution} $P \ast Q$ is the probability on $G$ defined by
\[
(P \ast Q)(g) = \sum_{h \in G}{P(gh^{-1}) Q(h)}.
\]
The \emph{Fourier transform} $\hat{P}$ of $P$ is defined on representations $\rho$ of $G$ by
\[
\hat{P}(\rho) = \sum_{g \in G}{P(g) \rho(g)}.
\]
Convolution and the Fourier transform are related by the equation $\widehat{P \ast Q} = \hat{P} \hat{Q}$.  The \emph{total variation distance} between $P$ and $Q$ is given by
\[
\| P - Q \|_{\mathrm{TV}} = \max_{A \subset G}{|P(A) - Q(A)|} = \frac{1}{2} \sum_{g \in G}{|P(g) - Q(g)|}.
\]
The \emph{Upper Bound Lemma} of Diaconis and Shahshahani~\cite{ds} bounds the total variation distance between $P$ and the uniform distribution $U$ using the Fourier transform:
\[
4 \| P - U \|_{\mathrm{TV}}^2 \le |G| \sum_{g \in G}{|P(g) - U(g)|^2} = \sum_{1 \neq \rho \in \mathrm{Irr}(G)}{\dim(\rho) \cdot \mathrm{Tr}(\hat{P}(\rho) \hat{P}(\rho)^{\ast})},
\]
where ${}^\ast$ denote complex conjugation.  The inequality is an application of the Cauchy-Schwartz inequality, the intermediate term is the chi-square distance between $P$ and $U$, and the equality follows from the Plancherel Theorem.  When $G$ is abelian, the Upper Bound Lemma reduces to
\[
\| P - U \|_{\mathrm{TV}}^2 \le \frac{1}{4} \sum_{1 \neq \rho \in \mathrm{Irr}(G)}{|\hat{P}(\rho)|^2}.
\]
Now we can proceed to discuss certain Markov chains on $G$.  Suppose that $\sigma$ is an automorphism of $G$ and $P$ is a probability on $G$.  Let $X_0$ be the identity element $e$ of $G$, and define random variables $X_1, X_2, \dots$ on $G$ by
\[
X_{n+1} = \sigma(X_n) g_n,
\]
where the $g_n$ are independent random variables on $G$, each with distribution $P$.  Then $\{X_n\}$ is a Markov chain.  Let $P_n$ be the probability distribution induced by $X_n$.  We can express $P_n$ as a simple convolution of probabilities by writing
\begin{equation}
\label{convolution}
X_n = \sigma^{n-1}(g_1) \sigma^{n-2}(g_2) \cdots g_n.
\end{equation}
Then $P_n = \sigma^{n-1}(P) \ast \sigma^{n-2}(P) \ast \cdots \ast P$, where $\sigma^i(P)$ denotes the probability distribution given by $\sigma^i(P)(g) = P(\sigma^{-i}(g))$.  We would like to know bounds on the distance $\| P_n - U \|_{\mathrm{TV}}$; that is, how fast does $P_n$ converge to the uniform distribution?

Various cases of this Markov chain have been examined by several authors.  Chung, Diaconis, and Graham~\cite{cdg} analyze $\{X_n\}$ when $G = C_p$, where $p$ is an odd prime, $\sigma$ is multiplication by 2, and $P(0) = P(1) = P(-1) = 1/3$.    They show that the chain converges in $O(\log{p} \log{\log{p}})$ steps.  Furthermore, although this is the correct rate of convergence for infinitely many $p$, the correct rate of convergence is $O(\log{p})$ time for almost all $p$ (although no infinite sequence of primes $p$ for which this is true is known).  They also discuss some other choices of $\sigma$ and $P$.

The motivation in~\cite{cdg} stems from the theory of pseudorandom number generators.  The Markov chain $\{X_n\}$ on $G = C_p$ when $\sigma$ is the identity automorphism and $P(0) = P(1) = P(-1) = 1/3$ takes $O(p^2)$ steps to converge.  The Markov chain analyzed by Chung, Diaconis, and Graham requires the same number of random bits but converges much faster.  For the general problem, we can view $\{X_n\}$ as a ``twist'' of the Markov chain obtained when $\sigma$ is the identity automorphism.  Intuitively, we might expect that the twisted chain converges faster (or at least as fast).  If so, it would be interesting to know when the convergence of a Markov chain can be sped up by twisting it with an automorphism.

Diaconis and Graham~\cite{dg} analyze $\{X_n\}$ when $G = C_2^d$, $\sigma$ is multiplication by a matrix in a specific conjugacy class of $\GL(d, \F_2)$, and $P$ is the probability on $G$ satisfying $P(1,0,\dots,0) = \theta$ and $P(0,\dots,0) = 1-\theta$ with $0 < \theta < 1$.  They show that the chain converges in $O(d \log{d})$ time.

Asci~\cite{asc} examines the case when $G = C_p^d$ and $\sigma$ is a general element of $\GL(d, \F_p)$, generalizing the work of Chung, Diaconis, and Graham.  Asci shows that the Markov chain converges in $O(p^2 \log{p})$ steps in general, while if $\sigma$ has integer eigenvalues which are all neither 1 nor -1, then $O((\log{p})^2)$ steps suffice.  Also, if $\sigma$ has integer eigenvalues and exactly one eigenvalue has absolute value 1, then $O(p^2)$ steps are necessary and sufficient for convergence.

In this section, we sharpen Asci's results in certain cases.  In the context of the general problem, we consider the group $G = C_p^d$, where $p$ is an odd prime.  The automorphism $\sigma$ will be multiplication by a diagonalizable matrix in $\GL(d,\F_p)$, and $P$ will satisfy $P(0) = 1-q$ and $P(e_k) = P(-e_k) = q/2d$, where $e_k$ is the $k$-th unit vector, $k = 1,2,\dots,d$, and $0 \le q \le 1$.  In Subsection~\ref{upperboundsec}, we prove an upper bound on the convergence rate of $\{X_n\}$.  In Subsection~\ref{lowerboundsec}, we consider a special case of $\{X_n\}$ and show that in this case, the convergence rate is within a constant multiple of the upper bound.  Our methods are largely direct generalizations of the methods of Chung, Diaconis, and Graham used in~\cite{cdg}.

Before moving on the statements and proofs, it should be mentioned that Hildebrand~\cite{hil3, hil1, hil2} has written several papers generalizing the work of Chung, Diaconis, and Graham so that at each step of the Markov chain on $G = C_p$, the automorphism $\sigma$ is chosen independently from a fixed probability distribution on $\Aut(C_p)$.  Among many other results, Hildebrand shows that in all non-trivial cases, the Markov chain converges in $O((\log{p})^2)$ steps, and under more restrictive conditions, the Markov chain converges in $O(\log{p} \log{\log{p}})$ steps.

\subsection{An Upper Bound on the Convergence Rate}
\label{upperboundsec}

For the remainder of this section, let $G = C_p^d$ and define a probability distribution $P$ on $G$ by $P(0) = 1-q$ and $P(e_k) = P(-e_k) = q/2d$, where $e_k$ is the $k$-th unit vector, $k = 1, 2, \dots, d$, and $0 \le q \le 1$.  Suppose $A \in \GL(d, \F_p)$ is diagonalizable (over $\F_p$) and has eigenvalues $a_1, a_2, \dots, a_d$.  Let $\{X_n\}$ be the Markov chain on $G$ given by $X_0 = 0$ and $X_{n+1} = AX_n + g_n$, where the $g_n$ are independent random variables with distribution $P$, and let $P_n$ be the probability distribution on $G$ induced by $X_n$.

The upper bound we derive for the convergence rate of $\{X_n\}$ depends on the simpler Markov chain that is obtained when $d = 1$ and is discussed in~\cite{cdg}.  Define a Markov chain $\{Y_n\}$ on $C_p$ by
\[
Y_{n+1} = b Y_n + g_n \textrm{ and } Y_0 = 0,
\]
where $b$ is a non-zero element of $C_p$ and the $g_n$ are independent random variables with distribution $Q$, where $Q(0) = 1-q$ and $Q(1) = Q(-1) = q/2$.  Write $Q_n$ for the measure induced by $Y_n$.  Write $D_n(b, q)$ for the expression given by the Upper Bound Lemma for $4 \| Q_n - U \|^2_{\mathrm{TV}}$.  The irreducible characters of $C_p$ are $\rho_y(g) = e^{2 \pi i y g/ p}$ for $y = 0, 1, \dots, p-1$.  Therefore,
\begin{eqnarray*}
D_n(b, q)
&=& \sum_{y=1}^{p-1}{\prod_{j=0}^{n-1}{\left| 1 - q + \frac{q}{2} e^{2 \pi i b^j y/p} + \frac{q}{2} e^{-2\pi i b^j y/p} \right|^2}} \\
&=& \sum_{y=1}^{p-1}{\prod_{j=0}^{n-1}{\left( 1 - q + q \cos{\left( \frac{2 \pi i b^j y}{p} \right)} \right)^2}}
\end{eqnarray*}
In particular, when $b = 2$, the proof of~\cite[Theorem 1]{cdg} directly extends to show that
\[
4 \| Q_n - U \|^2_{\mathrm{TV}} \le D_n(2, 1/8d) \le 2 e^{t (1-q)^{2r}} - 2,
\]
where $t = \lceil \log_2{p} \rceil$, $r = 4d (\ln{d} + \ln{t} + c)$, and $n = rt$.  We will apply this result to our chain in Corollary~\ref{ubcor}; analogous results for other $b$ would lead to analogous results for $C_p^d$.  Returning to our chain $\{X_n\}$, we can bound the convergence rate of $\{X_n\}$ using the expressions $D_n(b,q)$.

\begin{thm}
\label{ubthm}
Let $p$ be an odd prime.  Then
\[
4 \| P_n - U \|_{\mathrm{TV}}^2 \le \exp{ \left( D_n(a_1, q/8d) + D_n(a_2, q/8d) + \cdots + D_n(a_d, q/8d) \right) } - 1.
\]
\end{thm}

\begin{proof}
The matrix $A$ is diagonalizable, so we can choose $C \in \GL(d,\F_p)$ so that $B = C A C^{-1}$ is a diagonal matrix (with entries $a_1, a_2, \dots, a_d$).  Define a new Markov chain $\{Z_n\}$ on $G$ by $Z_0 = 0$ and $Z_{n+1} = BZ_n + h_n$, where the $h_n$ are independent random variables with distribution $P$.  From Equation~\ref{convolution}, we see that the random variable $C^{-1} X_n C$ is identically distributed to $Z_n$.  Thus $\|X_n -  U\|_{\mathrm{TV}}$ and $\|Z_n - U\|_{\mathrm{TV}}$ are equal, and we may assume that $A$ is a diagonal matrix.  (Diaconis and Graham~\cite{dg} were the first to observe that the convergence rate of the Markov chain only depends on the conjugacy class of $A$, or the conjugacy class of $\sigma$ in $\Aut(G)$ in the general case.)

The measure $P_n$ is the convolution of the measures $\mu_j$, given by
\[
\mu_j(\pm a_k^j e_k) = \frac{q}{2d}
\]
for $k = 1,\dots,d$ and
\[
\mu_j(0) = 1-q,
\]
for $j = 0,\dots,n-1$.  The Fourier transform of $\mu_j$ is given by
\begin{eqnarray*}
\hat{\mu}_j(y) &=& 1-q + \sum_{k=1}^d{\frac{q}{2d} \left[ \exp{\left( \frac{2\pi i(y \cdot A^j e_k)}{p} \right)} + \exp{\left( \frac{2\pi i(y \cdot -A^j e_k)}{p} \right)} \right]} \\
&=& 1-q + \frac{q}{d} \sum_{k=1}^d{\cos{\left( \frac{2\pi \cdot a_k^j y_k}{p} \right)}}.
\end{eqnarray*}
Let $[d] = \{1, 2, \dots, d\}$.  By the Upper Bound Lemma,
\begin{eqnarray*}
&& 4 \| P_n - U \|^2_{\mathrm{TV}} \\
&\le& \sum_{0 \neq y \in C_p^d}{\prod_{j=0}^{n-1}{\left[ 1-q + \frac{q}{d} \sum_{k=1}^d{\cos{\left( \frac{2\pi \cdot a_k^j y_k}{p} \right)}} \right]^2}} \\
&=& \sum_{I \subset [d]}{\sum_{\{y \; : \; y_i \neq 0 \Leftrightarrow i \in I\}}{\prod_{j=0}^{n-1}{\left[ 1-q + \frac{q(d-|I|)}{d} + \frac{q}{d} \sum_{i \in I}{\cos{ \left( \frac{2\pi \cdot a_i^j y_i}{p} \right) }} \right]^2}}} \\
&=& \sum_{I \subset [d]}{\sum_{\{y \; : \; y_i \neq 0 \Leftrightarrow i \in I\}}{\prod_{j=0}^{n-1}{\left[ 1-q + \frac{q(d-m)}{d} + \frac{q}{d} \sum_{i \in I}{ \left| \cos{ \left( \frac{2\pi \cdot a_i^j y_i}{p} \right)  } \right| } \right]^2}}}.
\end{eqnarray*}
By the arithmetic-geometric mean inequality and the fact that $1 + x \le e^x$ for all $x \ge 0$,
\begin{eqnarray*}
&& 4 \| P_n - U \|^2_{\mathrm{TV}} \\
&\le& \sum_{I \subset [d]}{\sum_{\{y \; : \; y_i \neq 0 \Leftrightarrow i \in I\}}{ \left[ 1-q + \frac{q(d-|I|)}{d} + \frac{q}{dn} \sum_{j=0}^{n-1}{ \sum_{i \in I}{ \left| \cos{ \left( \frac{2\pi \cdot a_i^j y_i}{p} \right)  } \right| }} \right]^{2n} }} \\
&=& \sum_{I \subset [d]}{\sum_{\{y \; : \; y_i \neq 0 \Leftrightarrow i \in I\}}{ \left[ 1 - \frac{q}{dn} \sum_{j=0}^{n-1}{ \sum_{i \in I}{ \left(1 - \left| \cos{ \left( \frac{2\pi \cdot a_i^j y_i}{p} \right)  } \right| \right) }} \right]^{2n} }} \\
&\le& \sum_{I \subset [d]}{\sum_{\{y \; : \; y_i \neq 0 \Leftrightarrow i \in I\}}{ \exp{ \left( -\frac{2q}{d} \sum_{j=0}^{n-1}{ \sum_{i \in I}{ \left(1 - \left| \cos{ \left( \frac{2\pi \cdot a_i^j y_i}{p} \right)  } \right| \right) } } \right)}}} \\
&=& \sum_{I \subset [d]}{\sum_{\{y \; : \; y_i \neq 0 \Leftrightarrow i \in I\}}{ \prod_{i \in I}{\prod_{j=0}^{n-1}{ \exp{ \left( -\frac{2q}{d} + \frac{2q}{d} \left| \cos{ \left( \frac{2\pi \cdot a_i^j y_i}{p} \right)  } \right| \right) } }}}} \\
&=& \sum_{I \subset [d]}{ \prod_{i \in I}{ \sum_{y=1}^{p-1}{\prod_{j=0}^{n-1}{ \exp{ \left( -\frac{2q}{d} + \frac{2q}{d} \left| \cos{ \left( \frac{2\pi \cdot a_i^j y}{p} \right)  } \right| \right) } }}}} \\
&=& \prod_{i=1}^d{\left( 1 + \sum_{y=1}^{p-1}{\prod_{j=0}^{n-1}{ \exp{ \left( -\frac{2q}{d} + \frac{2q}{d} \left| \cos{ \left( \frac{2\pi \cdot a_i^j y}{p} \right) } \right| \right) }}} \right)} - 1 \\
&\le& \exp{ \left( \sum_{i=1}^d{\sum_{y=1}^{p-1}{\prod_{j=0}^{n-1}{ \exp{ \left( -\frac{2q}{d} + \frac{2q}{d} \left| \cos{ \left( \frac{2\pi \cdot a_i^j y}{p} \right) } \right| \right) }}}} \right) } - 1.
\end{eqnarray*}
Finally, the inequalities $e^{-x} \le 1 - x/2$ for $0 \le x \le 1$ and $a - a|\cos{x}| \ge a/4 - (a/4)\cos{(2x)}$ for all $x$ and $0 \le a \le 1$ show that
\begin{eqnarray*}
&& 4 \| P_n - U \|^2_{\mathrm{TV}} \\
&\le& \exp{ \left( \sum_{i=1}^d{\sum_{y=1}^{p-1}{\prod_{j=0}^{n-1}{ \left( 1  -\frac{q}{2d} + \frac{q}{2d} \left| \cos{ \left( \frac{2\pi \cdot a_i^j y}{p} \right) } \right| \right)^2 }} } \right) } - 1 \\
&\le& \exp{ \left( \sum_{i=1}^d{\sum_{y=1}^{p-1}{\prod_{j=0}^{n-1}{ \left( 1  -\frac{q}{8d} + \frac{q}{8d} \cos{ \left( \frac{
2\pi \cdot a_i^j y}{p} \right) } \right)^2 }} } \right) } - 1 \\
&=& \exp{ \left( D_n(a_1, q/8d) + D_n(a_2, q/8d) + \cdots + D_n(a_d, q/8d) \right) } - 1.
\end{eqnarray*}
\end{proof}

\begin{cor}
\label{ubcor}
Let $t = \lceil \log_2{p} \rceil$.  When $A = 2I$ and $q = 1$, if $n = 4dt(\ln{t} + \ln{d} + c)$ for $c > 0$, then $\| P_n - U \|_{\mathrm{TV}}^2 \le 2e^{-c}$.
\end{cor}

\begin{proof}
Let $r = 4d(\ln{t} + \ln{d} + c)$.  By Theorem~\ref{ubthm} and the previous comment bounding $\| Q_n - U \|_{\mathrm{TV}}^2$,
\begin{eqnarray*}
4 \| P_n - U \|_{\mathrm{TV}}^2
&\le& \exp{ \left( 2d(e^{t(1-1/8d)^{2r}} - 1) \right) } - 1 \\
&\le& \exp{ \left( 2d(e^{-\ln{d}-c} - 1) \right) } - 1 \\
&\le& 8e^{-c}.
\end{eqnarray*}
\end{proof}

\subsection{A Lower Bound on the Convergence Rate}
\label{lowerboundsec}

In this subsection, we will show that when $p = 2^t-1$, $A = 2I$, and $q = 1$, the upper bound given by Corollary~\ref{ubcor} for the rate of convergence of the Markov chain $\{X_n\}$ defined in Subsection~\ref{upperboundsec} is correct up to a constant multiple.

\begin{thm}
\label{lowerboundthm}
Suppose $p = 2^t - 1$ is prime.  For the Markov chain $\{X_n\}$ with $A=2I$ and $q = 1$, if
\[
n < \frac{dt (\ln{t}+\ln{d}-1)}{6 \pi^2},
\]
then
\[
\| P_n - U \|^2_{TV} \ge \frac{1}{4}
\]
for large $t$.
\end{thm}

\begin{proof}
Our proof uses the Second Moment Method developed by Wilson; see~\cite{sc} for more details.  For any function $f : C_p^d \to \C$ and any constants $\alpha, \beta > 0$, Chebyshev's inequality says that
\[
\mathrm{Pr}_U \left\{ x \; : \; |f(x) - E_U(f)| \ge \alpha \sqrt{\mathrm{Var}_U(f)} \right\} \le \frac{1}{\alpha^2}
\]
and
\[
\mathrm{Pr}_{P_n} \left\{ x \; : \; |f(x) - E_{P_n}(f)| \ge \beta \sqrt{\mathrm{Var}_U(f)} \right\} \le \frac{1}{\beta^2}.
\]
If $X$ and $Y$ denote the complements of the sets in question and they are disjoint, then
\[
\| P_n - U \|_{\mathrm{TV}} \ge 1 - \frac{1}{\alpha} - \frac{1}{\beta}.
\]
To prove the theorem, it is enough to choose an appropriate function $f : C_p^d \to \C$ so that if $\alpha = \beta = 2$ and $t$ is large enough, then $X$ and $Y$ are disjoint.  Define $f$ to be the function
\[
f(y) = \sum_{i=1}^d{\sum_{j=0}^{t-1}{q^{2^j y}}},
\]
where $q = e^{2 \pi i / p}$.  Let $n = rt$ with $r$ to be chosen later.  Recall that
\[
\widehat{P}_n(z) = \sum_{y \in C_p^d}{P_n(y) q^{y \cdot z}}.
\]
The expectation and variance of $f$ under the uniform distribution $U$ and the distribution $P_n$ on $C_p^d$ are calculated in Lemma~\ref{meanlem}.  Let
\begin{eqnarray*}
\Pi_j &=& \prod_{a = 0}^{t-1}{\left[ \frac{d-1}{d} + \frac{1}{d} \cos{ \left( \frac{2\pi \cdot 2^a(2^j-1)}{p} \right) } \right]} \\
\Gamma_j &=& \prod_{a = 0}^{t-1}{\left[ \frac{d-2}{d} + \frac{1}{d} \cos{ \left( \frac{2\pi \cdot 2^a}{p} \right)} + \frac{1}{d} \cos{ \left( \frac{2\pi \cdot 2^a 2^j}{p} \right)} \right]}.
\end{eqnarray*}
Then, by Lemma~\ref{meanlem},
\begin{eqnarray*}
E_U(f) &=& 0 \\
E_U(f \overline{f}) &=& dt \\
\mathrm{Var}_U(f) &=& dt \\
E_{P_n}(f) &=& dt \Pi_1^r \\
E_{P_n}(f \overline{f}) &=& dt \sum_{j=0}^{t-1}{\Pi_j^r} + t(d^2-d) \sum_{j=0}^{t-1}{\Gamma_j^r} \\
\mathrm{Var}_{P_n}(f) &=& dt \sum_{j=0}^{t-1}{\Pi_j^r} + t(d^2-d) \sum_{j=0}^{t-1}{\Gamma_j^r} - d^2 t^2 \Pi_1^{2r}.
\end{eqnarray*}
Using this information in the bounds obtained by Chebyshev's inequality (with $\alpha = \beta = 2$) gives
\[
\mathrm{Pr}_U \left\{ x \; : \; |f(x)| \ge 2 d^{1/2} t^{1/2} \right\} \le \frac{1}{4}
\]
and
\[
\mathrm{Pr}_{P_n} \left\{ x : |f(x) - dt \Pi_1^r| \ge 2 \sqrt{\mathrm{Var}_{P_n}(f)} \right\} \le \frac{1}{4}.
\]
Choose $r$ to be an even integer of the form
\[
r = \frac{d (\ln{t} + \ln{d})}{-2 \ln{|\Pi_1^d|}} - d\lambda.
\]
Then
\[
dt \Pi_1^r = d^{1/2} t^{1/2} |\Pi_1|^{-d\lambda}
\]
To show that $X$ and $Y$ are disjoint for large enough $t$, it is enough to show that $E_{P_n}(f) - 2 \sqrt{\mathrm{Var}_U(f)} \to \infty$ and $\sqrt{\mathrm{Var}_{P_n}(f)}/E_{P_n}(f) \to 0$ as $t \to \infty$.  The first fact is proved in Lemma~\ref{pi1lem} and the second fact is proved in Lemma~\ref{biglem}.  The theorem follows from this and the bound $-2 \ln{|\Pi_1^d|} < 6 \pi^2$ obtained from Lemma~\ref{pi1lem}.
\end{proof}

\appendix
\chapter{Numerical Estimates for Theorem~\ref{mainthm}}
\label{a_estimates}

The purpose of this section is to prove several estimates needed in Chapters~\ref{c_normal} and~\ref{c_submodules}.  Most of the estimates involve Gaussian coefficients, and so we will begin with the relevant definitions and bounds on the Gaussian coefficients obtained by Wilf~\cite{wil}.

The \emph{Gaussian coefficient} (also called the \emph{$q$-binomial coefficient})
\[
\sbinom{n}{k}_q = \frac{(q^n-1) \cdots (q^n-q^{k-1})}{(q^k-1) \cdots (q^k-q^{k-1})}
\]
is the number of $k$-dimensional subspaces of a vector space of dimension $n$ over $\F_q$.  We shall be concerned with estimates for $\sbinom{n}{k}_q$ and for the \emph{Galois number}
\[
\mathcal{G}_n(q) = \sum_{k=0}^n{\sbinom{n}{k}_q},
\]
which is the total number of subspaces of a vector space of dimension $n$ over $\F_q$.  (A survey of these numbers is given by Goldman and Rota~\cite{gr}.)  First we need a technical lemma.

\begin{lem}
\label{polybound}
Let $q > 1$ and define
\[
C(q) = \sum_{r=-\infty}^{\infty}{q^{-r^2}}.
\]
Let $f(x) = -ax^2 + bx + c$ with $a > 0$.  For any pair of integers $t \le u$, set
\[
A(q) = \sum_{r=t}^u{q^{f(r)}}.
\]
Then $A(q) \le C(q^a) q^{f(y)}$ for some real number $y \in [t, u]$.
\end{lem}

\begin{proof}
Suppose the maximum of $f(x)$ in $[t, u]$ occurs at $x = y$.  The global maximum of $f(x)$ occurs at $x = b/2a$, so one of three cases holds: $b/2a \le y = t$, $t \le y = b/2a \le u$, or $u = y \le b/2a$.  In each case, if $t \le r \le u$, then
\begin{eqnarray*}
&&-a(r-y)^2 - f(r) + f(y) \\
&=& -a(r-y)^2 - (-ar^2+br+c) + (-ay^2+by+c) \\
&=& (2ay-b)(r-y) \\
&\ge& 0.
\end{eqnarray*}
Thus
\begin{eqnarray*}
A(q) &=& q^{f(y)} \sum_{r = t}^u{q^{f(r)-f(y)}} \\
&\le& q^{f(y)} \sum_{r = t}^u{q^{-a(r-y)^2}} \\
&\le& q^{f(y)} \sum_{r = -\infty}^{\infty}{q^{-a(r-y)^2}},
\end{eqnarray*}
and it suffices to show that
\[
g(y) = \sum_{r = -\infty}^{\infty}{s^{-(r-y)^2}} \le g(0),
\]
where $s = q^a$.  To prove this inequality, define the theta function
\[
\theta_3(z, w) = \sum_{r = -\infty}^{\infty}{e^{r^2 \pi i w} e^{2 r i z}},
\]
where $|e^{\pi iw}| < 1$.  Jacobi's functional equation for this function (see Whittaker and Watson~\cite[Section 21.51]{ww}) asserts that
\[
\theta_3(z, w) = \frac{1}{\sqrt{-iw}} e^{z^2/\pi i w} \theta_3(z/w, -1/w),
\]
where $\sqrt{e^{i \theta}}$ denotes $e^{i\theta/2}$ for $0 \le \theta \le 2\pi$.  The function $g(y)$ is related to the theta function as follows:
\begin{eqnarray*}
g(y) &=& s^{-y^2} \sum_{r = -\infty}^{\infty}{s^{-r^2} e^{-2ri(iy \ln{s})}} \\
&=& s^{-y^2} \theta_3(-iy \ln{s}, w),
\end{eqnarray*}
where $\pi i w = -\ln{s}$.  Applying the functional equation leads to
\begin{eqnarray*}
g(y) &=& \frac{s^{-y^2} \sqrt{\pi}}{\sqrt{\ln{s}}} e^{y^2 \ln{s}} \theta_3(-\pi y, -1/w) \\
&=& \sqrt{\frac{\pi}{\ln{s}}} \sum_{r = -\infty}^{r = \infty}{e^{-r^2 \pi^2 / \ln{s}} e^{-2ir\pi y}} \\
&=& \sqrt{\frac{\pi}{\ln{s}}} (1 + 2 \sum_{r=1}^{\infty}{e^{-r^2 \pi^2 / \ln{s}} \cos{2r\pi y}}) \\
&\le& \sqrt{\frac{\pi}{\ln{s}}} (1 + 2 \sum_{r=1}^{\infty}{e^{-r^2 \pi^2 / \ln{s}}}) \\
&=& g(0).
\end{eqnarray*}
This completes the proof.
\end{proof}

For $q > 1$, define
\begin{eqnarray*}
D(q) &=& \prod_{j=1}^{\infty}{(1-q^{-j})^{-1}} \\
S_n(q) &=& \sum_{k=0}^n{q^{k(n-k)}} = q^{n^2/4} \sum_{k=0}^n{q^{-(k-n/2)^2}}.
\end{eqnarray*}

Note that both $C(q)$ and $D(q)$ decrease to 1 as $q \to \infty$.  If $q \ge 2$, then $C(q) \le C(2) < 9/4$ and $D(q) \le D(2) < 7/2$.  The following estimates on Gaussian coefficients and Galois numbers were either obtained by Wilf~\cite{wil} or follow from his work.

\begin{lem} \label{qests}
Let $q$ be a prime power.  Then
\begin{eqnarray}
\sbinom{n}{k}_q &\le& D(q) q^{k(n-k)} \label{coefbounds} \\
D(q) q^{n^2/4-1/4} \left(2 - \frac{9 q^{(1-n)/2}}{2} \right) &\le& \mathcal{G}_n(q) \label{gnbounds} \\
&\le& S_n(q) D(q) \nonumber \\
&\le& C(q) D(q) q^{n^2/4} \nonumber
\end{eqnarray}
\end{lem}

\begin{proof}
Equation~\ref{coefbounds} and $\mathcal{G}_n(q) \le S_n(q) D(q)$ are proved in~\cite{wil}.  The inequality $S_n(q) \le C(q) q^{n^2/4}$ follows from Lemma~\ref{polybound}, taking $f(x) = x(n-x) = -x^2+nx$ and noting that $x(n-x) \le n^2/4$ for all $x$.  This proves $\mathcal{G}_n(q) \le C(q) D(q) q^{n^2/4}$.

The lower bound for $\mathcal{G}_n(q)$ is slightly more complicated, but it is easy to see from~\cite{wil}, Lemma~\ref{polybound}, and the definition of $S_n(q)$ that
\begin{eqnarray*}
\mathcal{G}_n(q) &\ge& S_n(q) - \frac{2S_{n-1}(q)+2q^{-2n}}{q-1} \\
&\ge& 2 q^{n^2/4-1/4} - \frac{2 C(q) q^{(n-1)^2/4}}{q-1} \\
&\ge& q^{n^2/4-1/4} \left( 2 - \frac{2 C(q) q^{(1-n)/2}}{q-1} \right) \\
&\ge& q^{n^2/4-1/4} \left( 2 - \frac{9q^{(1-n)/2}}{2} \right),
\end{eqnarray*}
where the last inequality uses the fact that $2C(q)/(q-1) < 9/2$.
\end{proof}

Next we will find numerical bounds for $d_n$, the rank of $F_n/F_{n+1}$.  These are needed for Lemma~\ref{gaussprods} and Theorems~\ref{limit1cor} and~\ref{limit2cor}.  Recall that
\[
d_n = \sum_{i=1}^n{\frac{1}{i} \sum_{j|i}{\mu(i/j) \cdot d^j}}.
\]

\begin{lem}
\label{dnasymptotics1}
For any positive integer $n$ and $d \ge 5$,
\[
d_n \le \frac{10}{7} \cdot \frac{d^n}{n}.
\]
\end{lem}

\begin{proof}
For any $i$ and $d$, the expression
\[
\sum_{j|i}{\mu(i/j) \cdot d^j}
\]
counts (for example) the number of infinite $d$-ary sequences with (minimum) period $i$.  This is at most $d^i$, the number of infinite $d$-ary sequences whose period divides $i$.  Thus
\[
d_n \le d + \frac{d^2}{2} + \cdots + \frac{d^n}{n}.
\]
We will prove the claim by induction on $n$.  When $n = 1$, this is trivially true.  When $n = 2$,
\begin{eqnarray*}
d_2 &\le& d + \frac{d^2}{2} \\
&=& \left(1 + \frac{2}{d} \right) \frac{d^2}{2} \\
&\le& \frac{7}{5} \cdot \frac{d^2}{2} \\
&\le& \frac{10}{7} \cdot \frac{d^2}{2},
\end{eqnarray*}
and the claim is true.  Now suppose $n > 2$ and assume that
\[
d_{n-1} \le \frac{10}{7} \cdot \frac{d^{n-1}}{n-1}.
\]
Then
\begin{eqnarray*}
d_n &\le& d_{n-1} + \frac{d^n}{n} \\
&\le& \frac{10}{7} \cdot \frac{d^{n-1}}{n-1} + \frac{d^n}{n} \\
&=& \left( 1 + \frac{10n}{7d(n-1)} \right) \frac{d^n}{n} \\
&\le& \left( 1 + \frac{10\cdot3}{7\cdot5\cdot2} \right) \frac{d^n}{n} \\
&=& \frac{10}{7} \cdot \frac{d^n}{n}.
\end{eqnarray*}
The claim holds by induction.
\end{proof}

\begin{lem}
\label{dnasymptotics2}
Suppose $n \ge 3$ and $d \ge 6$ or $n \ge 10$ and $d \ge 5$.  Then
\begin{equation}
\label{dninequality1}
d_n - 4 d_{n-1} - 2 d_{n-2} \ge -\frac{15}{2}
\end{equation}
and
\begin{equation}
\label{dninequality2}
d_n - 2 d_{n-1} - \frac{2}{n-2} \cdot d_{n-2} \ge -1.
\end{equation}
Suppose $n \ge 10$ and $d \ge 5$, or $n \ge 5$ and $d \ge 6$, or $n \ge 4$ and $d \ge 8$, or $n \ge 3$ and $d \ge 17$.  Then
\begin{equation}
\label{dninequality3}
d_n - 4 d_{n-1} - 4d^2 + 11/16 > 0.
\end{equation}
\end{lem}

\begin{proof}
By the definition of $d_n$ and Lemma~\ref{dnasymptotics1},
\begin{eqnarray*}
d_n - 4 d_{n-1} - 2 d_{n-2} &=& \frac{d^n}{n} - 3 d_{n-1} - 2 d_{n-2} \\
&\ge& \frac{d^n}{n} - \frac{30}{7} \cdot \frac{d^{n-1}}{n-1} - \frac{20}{7} \cdot \frac{d^{n-2}}{n-2}.
\end{eqnarray*}
To prove Equation~\ref{dninequality1} for given values of $n$ and $d$, it is certainly sufficient to show that
\[
\frac{d^n}{n} - \frac{30}{7} \cdot \frac{d^{n-1}}{n-1} - \frac{20}{7} \cdot \frac{d^{n-2}}{n-2} \ge 0,
\]
or equivalently that
\begin{equation}
\label{monotoneeqn}
\frac{1}{n} - \frac{30}{7d(n-1)} - \frac{20}{7d(n-2)} \ge 0.
\end{equation}
The left-hand side of this equation is obviously increasing as a function of $d$.  Furthermore,
\begin{eqnarray*}
\frac{1}{n+1} - \frac{30}{7dn} + \frac{20}{7d(n-1)}
&=& \frac{n}{n+1} \cdot \frac{1}{n} - \frac{n-1}{n} \cdot \frac{30}{7d(n-1)} - \frac{n-2}{n-1} \cdot \frac{20}{7d(n-2)} \\
&\ge& \frac{n}{n+1} \cdot \frac{1}{n} - \frac{n}{n+1} \cdot \frac{30}{7d(n-1)} - \frac{n}{n+1} \cdot \frac{20}{7d(n-2)} \\
&\ge& \frac{n}{n+1} \left( \frac{1}{n} - \frac{30}{7d(n-1)} - \frac{20}{7d(n-2)} \right).
\end{eqnarray*}
Thus if Equation~\ref{monotoneeqn} holds for some positive integers $d$ and $n$, it holds for all larger $d$ and $n$.  It turns out that Equation~\ref{monotoneeqn} holds for the following pairs of values: $d=5$ and $n=40$; $d=6$ and $n=6$; $d=7$ and $n=4$, and $d=8$ and $n=3$.  This proves Equation~\ref{dninequality1} for all values of $d$ and $n$ except: $d = 5$ and $10 \le n \le 39$; $d = 6$ and $3 \le n \le 5$; and $d = 7$ and $n = 3$.  A computer check shows that Equation~\ref{dninequality1} in these cases as well, proving the general claim about Equation~\ref{dninequality1}.

Turning to Equation~\ref{dninequality2}, we find
\begin{eqnarray*}
d_n - 2 d_{n-1} - \frac{2}{n-2} \cdot d_{n-2}
&=& \frac{d^n}{n} - d_{n-1} - \frac{2}{n-2} \cdot d_{n-2} \\
&\ge& \frac{d^n}{n} - \frac{10}{7} \cdot \frac{d^{n-1}}{n-1} - \frac{20}{7(n-2)} \cdot \frac{d^{n-2}}{n-2}.
\end{eqnarray*}
To prove Equation~\ref{dninequality2} for given values of $n$ and $d$, it suffices to show that
\[
\frac{d^n}{n} - \frac{10}{7} \cdot \frac{d^{n-1}}{n-1} - \frac{20}{7(n-2)} \cdot \frac{d^{n-2}}{n-2} \ge 0,
\]
or equivalently that
\begin{equation}
\label{monotoneeqn2}
\frac{1}{n} - \frac{10}{7d(n-1)} - \frac{20}{7d^2(n-2)^2} \ge 0.
\end{equation}
As with Equation~\ref{monotoneeqn}, if this equation holds for some positive integers $d$ and $n$, then it holds for all larger $d$ and $n$.  In fact, it holds for $d = 5$ and $n = 10$ and for $d = 6$ and $n = 3$.  This proves Equation~\ref{dninequality2}.

Finally, for Equation~\ref{dninequality3}, we have
\begin{eqnarray*}
d_n - 4 d_{n-1} - 4d^2
&=& \frac{d^n}{n} - 3 d_{n-1} - 4d^2 \\
&\ge& \frac{d^n}{n} - \frac{30}{7} \cdot \frac{d^{n-1}}{n-1} - 4 d^2.
\end{eqnarray*}
To prove Equation~\ref{dninequality3} for given values of $n$ and $d$, it suffices to show that
\[
\frac{d^n}{n} - \frac{30}{7} \cdot \frac{d^{n-1}}{n-1} - 4 d^2 > 0,
\]
or equivalently that
\[
\label{monotoneeqn3}
\frac{1}{n} - \frac{30}{7d(n-1)} - \frac{4}{d^{n-2}} > 0.
\]
As with Equations~\ref{monotoneeqn} and~\ref{monotoneeqn2}, if this equation holds for some positive integers $d$ and $n$, then it holds for all larger $d$ and $n$.  In fact, it holds for the following pairs of values: $d = 5$ and $n = 14$; $d = 6$ and $n = 7$; $d = 7$ and $n = 5$; $d = 10$ and $n = 4$, and $d = 25$ and $n = 3$.  The finitely many cases of Equation~\ref{dninequality3} remaining are true by a computer check.
\end{proof}

The following lemma is needed in Chapter~\ref{c_normal} to bound products of Gaussian coefficients, and we will finish this appendix with Lemma~\ref{quadbound}, which is used in Chapter~\ref{c_submodules}.

\begin{lem} \label{gaussprods}
Fix a prime $p$ and integers $n \ge 3$ and $d \ge 6$ or $n \ge 10$ and $d \ge 5$.  Let $F$ be the free group of rank $d$, and let $d_i$ be the dimension of $F_i/F_{i+1}$ for all $i$.  For $1 \le i \le n-1$ and $0 \le u_i \le d_i$, let
\[
A_i(u_i) = \sum{\prod_{j=i}^{n-1}{p^{-(u_{j+1}-d_{j+1})(u_{j+1}-u_j/2)}}},
\]
where the sum is over all integers $u_{i+1}, \dots, u_n$ such that
\begin{eqnarray*}
0 \le &u_j& \le d_j \qquad \textrm{for } i+1 \le j \le n-2 \\
1 \le &u_{n-1}& \le d_{n-1} \\
2 \le &u_n& \le d_n.
\end{eqnarray*}
Then for $1 \le i \le n-2$,
\[
A_i(u_i) \le C(p^{15/16}) C(p)^{n-i-1} p^{-15/16+d_n^2/4+d_{n-1}-d_n/4} p^{-u_i(d_{i+1}-1)/2}.
\]
\end{lem}

\begin{proof}
First note that
\[
A_{n-1}(u_{n-1}) = \sum_{u_n=2}^{d_n}{p^{-(u_n-d_n)(u_n-u_{n-1}/2)}}.
\]
As a function of $u_n$, the expression $-(u_n-d_n)(u_n-u_{n-1}/2)$ is at most $(d_n-u_{n-1}/2)^2/4$, so that
\[
A_{n-1}(u_{n-1}) \le C(p) p^{(d_n-u_{n-1}/2)^2/4}
\]
by Lemma~\ref{polybound}.  The proof of the theorem is by backward induction on $i$.  Note that
\[
A_i(u_i) = \sum_{u_{i+1}}{p^{-(u_{i+1}-d_{i+1})(u_{i+1}-u_i/2)} A_{i+1}(u_{i+1})}.
\]
When $i = n-2$, using our bound on $A_{n-1}(u_{n-1})$ gives
\begin{eqnarray*}
&& A_{n-2}(u_{n-2}) \\
&\le& C(p) p^{d_n^2/4} \sum_{u_{n-1} = 1}^{d_{n-1}}{p^{u_{n-1}^2/16-u_{n-1}d_n/4 + (d_{n-1}-u_{n-1})(u_{n-1}-u_{n-2}/2)}} \\
&=& C(p) p^{d_n^2/4} \sum_{u_{n-1} = 1}^{d_{n-1}}{p^{-15u_{n-1}^2/16+(-d_n/4+u_{n-2}/2+d_{n-1})u_{n-1}-d_{n-1}u_{n-2}/2}}
\end{eqnarray*}
As a function of $u_{n-1}$, the polynomial
\[
-15u_{n-1}^2/16+(-d_n/4+u_{n-2}/2+d_{n-1})u_{n-1}-d_{n-1}u_{n-2}/2
\]
is maximized at
\[
u_{n-1} = 8 (-d_n/4+u_{n-2}/2+d_{n-1})/15.
\]
By Lemma~\ref{dnasymptotics2}, Equation~\ref{dninequality1}, this is at most 1 when $n \ge 3$ and $d \ge 6$ or $n \ge 10$ and $d \ge 5$.  So as $u_{n-1}$ ranges from $1$ to $d_{n-1}$, the polynomial is maximized at $u_{n-1} = 1$.  By Lemma~\ref{polybound},
\begin{eqnarray*}
A_{n-2}(u_{n-2}) &\le& C(p^{15/16}) C(p) p^{d_n^2/4-15/16-d_n/4+d_{n-1}} p^{(1-d_{n-1})u_{n-2}/2}.
\end{eqnarray*}
This proves the theorem for the base case $i = n-2$.  By induction, for $i \le n-3$,
\begin{eqnarray*}
A_{i}(u_i) &=& \sum_{u_{i+1} = 0}^{d_{i+1}}{p^{-(u_{i+1}-d_{i+1})(u_{i+1}-u_i/2)} A_{i+1}(u_{i+1})} \\
&\le& C(p)^{n-i-2} p^{-15/16+d_n^2/4+d_{n-1}-d_n/4} \\
&& \qquad \cdot \sum_{u_{i+1} = 0}^{d_{i+1}}{p^{-(u_{i+1}-d_{i+1})(u_{i+1}-u_i/2)-u_{i+1}(d_{i+2}-1)/2}}.
\end{eqnarray*}
As a function of $u_{i+1}$, the polynomial
\begin{eqnarray*}
&& -(u_{i+1}-d_{i+1})(u_{i+1}-u_i/2)-u_{i+1}(d_{i+2}-1)/2) \\
&=& -u_{i+1}^2+(d_{i+1}+u_i/2-(d_{i+2}-1)/2)u_{i+1} - d_{i+1}u_i/i
\end{eqnarray*}
is maximized at
\[
u_{i+1} = (d_{i+1} + u_i/i - (d_{i+2}-1)/2))/2.
\]
By Lemma~\ref{dnasymptotics2}, Equation~\ref{dninequality2}, this is at most 1/2 when $n \ge 3$ and $d \ge 6$ or $n \ge 10$ and $d \ge 5$.  So as $u_{i+1}$ ranges from $0$ to $d_{i+1}$, the polynomial is maximized at $u_{i+1} = 0$.  Thus
\[
A_i(u_i) \le C(p)^{15/16} C(p)^{n-i-1} p^{-15/16+d_n^2/4+d_{n-1}-d_n/4} p^{-(d_{i+1}-1)u_i/i}
\]
and the result is proved by induction.
\end{proof}

\begin{lem} \label{quadbound}
Suppose that $\alpha_1, \dots, \alpha_r$ are positive integers with $n = \alpha_1 + \cdots + \alpha_r$.  Then
\begin{equation} \label{ub1}
\alpha_1^2 + \cdots + \alpha_r^2 \le (n-r+1)^2 + (r-1),
\end{equation}
and this bound is achieved when $\alpha_1 = \alpha_2 = \cdots = \alpha_{r-1} = 1$.  Furthermore, if $n \ge \ep + 1$ and $r \ge 2$, then
\begin{equation} \label{ub2}
\alpha_1^2 + \cdots + \alpha_r^2 + \ep r \le (n-1)^2 + 1 + 2\ep.
\end{equation}
\end{lem}

\begin{proof}
For Equation~\ref{ub1}, we use a simple induction argument.  It is clearly true for $r = 1$.  Suppose it is true up through $r$; we will prove it for $r+1$.
\begin{eqnarray*}
\alpha_1^2 + \cdots + \alpha_r^2 + \alpha_{r+1}^2 &\le&
(n-\alpha_{r+1}-r+1)^2 + (r-1) + \alpha_{r+1}^2 \\
&\le& (n-r+1-\alpha_{r+1})^2 + \alpha_{r+1}^2 + (r-1) \\
&\le& (n-r+1-1)^2 + 1^2 + (r-1) \\
&=& (n-r)^2 + r,
\end{eqnarray*}
proving Equation~\ref{ub1}.  As for Equation~\ref{ub2},
\begin{eqnarray*}
\alpha_1^2 + \cdots + \alpha_r^2 + \ep r &\le& (n-r+1)^2 + (r-1) + \ep r \\
&=& ((n-1)-(r-2))^2 + r - 1 + \ep r \\
&=& (n-1)^2 - 2(n-1)(r-2) + (r-2)^2 + r - 1 + \ep r \\
&\le& (n-1)^2 - (\ep + r - 1)(r-2) + (r-2)^2 + r - 1 + \ep r \\
&=& (n-1)^2 + 1 + 2 \ep,
\end{eqnarray*}
where the first inequality follows from Equation~\ref{ub1} and the second inequality follows from the fact that since $n \ge \ep + 1$ and $n \ge r$, we know that $n \ge (\ep + r + 1)/2$.
\end{proof}

\chapter{Numerical Estimates for Theorem~\ref{ubthm}}
\label{a_rwestimates}

This appendix contains numerical estimates used in the proof of Theorem~\ref{lowerboundthm}.  All of the following results assume that $p = 2^t-1$ is prime, $q = e^{2\pi i/p}$, and $n = rt$.  Let $G = C_p^d$ and define a probability distribution on $G$ by $P(e_k) = P(-e_k) = 1/2d$, where $e_k$ is the $k$-th unit vector and $k = 1, 2, \dots, d$.  Let $\{X_n\}$ be the Markov chain on $G$ given by $X_0 = 0$ and $X_{n+1} = 2X_n + g_n$, where the $g_n$ are independent random variables with distribution $P$, and let $P_n$ be the probability distribution on $G$ induced by $X_n$.  The function $f : C_p^d \to \C$ is defined by
\[
f(y) = \sum_{i=1}^d{\sum_{j=0}^{t-1}{q^{2^j y}}}.
\]
Finally, $\Pi_j$ and $\Gamma_j$ are defined by the product formulas
\[
\Pi_j = \prod_{a = 0}^{t-1}{\left[ \frac{d-1}{d} + \frac{1}{d} \cos{ \left( \frac{2\pi \cdot 2^a(2^j-1)}{p} \right) } \right]}
\]
and
\[
\Gamma_j = \prod_{a = 0}^{t-1}{\left[ \frac{d-2}{d} + \frac{1}{d} \cos{ \left( \frac{2\pi \cdot 2^a}{p} \right)} + \frac{1}{d} \cos{ \left( \frac{2\pi \cdot 2^a 2^j}{p} \right)} \right]}.
\]

\begin{lem}
\label{meanlem}
\begin{eqnarray*}
E_U(f) &=& 0 \\
E_U(f \overline{f}) &=& dt \\
\mathrm{Var}_U(f) &=& dt \\
E_{P_n}(f) &=& dt \Pi_1^r \\
E_{P_n}(f \overline{f}) &=& dt \sum_{j=0}^{t-1}{\Pi_j^r} + t(d^2-d) \sum_{j=0}^{t-1}{\Gamma_j^r} \\
\mathrm{Var}_{P_n}(f) &=& dt \sum_{j=0}^{t-1}{\Pi_j^r} + t(d^2-d) \sum_{j=0}^{t-1}{\Gamma_j^r} - d^2 t^2 \Pi_1^{2r}.
\end{eqnarray*}
\end{lem}

\begin{proof}
\begin{eqnarray*}
E_U(f) &=& \frac{1}{p^d} \sum_{y \in C_p^d}{\sum_{i=1}^d{\sum_{j=0}^{t-1}{q^{2^j y_i}}}} \\
&=& \frac{1}{p^d} \sum_{i=1}^d{p^{d-1} \sum_{y_i = 0}^{p}{\sum_{j=0}^{t-1}{q^{2^j y_i}}}} \\
&=& \frac{d}{p} \sum_{j = 0}^{t-1}{\sum_{y = 0}^p{q^{2^j y}}} \\
&=& 0. \\
E_U(f \overline{f}) &=& \frac{1}{p^d} \sum_{y \in C_p^d}{\sum_{i,i'=1}^d{\sum_{j,j'=0}^{t-1}{q^{2^j y_i} q^{-2^{j'} y_{i'}}}}} \\
&=& \frac{1}{p^d} \sum_{i,i'=1}^d{\sum_{y \in C_p^d}{\sum_{j,j'=0}^{t-1}{q^{2^j y_i} q^{-2^{j'} y_{i'}}}}} \\
&=& \frac{1}{p^d} \left[ d \sum_{y \in C_p^d}{\sum_{j,j'=0}^{t-1}{q^{2^j y_1} q^{-2^{j'} y_{1}}}} + (d^2-d) \sum_{y \in C_p^d}{\sum_{j,j'=0}^{t-1}{q^{2^j y_1} q^{-2^{j'} y_2}}} \right] \\
&=& \frac{1}{p^d} \left[ d p^{d-1} \sum_{y = 0}^{p-1}{\sum_{j,j'=0}^{t-1}{q^{(2^j-2^{j'})y}}} \right. \\
&& \hspace{1in} \left. + (d^2-d) p^{d-2} \sum_{y=0}^{p-1}{\sum_{z=0}^{p-1}{\sum_{j,j'=0}^{t-1}{q^{2^j y - 2^{j'} z}}}} \right] \\
&=& dt + \frac{d^2-d}{p^2} \left( \sum_{y=0}^{p-1}{\sum_{j=0}^{t-1}{q^{2^j y}}} \right)^2 \\
&=& dt \\
E_{P_n}(f) &=& \sum_{y \in C_p^d}{P_n(y) f(y)} \\
&=& \sum_{i = 1}^d{\sum_{j = 0}^{t-1}{\sum_{y \in C_p^d}{P_n(y) q^{2^j y_i}}}} \\
&=& \sum_{i = 1}^d{\sum_{j = 0}^{t-1}{\widehat{Q}_N(2^j e_i)}} \\
&=& \sum_{i=1}^d{\sum_{j=0}^{t-1}{\prod_{k=0}^{N-1}{\left[ \frac{d-1}{d} + \frac{1}{d} \cos{ \left( \frac{2\pi \cdot 2^k 2^j}{p} \right)} \right]}}} \\
&=& d \sum_{j=0}^{t-1}{\prod_{a=0}^{t-1}{\left[ \frac{d-1}{d} + \frac{1}{d} \cos{\left( \frac{2\pi \cdot 2^a 2^j}{p} \right)} \right]^r}} \\
&=& dt \Pi_1^r. \\
E_{P_n}(f \overline{f}) &=& \sum_{y \in C_p^d}{P_n(y) f(y) \overline{f}(y)} \\
&=& \sum_{i,i'=1}^{d}{\sum_{j,j'=0}^{t-1}{\sum_{y \in C_p^d}{P_n(y) q^{2^j y_i} q^{-2^{j'} y_{i'}}}}} \\
&=& \sum_{i,i'=1}^{d}{\sum_{j,j'=0}^{t-1}{\widehat{P}_n(2^j e_i - 2^{j'} e_{j'})}} \\
&=& d \sum_{j,j'=0}^{t-1}{\widehat{P}_n((2^j-2^{j'})e_1)} + (d^2-d) \sum_{j,j'=0}^{t-1}{\widehat{P}_n(2^j e_1 - 2^{j'} e_2)} \\
&=& dt \sum_{j=0}^{t-1}{\Pi_j^r} + (d^2-d) \cdot \sum_{j,j'=0}^{t-1}{\prod_{a=0}^{t-1}{\left[ \frac{d-2}{d} + \frac{1}{d} \cos{ \left( \frac{2\pi \cdot 2^a 2^j}{p} \right) } \right. }} \\
&& \hspace{2.3in} + \left. \frac{1}{d} \cos{ \left( \frac{2\pi \cdot 2^a 2^{j'}}{p} \right) } \right]^r \\
&=& dt \sum_{j=0}^{t-1}{\Pi_j^r} + t(d^2-d) \sum_{j=0}^{t-1}{\prod_{a=0}^{t-1}{ \left[ \frac{d-2}{d} + \frac{1}{d} \cos{ \left( \frac{2\pi \cdot 2^a}{p} \right) } \right. }} \\
&& \hspace{2.3in} + \left. \frac{1}{d} \cos{ \left( \frac{2\pi \cdot 2^a 2^j}{p} \right) } \right]^r \\
&=& dt \sum_{j=0}^{t-1}{\Pi_j^r} + t(d^2-d) \sum_{j=0}^{t-1}{\Gamma_j^r}.
\end{eqnarray*}
\end{proof}

\begin{lem}
\label{pi1lem}
$|\Pi_1|^d$ is bounded away from $0$ and $1$ independent of $d$ and $t$, in particular,
\[
e^{-3\pi^2} \le |\Pi_1|^d \le e^{-\pi^2/4}.
\]
Thus $dt \Pi_1^r - 2 d^{1/2} t^{1/2} \to 0$ as $t \to \infty$ when $\lambda \ge 1$.
\end{lem}

\begin{proof}
\begin{eqnarray*}
|\Pi_1|^d &=& \prod_{a = 0}^{t-1}{\left( \frac{d-1}{d} + \frac{1}{d} \cos{\left( \frac{2\pi \cdot 2^a}{p} \right)} \right)^d} \\
&\ge& \prod_{a = 0}^{t-1}{\left( \frac{d-1}{d} + \frac{1}{d} \left( 1 - \frac{1}{2} \left( \frac{2\pi \cdot 2^a}{p} \right)^2 \right) \right)^d} \\
&=& \prod_{a = 0}^{t-1}{\left( 1 - \frac{2\pi^2}{d} \cdot \left( \frac{2^a}{p} \right)^2 \right)^d} \\
&\ge& \prod_{a=0}^{t-1}{e^{-4\pi^2 \cdot 4^a / p^2}} \\
&=& \exp{\left( \frac{-4\pi^2}{p^2} \sum_{a=0}^{t-1}{4^a} \right)} \\
&=& \exp{\left( \frac{-4\pi^2}{p^2} \cdot \frac{4^t-1}{3} \right)} \\
&=& \exp{\left( \frac{-4\pi^2}{p^2} \cdot \frac{(p+1)^2-1}{3} \right)} \\
&\ge& \exp{\left( \frac{-4\pi^2}{3} \cdot \left( 1 + \frac{1}{p} \right)^2 \right)} \\
&\ge& e^{-3\pi^2}.
\end{eqnarray*}
\begin{eqnarray*}
|\Pi_1|^d &=& \prod_{a = 0}^{t-1}{\left( \frac{d-1}{d} + \frac{1}{d} \cos{\left( \frac{2\pi \cdot 2^a}{p} \right)} \right)^d} \\
&\le& \prod_{a = 0}^{t-1}{\left( \frac{d-1}{d} + \frac{1}{d} \left( 1 - \left( \frac{2\pi \cdot 2^a}{p} \right)^2 \right) \right)^d} \\
&=& \prod_{a = 0}^{t-1}{\left( 1 - \frac{\pi^2 \cdot 4^a}{dp^2} \right)^d} \\
&\le& \prod_{a = 0}^{t-1}{e^{-\pi^2 (2^a)^2/p^2}}\\
&=& \exp{\left( -\frac{\pi^2}{p^2} \sum_{a=0}^{t-1}{4^a} \right)} \\
&=& \exp{\left( -\frac{\pi^2\cdot (4^t-1)}{3p^2} \right)} \\
&=& \exp{\left( -\frac{\pi^2 \cdot ((p+1)^2-1)}{3p^2} \right)} \\
&\le& \exp{\left( -\frac{\pi^2}{3} \cdot (1 - 1/(p+1)^2) \right)} \\
&\le& e^{-\pi^2/4}.
\end{eqnarray*}
\end{proof}
For the following lemmas, let
\[
G(x,y) = \left| \frac{d-2}{d} + \frac{1}{d} \cos{2\pi x} + \frac{1}{d} \cos{2\pi y} \right|,
\]
and for convenience write $G(x) = G(x,0)$.

\begin{lem}
\label{ineqlem}
$|\Pi_j| \le |\Pi_1|$ and $|\Gamma_j| \le |\Pi_1|$ for all $j \ge 1$.
\end{lem}

\begin{proof}
This follows from Fact 1 in~\cite{cdg} in the case of $\Pi_j$ and is obvious in the case of $\Gamma_j$.
\end{proof}

\begin{lem}
\label{simprodlem}
There exists a constant $c_2$ independent of $d$ and $t$ so that for $k \ge 2$,
\[
\prod_{b = k+1}^{\ell}{G\left( 2^{-b} + \frac{2^{-b}}{p} \right)^{-1}} \le 1 + \frac{c_2}{d \cdot 4^k}.
\]
\end{lem}

\begin{proof}
\begin{eqnarray*}
&&\prod_{b = k+1}^{\ell}{G\left( 2^{-b} + \frac{2^{-b}}{p} \right)^{-1}} \\
&=& \prod_{b = k+1}^{\ell}{\left| \frac{d-1}{d} + \frac{1}{d} \cos{2\pi \cdot \left( 2^{-b} + \frac{2^{-b}}{p} \right)} \right|^{-1}} \\
&\le& \prod_{b = k+1}^{\infty}{\left| \frac{d-1}{d} + \frac{1}{d} \left( 1 - \frac{1}{2} \left( 2\pi \cdot \left( 2^{-b} + \frac{2^{-b}}{p} \right) \right)^2 \right) \right|^{-1}} \\
&=& \prod_{b = k+1}^{\infty}{\left( 1 - \frac{\pi^2}{2d} \left(1 + \frac{1}{p} \right)^2 4^{-b} \right)^{-1}} \\
&\le& \prod_{b = k+1}^{\infty}{\left( 1 + \frac{2\pi^2}{d} 4^{-b} \right)} \\
&\le& \exp{\left( \frac{2\pi^2}{d} \sum_{b=k+1}^{\infty}{4^{-b}} \right)} \\
&\le& \exp{\left( \frac{2\pi^2}{d} \cdot 4^{-k} \right)} \\
&\le& 1 + \frac{4\pi^2}{d \cdot 4^k}.
\end{eqnarray*}
\end{proof}

\begin{lem}
\label{constlem1}
There exists a constant $c_0$ independent of $d$ and $t$ so that for $t^{1/3} \le j \le t/2$,
\[
1 \le \left| \frac{\Pi_j}{\Pi_1^2} \right| \le 1 + \frac{c_0}{d \cdot 2^j}.
\]
\end{lem}

\begin{proof}
\begin{eqnarray}
\left| \frac{\Pi_j}{\Pi_1^2} \right| &=&
\prod_{a=0}^{t-1}{G\left(\frac{2^a(2^j-1)}{p} \right) G\left(\frac{2^a}{p} \right)^{-2}} \nonumber \\
&=& \prod_{b=1}^t{G\left(\frac{2^{t-b}(2^j-1)}{p} \right)G\left(\frac{2^{t-b}}{p} \right)^{-2}} \nonumber \\
&=& \prod_{b=1}^t{G\left(\frac{(p+1)2^{-b}(2^j-1)}{p}\right) G\left(\frac{(
p+1)2^{-b}}{p}\right)^{-2}}  \nonumber \\
&=& \prod_{b=1}^t{G\left(2^{j-b} - 2^{-b} + \frac{2^{j-b}}{p} - \frac{2^{-b}}{p} \right) G\left(2^{-b} + \frac{2^{-b}}{p} \right)^{-2}} \nonumber \\
&=& \prod_{b=1}^j{G\left(2^{-b} + \frac{2^{-b}}{p} - \frac{2^{j-b}}{p}\right) G\left(2^{-b} + \frac{2^{-b}}{p} \right)^{-1}} \nonumber \\
&& \cdot \prod_{b=j+1}^t{G\left(2^{-b} + \frac{2^{-b}}{p} - 2^{j-b} - \frac{2^{j-b}}{p}\right) G\left(2^{-b} + \frac{2^{-b}}{p} \right)^{-1}} \nonumber \\
&& \cdot \prod_{b=1}^{t-j}{G\left(2^{-b}+\frac{2^{-b}}{p}\right)^{-1}}
\cdot \prod_{b=t-j+1}^t{G\left(2^{-b}+\frac{2^{-b}}{p}\right)^{-1}} \nonumber \\
&=& \prod_{b=1}^j{G\left(2^{-b} + \frac{2^{-b}}{p} - \frac{2^{j-b}}{p}\right) G\left(2^{-b} + \frac{2^{-b}}{p} \right)^{-1}} \label{atleastone} \\
&& \cdot \prod_{b=1}^{t-j}{G\left(-2^{-j-b} - \frac{2^{-j-b}}{p} + 2^{-b} + \frac{2^{-b}}{p}\right)G\left(2^{-b}+\frac{2^{-b}}{p}\right)^{-1}}  \nonumber \\
&& \cdot \prod_{b=j+1}^t{G\left(2^{-b} + \frac{2^{-b}}{p} \right)^{-1}}
\prod_{b=t-j+1}^t{G\left(2^{-b}+\frac{2^{-b}}{p}\right)^{-1}}. \nonumber
\end{eqnarray}
Note that by Equation~\ref{atleastone}, it follows that $| \Pi_j/\Pi_1^2 | \ge 1$.  It follows from Lemma~\ref{simprodlem} that
\begin{eqnarray*}
\prod_{b=j+1}^t{G\left( 2^{-b} + \frac{2^{-b}}{p} \right)^{-1}} &\le& 1 + \frac{c_2}{d \cdot 4^j} \\
\prod_{b=t-j+1}^t{G\left( 2^{-b} + \frac{2^{-b}}{p} \right)^{-1}} &\le& 1 + \frac{c_2}{d \cdot 4^{t-j}}. \\
\end{eqnarray*}
Furthermore (using the fact that $G(x) \ge (d-2)/d$),
\begin{eqnarray*}
&& \prod_{b=1}^j{G\left(2^{-b}+\frac{2^{-b}}{p}-\frac{2^{j-b}}{p}\right) G\left(2^{-b}+\frac{2^{-b}}{p}\right)^{-1}} \\
&\le&
\prod_{b=1}^j{\frac{G\left(2^{-b} + \frac{2^{-b}}{p} \right) + \frac{1}{d} \left| \cos{2\pi \left(2^{-b} + \frac{2^{-b}}{p} - \frac{2^{j-b}}{p} \right)} - \cos{2\pi \left( 2^{-b} + \frac{2^{-b}}{p} \right)} \right|}{G\left(2^{-b} + \frac{2^{-b}}{p} \right)}} \\
&\le& \prod_{b = 1}^j{\left( 1 + \frac{2\pi}{d} \cdot \frac{2^{j-b}}{p} \cdot G\left(2^{-b}+\frac{2^{-b}}{p}\right)^{-1} \right)} \\
&\le& \prod_{b = 1}^j{\left( 1 + \frac{2\pi}{p(d-2)} 2^{j-b} \right)} \\
&\le& \exp{\left( \frac{2\pi}{p(d-2)} \sum_{b=1}^j{2^{j-b}} \right)} \\
&\le& \exp{\left( \frac{2\pi \cdot 2^j}{p(d-2)} \right)} \\
&\le& 1 + \frac{c_3}{d \cdot 2^j},
\end{eqnarray*}
and similarly
\begin{eqnarray*}
&& \prod_{b=1}^{t-j}{\frac{G\left( 2^{-b} + \frac{2^{-b}}{p} - 2^{-j-b} - \frac{2^{-j-b}}{p} \right)}{G\left( 2^{-b} + \frac{2^{-b}}{p} \right)}} \\
&\le& \prod_{b=1}^{t-j}{\left( 1 + \frac{2\pi}{d-2} \cdot \left(2^{-j-b} + \frac{2^{-j-b}}{p} \right) \right)} \\
&\le& \exp{\left( \frac{4\pi}{d-2} \cdot 2^{-j} \right)} \\
&\le& 1 + \frac{c_4}{d \cdot 2^j}.
\end{eqnarray*}
It follows that there is an absolute constant $c_0$ independent of $d$ and $t$ such that
\[
1 \le \left| \frac{\Pi_j}{\Pi_1^2} \right| \le 1 + \frac{c_0}{d \cdot 2^j}
\]
for $t^{1/3} \le j \le t/2$.
\end{proof}

\begin{lem}
\label{constlem2}
There exists a constant $c_1$ independent of $d$ and $t$ so that for $t^{1/3} \le j \le t/2$,
\[
1 - \frac{c_1}{d \cdot 2^j} \le \left| \frac{\Gamma_j}{\Pi_1^2} \right| \le 1 + \frac{c_1}{d \cdot 2^j}.
\]
\end{lem}

\begin{proof}
\begin{eqnarray*}
\left| \frac{\Gamma_j}{\Pi_1^2} \right|
&=& \prod_{a=0}^{t-1}{G\left(\frac{2^a}{p}, \frac{2^{a+j}}{p}\right)G\left(\frac{2^a}{p} \right)^{-2}} \\
&=& \prod_{b=1}^t{G\left(\frac{2^{t-b}}{p}, \frac{2^{t-b+j}}{p} \right) G\left(\frac{2^{t-b}}{p} \right)^{-2}} \\
&=& \prod_{b=1}^t{G\left(\frac{(p+1)2^{-b}}{p}, \frac{(p+1)2^{-b+j}}{p} \right) G\left(\frac{(p+1)2^{-b}}{p}\right)^{-2}} \\
&=& \prod_{b=1}^t{G\left(2^{-b} + \frac{2^{-b}}{p}, 2^{-b+j} +  \frac{2^{-b+j}}{p} \right) G\left(2^{-b} + \frac{2^{-b}}{p}\right)^{-2}} \\
&=& \prod_{b=1}^j{G\left(2^{-b} + \frac{2^{-b}}{p}, \frac{2^{-b+j}}{p} \right) G\left(2^{-b} + \frac{2^{-b}}{p}\right)^{-1}} \\
&& \cdot \prod_{b=j+1}^t{G\left(2^{-b} + \frac{2^{-b}}{p}, 2^{-b+j} + \frac{2^{-b+j}}{p} \right) G\left(2^{-b} + \frac{2^{-b}}{p} \right)^{-1}} \\
&& \cdot \prod_{b=1}^{t-j}{G\left(2^{-b} + \frac{2^{-b}}{p}\right)^{-1}}
\cdot \prod_{b=t-j+1}^t{G\left(2^{-b} + \frac{2^{-b}}{p}\right)^{-1}} \\
&=& \prod_{b=1}^j{G\left(2^{-b} + \frac{2^{-b}}{p}, \frac{2^{-b+j}}{p} \right) G\left(2^{-b} + \frac{2^{-b}}{p}\right)^{-1}} \\
&& \cdot \prod_{b=1}^{t-j}{G\left(2^{-b-j} + \frac{2^{-b-j}}{p}, 2^{-b} + \frac{2^{-b}}{p} \right) G\left(2^{-b} + \frac{2^{-b}}{p}\right)^{-1}} \\
&& \cdot \prod_{b=j+1}^t{G\left(2^{-b} + \frac{2^{-b}}{p}\right)^{-1}}
\prod_{b=t-j+1}^t{G\left(2^{-b} + \frac{2^{-b}}{p}\right)^{-1}}
\end{eqnarray*}
As before,
\begin{eqnarray*}
\prod_{b=j+1}^t{\frac{1}{G\left( 2^{-b} + \frac{2^{-b}}{p} \right)}} &\le& 1 + \frac{c_2}{d \cdot 4^j} \\
\prod_{b=t-j+1}^t{\frac{1}{G\left( 2^{-b} + \frac{2^{-b}}{p} \right)}} &\le& 1 + \frac{c_2}{d \cdot 4^{t-j}}. \\
\end{eqnarray*}
Furthermore (using the fact that $G(x) \ge (d-2)/d$),
\begin{eqnarray*}
&& \prod_{b=1}^j{G\left(2^{-b} + \frac{2^{-b}}{p}, \frac{2^{-b+j}}{p} \right) G\left(2^{-b} + \frac{2^{-b}}{p}\right)^{-1}} \\
&\le& \prod_{b=1}^j{\frac{G\left(2^{-b} + \frac{2^{-b}}{p}\right) + \frac{1}{d} \left| 1 - \cos{2\pi \cdot \left( \frac{2^{-b+j}}{p} \right)} \right|}{G\left(2^{-b} + \frac{2^{-b}}{p}\right)}} \\
&\le& \prod_{b=1}^j{\left( 1 + \frac{2\pi}{d-2} \cdot \frac{2^{-b+j}}{p} \right)} \\
&\le& \exp{\left( \frac{2^{j+1}\pi}{(d-2)p} \sum_{b=1}^j{2^{-b}} \right)} \\
&\le& 1 + \frac{2^{j+2} \pi}{(d-2)p},
\end{eqnarray*}
and similarly,
\begin{eqnarray*}
&&\prod_{b=1}^{t-j}{G\left(2^{-b-j} + \frac{2^{-b-j}}{p}, 2^{-b} + \frac{2^{-b}}{p} \right) G\left(2^{-b} + \frac{2^{-b}}{p}\right)^{-1}} \\
&\le& \prod_{b=1}^{t-j}{\left( 1 + \frac{2\pi}{d-2} \cdot \left(2^{-b-j} + \frac{2^{-b-j}}{p} \right) \right)} \\
&\le& \exp{\left( \frac{2^{2-j}\pi}{d-2} \sum_{b=1}^{t-j}{2^{-b}} \right)} \\
&\le& 1 + \frac{2^{3-j}\pi}{d-2}.
\end{eqnarray*}
For the lower bound, we have
\begin{eqnarray*}
&& \prod_{b=1}^j{G\left(2^{-b} + \frac{2^{-b}}{p}, \frac{2^{-b+j}}{p} \right) G\left(2^{-b} + \frac{2^{-b}}{p}\right)^{-1}} \\
&\ge& \prod_{b=1}^j{\frac{G\left(2^{-b} + \frac{2^{-b}}{p}\right) - \frac{1}{d} \left| 1 - \cos{2\pi \cdot \left( \frac{2^{-b+j}}{p} \right)} \right|}{G\left(2^{-b} + \frac{2^{-b}}{p}\right)}} \\
&\ge& \prod_{b=1}^j{\left( 1 - \frac{2\pi}{d} \cdot \frac{2^{-b+j}}{p} \right)} \\
&\ge& \exp{\left( -\frac{2^j\pi}{dp} \sum_{b=1}^j{2^{-b}} \right)} \\
&\ge& 1 - \frac{2^{j-1} \pi}{dp},
\end{eqnarray*}
and similarly,
\begin{eqnarray*}
&&\prod_{b=1}^{t-j}{G\left(2^{-b-j} + \frac{2^{-b-j}}{p}, 2^{-b} + \frac{2^{-b}}{p} \right) G\left(2^{-b} + \frac{2^{-b}}{p}\right)^{-1}} \\
&\ge& \prod_{b=1}^{t-j}{\left( 1 - \frac{2\pi}{d} \cdot \left(2^{-b-j} + \frac{2^{-b-j}}{p} \right) \right)} \\
&\ge& \exp{\left( -\frac{2^{1-j}\pi}{d} \sum_{b=1}^{t-j}{2^{-b}} \right)} \\
&\ge& 1 - \frac{2^{-j}\pi}{d}.
\end{eqnarray*}
It follows that there is an absolute constant independent of $d$ and $t$ such that
\[
1 - \frac{c_1}{d \cdot 2^j} \le \left| \frac{\Gamma_j}{\Pi_1^2} \right| \le 1 + \frac{c_1}{d \cdot 2^j}
\]
for $t^{1/3} \le j \le t/2$.
\end{proof}

\begin{lem}
\label{symlem}
$\Pi_j = \Pi_{t-j}$ and $\Gamma_j = \Gamma_{t-j}$.
\end{lem}

\begin{proof}
\begin{eqnarray*}
\Pi_{t-j} &=& \prod_{a = 0}^{t-1}{G\left( \frac{2^a(2^{t-j}-1)}{p} \right)} \\
&=& \prod_{a = 0}^{t-1}{G\left( \frac{2^{a+j}((p+1)2^{-j}-1)}{p} \right)} \\
&=& \prod_{a = 0}^{t-1}{G\left( \frac{2^a((p+1)-2^j)}{p} \right)} \\
&=& \prod_{a = 0}^{t-1}{G\left( \frac{2^a(2^j-1)}{p} \right)} \\
&=& \Pi_j \\
\Gamma_{t-j} &=& \prod_{a = 0}^{t-1}{G\left( \frac{2^a}{p}, \frac{2^{a+t-j}}{p} \right)} \\
&=& \prod_{a = 0}^{t-1}{G\left( \frac{2^{a+j}}{p}, \frac{2^{a+t}}{p} \right)} \\
&=& \prod_{a = 0}^{t-1}{G\left( \frac{2^{a+j}}{p}, \frac{2^a}{p} \right)} \\
&=& \Gamma_j.
\end{eqnarray*}
\end{proof}

\begin{lem}
\label{biglem}
\[
\frac{1}{dt} \sum_{j=0}^{t-1}{\left( \frac{\Pi_j^r}{\Pi_1^2} \right)^r} + \frac{1}{t} \left( 1 - \frac{1}{d} \right) \sum_{j=0}^{t-1}{\left( \frac{\Gamma_j}{\Pi_1^2} \right)^r} \to 1,
\]
as $t \to \infty$.
\end{lem}

\begin{proof}
Note that
\begin{eqnarray*}
\Pi_j &=& \prod_{a = 0}^{t-1}{\left[ \frac{d-1}{d} + \frac{1}{d} \cos{ \left( \frac{2\pi \cdot 2^a(2^j-1)}{p} \right) } \right]} \\
&=& \prod_{a = 0}^{t-1}{G\left( \frac{2^a(2^j-1)}{p} \right)} \\
\Gamma_j &=& \prod_{a = 0}^{t-1}{\left[ \frac{d-2}{d} + \frac{1}{d} \cos{ \left( \frac{2\pi \cdot 2^a}{p} \right)} + \frac{1}{d} \cos{ \left( \frac{2\pi \cdot 2^{a+j}}{p} \right)} \right]} \\
&=& \prod_{a = 0}^{t-1}{G\left( \frac{2^a}{p}, \frac{2^{a+j}}{p} \right)}.
\end{eqnarray*}
By Lemmas~\ref{constlem1} and~\ref{constlem2},
\begin{eqnarray*}
\sum_{t^{1/3} \le j \le t/2}{\left| \left( \frac{\Pi_j}{\Pi_1^2} \right)^r - 1 \right|} &\le& \frac{c_5 t r}{d \cdot 2^{t^{1/3}}} < \frac{c_6 \ln{d}}{2^{t^{1/4}}} \\
\sum_{t^{1/3} \le j \le t/2}{\left| \left( \frac{\Gamma_j}{\Pi_1^2} \right)^r - 1 \right|} &\le& \frac{c_5 t r}{d \cdot 2^{t^{1/3}}}  < \frac{c_6 \ln{d}}{2^{t^{1/4}}}.
\end{eqnarray*}
Then, using Lemma~\ref{symlem} and Lemma~\ref{ineqlem}, it follows that
\begin{eqnarray*}
\frac{1}{t} \sum_{j=0}^{t-1}{\left( \frac{\Pi_j}{\Pi_1^2} \right)^r}
&\le& \frac{2}{t} \left( \sum_{0 \le j < t^{1/3}}{\left( \frac{\Pi_j}{\Pi_1^2} \right)^r} + \sum_{t^{1/3} \le j \le t/2}{\left( \frac{\Pi_j}{\Pi_1^2} \right)^r} \right) \\
&\le& \frac{2}{t} \left( \sum_{0 \le j < t^{1/3}}{\Pi_1^{-r}} + \sum_{t^{1/3} \le j \le t/2}{\left( \frac{\Pi_j}{\Pi_1^2} \right)^r} \right) \\
&\le& \frac{2}{t} \left( \sum_{0 \le j < t^{1/3}}{2 d^{1/2} t^{1/2}} + \sum_{t^{1/3} \le j \le t/2}{\left( \frac{\Pi_j}{\Pi_1^2} \right)^r} \right) \\
&=& 1 + o(1).
\end{eqnarray*}
Similarly,
\begin{eqnarray*}
\frac{1}{t} \sum_{j=0}^{t-1}{\left( \frac{\Pi_j}{\Pi_1^2} \right)^r} = 1 + o(1).
\end{eqnarray*}
This proves the lemma.
\end{proof}

\bibliographystyle{plain}
\def\cprime{$'$}

\end{document}